\documentclass[10pt,twoside]{article}

\usepackage{amsmath}
\usepackage{amsfonts}
\usepackage{amssymb}
\usepackage{amscd}
\usepackage{amsthm}
\usepackage{amsbsy}
\usepackage{graphicx}
\usepackage{bm}

\def\@oddhead{\hfill \shorttitle \hfill \thepage}
\def\@evenhead{\thepage \hfill \shortauthor \hfill}
\def\@oddfoot{}
\def\@evenfoot{}

\newtheorem{Thm}{Theorem}
\newtheorem{Cor}[Thm]{Corollary}

\newtheorem{Lem}[Thm]{Lemma}
\newtheorem{Con}[Thm]{Conjecture}

\newtheorem{Rem}[Thm]{Remark}
\newtheorem{Pro}[Thm]{Proposition}

\newcommand{\Energy}{{\text{E}}}
\newcommand{\ddd}{d}  

\newcommand{\Vol}{{\text{Vol}}}

\newcommand{\nn}{{\bf{n}}}
\newcommand{\Ric}{{\text{Ric}}}

\newcommand{\dist}{{\text {dist}}}

\newcommand{\Hess}{{\text {Hess}}}
\def\ZZ{{\bold Z}}
\def\RR{{\bold R}}

\def\SS{{\bold S}}

\def\CC{{\bold C }}
\newcommand{\dv}{{\text {div}}}
\newcommand{\e}{{\text {e}}}

\newcommand{\Area}{{\text {Area}}}
\newcommand{\Length}{{\text {Length}}}

\newcommand{\cT}{{\mathcal{T}}}

\newcommand{\cI}{{\mathcal{I}}}
 
\newcommand{\cSt}{{\mathcal{S}_{neck}}}
 
\newcommand{\cSu}{{\mathcal{S}_{ulsc}}}

\newcommand{\cB}{{\mathcal{B}}}

\newcommand{\cF}{{\mathcal{F}}}
\newcommand{\cL}{{\mathcal{L}}}

\newcommand{\cM}{{\mathcal{M}}}

\newcommand{\cP}{{\mathcal{P}}}
\newcommand{\cS}{{\mathcal{S}}}

\newcommand{\eqr}[1]{(\ref{#1})}

\date{}
\title{\ \\[0.4cm] \ \\ \bf  Minimal surfaces and mean curvature flow}
\author{Tobias H. Colding\footnote{MIT, Dept. of Math.
77 Massachusetts Avenue, Cambridge, MA 02139-4307.}\hspace{2mm}
and William P. Minicozzi II\footnote{Johns Hopkins University,
Dept. of Math.
3400 N. Charles St.
Baltimore, MD 21218}}
\begin{document}
%-------------------

\maketitle

%------------------------

\thispagestyle{empty}

%--------------------------------------
\begin{abstract}
\vskip 3mm\footnotesize{

\vskip 4.5mm

We discuss recent results on minimal surfaces and mean curvature flow, focusing on the classification and structure of embedded minimal surfaces and the stable singularities of mean curvature flow.   
This article is dedicated to Rick Schoen.

\vspace*{2mm}
{\bf 2000 Mathematics Subject Classification:}

\vspace*{2mm}
{\bf Keywords and Phrases: }}

\end{abstract}
\vspace*{-8.2cm}
%%%%%%%%%%%%%%%%%%%%%%%%%%%%%%%%%%%%%
\vspace*{2mm} \hspace*{82mm}
\begin{picture}(41,10)(0,0)\thicklines\setlength{\unitlength}{1mm}
\put(0,2){\line(1,0){41}} \put(0,16){\line(1,0){41}}%
\put(0,12.){\sl \copyright\hspace{1mm}Higher Education Press}
\put(0,7.8){\sl \hspace*{4.8mm}and International Press}%
\put(0,3.6){{\sl \hspace*{4.8mm}Beijing-Boston} }
\end{picture}

\vspace*{-17.6mm}{{\sl The title of\\
This book*****}\\ALM\,?, pp.\,1--?} 
%%%%%%%%%%%%%%%%%%
\vspace*{7.6cm}
% ----------------------------------------------------------------
\section{Introduction}

The main focus of this survey is on minimal surfaces and
mean curvature flow, but to put these topics in perspective we begin with more elementary analysis of the
energy of curves and functions.  This leads us to first variation formulas for energy and critical points for those.
The critical points are   of course geodesics and harmonic functions, respectively.  We continue by considering the gradient, or rather the negative gradient, flow for energy which leads us to the curve shortening flow and the heat equation.   
Having touched upon these more elementary topics, we move on to one of our main topics which is minimal surfaces.  We discuss first and second variations for area and volume and the gradient (or rather negative) gradient flow for area and volume which is the mean curvature flow.   Beginning as elementary as we do allows us later in the survey to draw parallels from the more advanced topics to the simpler ones.

The other topics that we cover are the Birkhoff min-max argument that produces closed geodesics and its higher dimensional analog that gives existence of closed immersed minimal surfaces.  We discuss stable and unstable critical points and index of critical points and eventually discuss the very recent classification of all stable self-similar shrinkers for the mean curvature flow.  For minimal surfaces, stability and Liouville type theorems have played a major role in later developments and we touch upon the Bernstein theorem that is the minimal surface analog of the Liouville theorem for harmonic functions and the curvature estimate that is the analog of the gradient estimate.  We discuss various monotone quantities under curve shortening and mean curvature flow like Huisken's volume, the width, and isoperimetric ratios of Gage and Hamilton.  We explain why Huisken's monotonicity leads to that blow ups of the flow at singular points in space time can be modeled by self-similar flows and explain why the classification of stable self-similar flows is expected to play a key role in understanding of generic mean curvature flow where the flow begins at a hypersurface in generic position.  One of the other main topics of this survey is that of embedded minimal surfaces where we discuss some of the classical examples going back to Euler and Monge's student Meusnier in the 18th century and the recent examples of Hoffman-Weber-Wolf and various examples that date in between.  The final main results that we discuss are the recent classification of embedded minimal surfaces and some of the uniqueness results that are now known.

\vskip2mm
It is a great pleasure 
for us to dedicate this article to Rick Schoen.

\section{Harmonic functions and the heat equation}

We begin with a quick review of the energy functional on functions, where the critical points are called harmonic functions and the gradient flow is the heat equation.  This will give some context for
the main topics of this survey, minimal surfaces and mean curvature flow, that are critical points and gradient flows, respectively, for the area functional.

\subsection{Harmonic functions}

Given a differentiable function
$
u:\RR^n\to \RR\notag \, ,
$
the energy is defined to be
\begin{equation}	\label{e:energy}
\Energy (u)=\frac{1}{2}\int |\nabla u|^2
=\frac{1}{2}\int \left(\left|\frac{\partial u}{\partial x_1}\right|^2+\cdots+\left|\frac{\partial u}{\partial x_n}\right|^2\right) \, . 
\end{equation}
This gives a functional defined on the space of functions.  We can construct a curve in the space of functions by taking a smooth function $\phi$ with compact support and considering the one-parameter family of functions
$
u + t \, \phi \, .
$
Restricting the energy functional to this curve gives
\begin{align}
\Energy (u+t\phi)=\frac{1}{2}\int|\nabla (u+t\,\phi)|^2
=\frac{1}{2}\int |\nabla u|^2+t\int \langle u,\nabla \phi\rangle+\frac{t^2}{2}\int |\nabla \phi|^2\, . 
\end{align}
Differentiating at $t=0$, we get that the directional derivative of the energy functional (in the direction $\phi$) is 
\begin{align}
\frac{d}{dt}_{t=0} \Energy (u+t\,\phi) = \int \langle \nabla u,\nabla \phi \rangle=-\int \phi \,\Delta u\, , 
\end{align}
where the last equality used the divergence theorem and the fact that $\phi$ has compact support.
We conclude that:

\begin{Lem}
The directional derivative
$\frac{d}{dt}_{t=0}\Energy =0 $ for all  $\phi$ if and only if
\begin{equation}	\label{e:harmonic}
\Delta u=\frac{\partial^2u}{\partial x_1^2}+\cdots+\frac{\partial^2u}{\partial x_n^2}=0\, . 
\end{equation}
\end{Lem}

Thus, we see that the critical points for energy are the functions $u$ with
  $\Delta u = 0$;  these are called  harmonic functions.   In fact, something stronger is true.
Namely, 
harmonic functions are not just critical points for the energy functional, 
but are actually minimizers.

\begin{Lem}
If $\Delta u = 0$ on 
  a bounded domain $\Omega$ and
$\phi$ vanishes on $\partial \Omega$, then 
\begin{align}
	\int_{\Omega} \left| \nabla (u+ \phi) \right|^2  = \int_{\Omega} \left| \nabla u \right|^2  
	 + \int_{\Omega} \left| \nabla \phi \right|^2 
	 \, . \notag 
\end{align}
\end{Lem}

\begin{proof}
Since $\Delta u = 0$ and $\phi$ vanishes on $\partial \Omega$, the divergence theorem gives
\begin{equation}
	0 = \int_{\Omega} \dv \, \left( \phi \nabla  u \right) = \int_{\Omega} \langle \nabla \phi , \nabla u \rangle \, .
\end{equation}
So we conclude that
\begin{align}
	\int_{\Omega} \left| \nabla (u+ \phi) \right|^2 &= \int_{\Omega} \left| \nabla u \right|^2 + 2 \, 
	 \langle \nabla \phi , \nabla u \rangle 
	 + \left| \nabla \phi \right|^2  \notag \\
	 & = \int_{\Omega} \left| \nabla u \right|^2  
	 + \int_{\Omega} \left| \nabla \phi \right|^2 
	 \, . \notag 
\end{align}
\end{proof}

\subsection{The heat equation}

The heat equation is the 
(negative) gradient flow (or steepest descent) for the energy functional.  This means that we evolve a function $u(x,t)$ over time in the direction of its Laplacian $\Delta u$, giving the linear parabolic heat equation
\begin{equation}
\frac{\partial u}{\partial t}   =\Delta u\, . 
\end{equation}

Given any finite energy solution $u$ of the heat equation that decays fast enough to justify integrating by parts, the energy is non-increasing along the flow.  In fact, we have
\begin{align}
\frac{d}{dt} E(u) =-\int \frac{\partial u}{\partial t}   \,\Delta u = -\int (\Delta u)^2 \, .\notag
\end{align}
 
 Obviously, harmonic functions are fixed points, or {\emph{static solutions}}, of the flow.  

 \subsection{Negative gradient flows near a critical point}

We are interested in the dynamical properties of the heat equation near a harmonic function.  Before getting to this, it is useful to recall the simple finite dimensional case.  Suppose therefore that $f: \RR^2 \to \RR$ is a smooth function with a non-degenerate critical point at $0$ (so $\nabla f (0) = 0$ but the Hessian of $f$ at $0$ has rank $2$).  The behavior of the negative gradient flow
$$
	(x' , y') = - \nabla f (x,y) 
$$
is determined by the Hessian of $f$ at $0$.

The behavior  depends on the index of the critical point, as is illustrated by the following examples:
\begin{enumerate}
\item[(Index 0):]
The function $f(x,y) = x^2 + y^2$ has a minimum at $0$.  The vector field is
$(-2x, -2y)$ and the flow lines are rays into the origin.  Thus every flow line limits to $0$.
\item[(Index 1):] The function $f(x,y) = x^2 - y^2$ has an index one critical point at $0$.  The vector field is
  $(-2x, 2y)$ and the flow lines are 
level sets of the function $h(x,y) = xy$.  
Only points where $y=0$ are on flow lines that limit to the origin. 
\item[(Index 2):]
 The function $f(x,y) = - x^2 - y^2$ has a maximum at $0$.  The vector field is
$(2x, 2y)$ and the flow lines are rays out of the origin.  Thus every flow line limits to $\infty$ and it is impossible to reach $0$.
\end{enumerate}

Thus, we see that the critical point $0$ is ``generic'', or dynamically stable, if and only if it has index $0$.  When the index is positive, the critical point is not generic and a ``random'' flow line will miss the critical point.

\vskip8mm
\includegraphics[totalheight=.35\textheight, width=.9\textwidth]{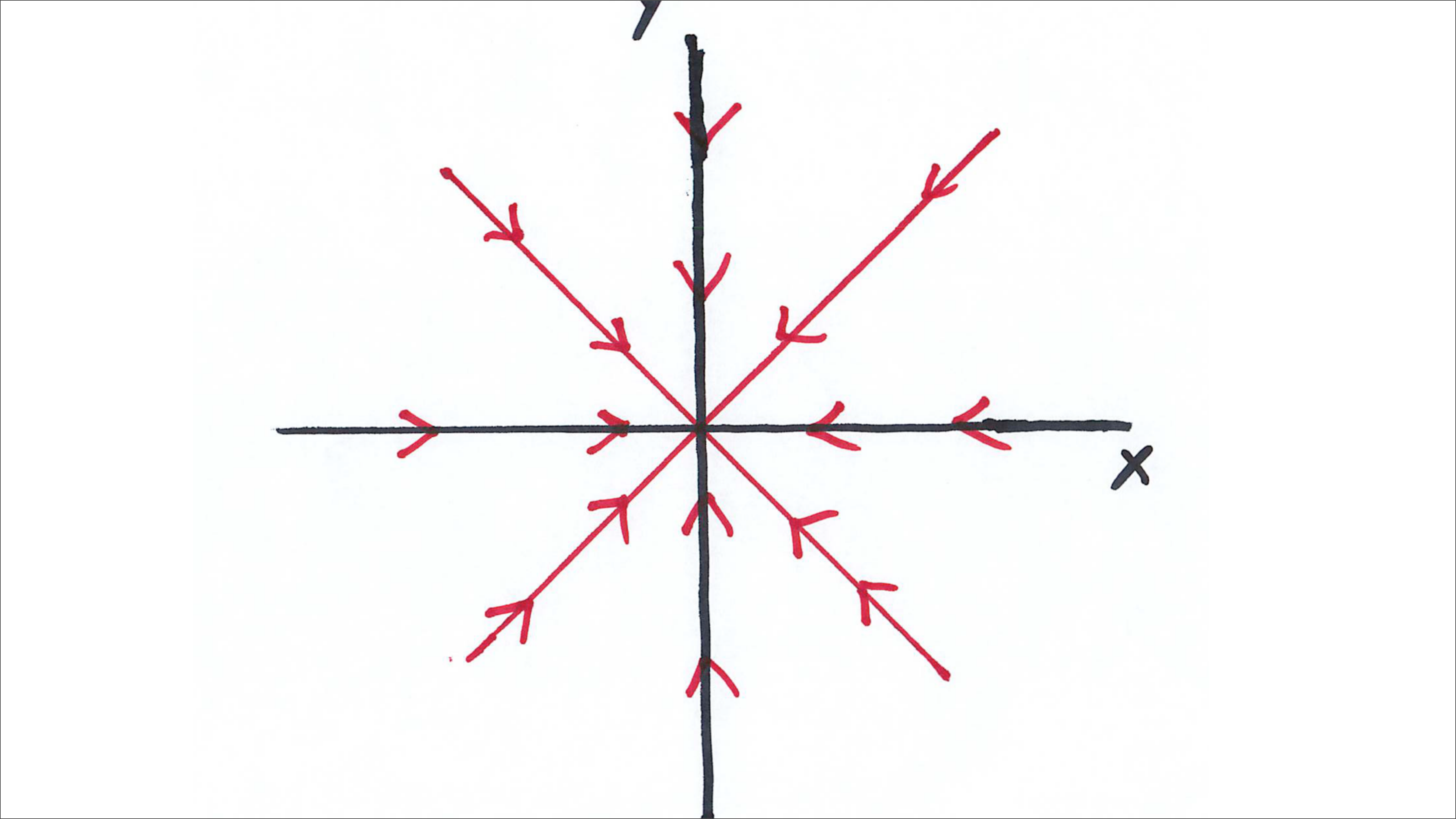}
\vskip1mm
\noindent
$f(x,y) = x^2 + y^2$ has a minimum at $0$.
Flow lines: Rays through the origin.

\vskip8mm
\includegraphics[totalheight=.35\textheight, width=.9\textwidth]{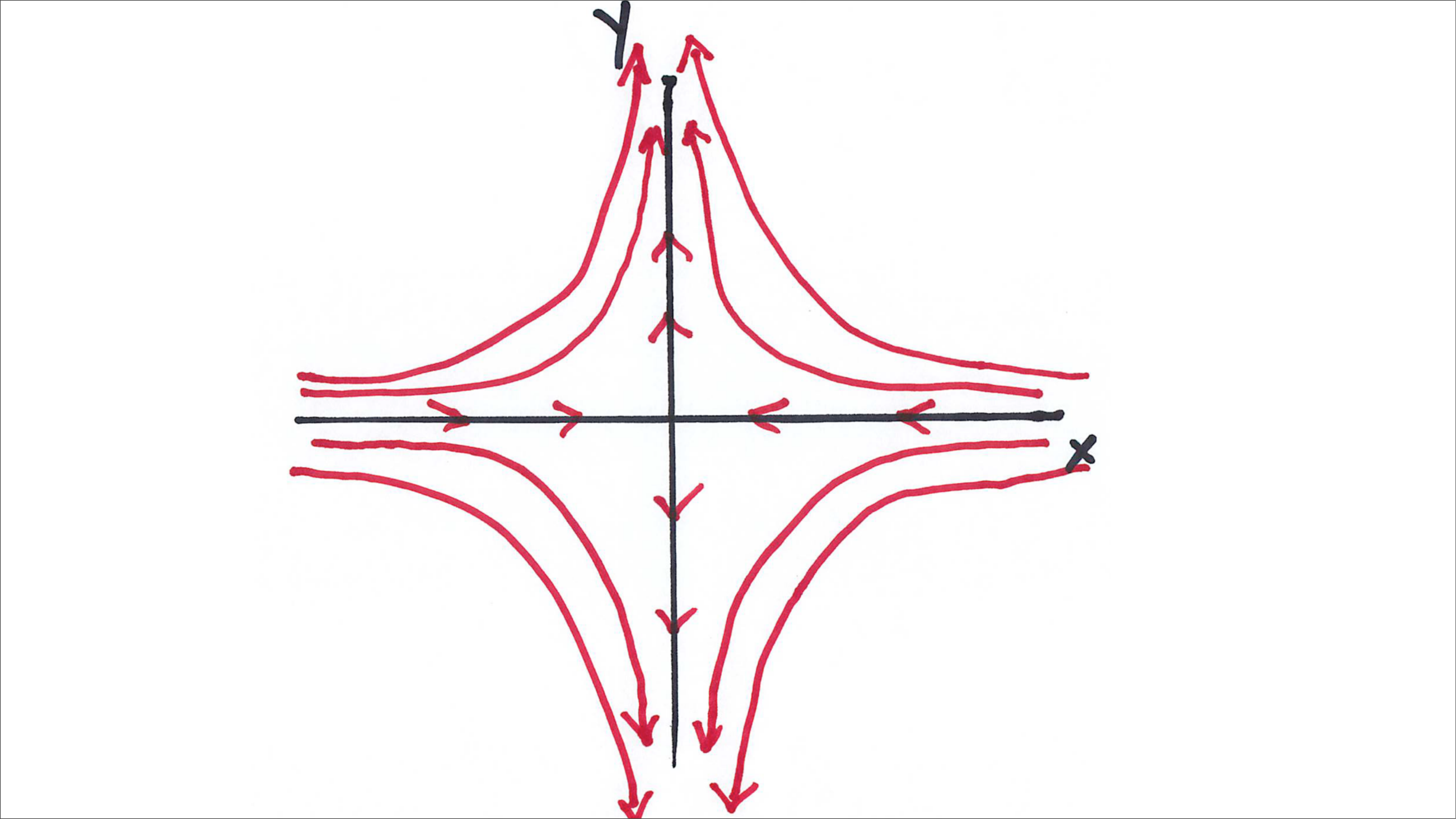}
\vskip1mm
\noindent
$f(x,y) = x^2 - y^2$ has an index one critical point at $0$.
Flow lines:  Level sets of   $ xy$.

\noindent
Only points where $y=0$ limit to the origin.

\subsection{Heat flow near a harmonic function}

To analyze the dynamical properties of the heat flow, 
suppose first that  $u$ satisfies the heat equation on a bounded domain $\Omega$ with $u=0$ on $\partial \Omega$.  By the maximum principle, a harmonic function that vanishes on $\partial \Omega$ is identically zero.  Thus, we expect that $u$ limits towards $0$.  We show this next.
 
Since $u=0$ on $\partial \Omega$, applying the divergence theorem to $u \nabla u$ gives
\begin{equation}
	\int_{\Omega} |\nabla u |^2 = - \int_{\Omega} u \, \Delta u \, . \notag 
\end{equation}
Applying Cauchy-Schwarz and then the Dirichlet Poincar\'e inequality gives
\begin{equation}
	\left( \int_{\Omega} |\nabla u |^2 \right)^2 \leq
	  \int_{\Omega} u^2 \, \int_{\Omega} \left( \Delta u  \right)^2 
	  \leq  C\,  \int_{\Omega} \left| \nabla  u \right|^2 \, \int_{\Omega} \left( \Delta u  \right)^2 
	  \, . \notag 
\end{equation}
Finally, dividing both sides by $\int_{\Omega} \left| \nabla  u \right|^2$ gives
\begin{equation}
	2 \, E(t) \equiv 2\, \Energy (u(\cdot , t)) = \int_{\Omega} |\nabla u |^2   
	  \leq  C\,    \int_{\Omega} \left( \Delta u  \right)^2 = - C \, E'(t)
	  \, , \notag
\end{equation}
where $C= C(\Omega)$ is the constant from the Dirichlet Poincar\'e inequality.  Integrating this 
gives  that $E(t)$ decays exponentially with
$$
E(t) \leq E(0) \, \e^{- \frac{2}{C} \, t } \, .
$$
Finally, another application of the Poincar\'e inequality shows that $\int u^2 (x,t)$ also decays exponentially in $t$, as we expected.

\vskip2mm
If $u$ does not vanish on $\partial \Omega$, there is a unique harmonic function $w$ with
\begin{equation}
	u \big|_{\partial \Omega} = w \big|_{\partial \Omega} \, .  \notag
\end{equation}
It follows that $(u-w)$ also solves the heat equation and is zero on $\partial \Omega$.
By the previous argument,   $(u-w)$ decays exponentially  and we conclude that all harmonic functions are attracting critical points of the flow.
Since we have already shown that harmonic functions are minimizers for the energy functional, and thus index zero critical points, this is 
exactly what the finite dimensional toy model suggests.

\section{Energy of a curve}

Geodesics in a Riemannian manifold $M^n$ arise variationally in two ways.  They are critical points of the energy functional restricted to  maps into $M$ (generalizing harmonic functions) and they are also critical points of the length functional.  We will first analyze the energy functional on curves.

\subsection{Critical points for energy are geodesics}

Suppose $\gamma$ is a closed curve in a Riemannian manifold $M^n$, i.e., 
\begin{equation}
\gamma :\SS^1\to M \,  , \notag
\end{equation}
where the circle $\SS^1$ is identified with $\RR / 2\pi \, \ZZ$.
The energy of $\gamma$ is 
\begin{equation}
\Energy (\gamma)=\frac{1}{2}\int_{\SS^1}|\gamma'|^2\, .\notag
\end{equation}
 A variation of $\gamma$ is a curve  in the space of curves that goes through $\gamma$.  We can specify this   by a map
$$F:\SS^1\times [-\epsilon,\epsilon]\to M^n $$
with $F(\cdot,0)=\gamma$.   The variation vector field $V$ is the tangent vector to this path given by
 $V=\frac{\partial F}{\partial t}$.  An easy calculation shows that
\begin{equation}
\frac{d}{dt} \big|_{t=0}\Energy (\gamma (\cdot , t))=\int \langle \gamma',F_{s,t}\rangle=-\int \langle \gamma'', V \rangle\, ,\notag
\end{equation}
where $\gamma'' = \nabla_{\gamma'} \gamma'$.
We conclude that
\begin{align}
\frac{d}{dt} \big|_{t=0}E=0 \text{  for all  }V  \notag
\end{align}
if and only if $\gamma'' = 0$.  Such a curve is called a geodesic.

\subsection{Second variation of energy of a curve in a surface}

We have seen that a closed geodesic 
 $\gamma: \SS^1\to M^2$ in a surface $M$   is a critical point for energy.   The hessian of the energy functional is given by the second variation formula.  For simplicity, we assume that
 $|\gamma'|=1$ and $V= \phi \, \nn$ is a normal variation  where $\nn$ is the unit normal to $\gamma$, so $V'=\phi' \,\nn$.   We compute
\begin{align}
\frac{d^2}{dt^2}_{t=0}E(t)&=\int \left(|F_{s,t}|^2-\langle \gamma',F_{s,tt}\rangle \right)
=\int \left(|F_{t,s}|^2-\langle \gamma',F_{s,tt}\rangle \right)\notag\\
&=\int \left(|V'|^2-\langle \gamma',F_{s,tt}\rangle \right)=\int \left( |\phi'|^2-K\,\phi^2\right)\notag\\
&=-\int \left(\phi''\,\phi+K\,\phi^2\right)\, ,\notag
\end{align}
where $K$ is the curvature of $M$.

In this calculation, we used that $F_{ss}=0$ since $\gamma$ is a geodesic and that the curvature $K$ comes in when one changes the order of derivatives, i.e.,
\begin{equation}
\langle F_s,F_{s,tt}\rangle =\langle F_s  , F_{tt,s} \rangle +K\, [|F_s|^2\,|F_t|^2-\langle F_s,F_t\rangle^2]\, .\notag
\end{equation}
Using that $F_s = V=\phi\,\nn$ is perpendicular to $F_t = \gamma'$ and $|\gamma'| = 1$ gives
\begin{equation}
\langle F_s,F_{s,tt}\rangle =\langle F_s , F_{tt,s} \rangle +K\, \phi^2\, .\notag
\end{equation}

A geodesic $\gamma_0$  is stable if the Hessian of the energy functional at $\gamma_0$ has index zero, i.e., if
\begin{equation}
\frac{d^2}{dt^2} \big|_{t=0} \, E(t)\geq 0\, ,  \notag
\end{equation}
for all variations of $\gamma_0$.  Roughly speaking, stable geodesics minimize energy compared to nearby curves.

\subsection{Geodesics in a free homotopy class}

The simplest way to produce geodesics is to look for minima of the energy functional.  To get a closed geodesic, the minimization is done in a free homotopy class.  A free homotopy class of a closed curve $c:\SS^1\to M$ on a manifold $M$ consists of all the curves that are homotopic to $c$.   Namely, a curve $\gamma$ is freely homotopic to $c$ if there exists a one parameter family $$F:\SS^1\times [0,1]\to M$$ so $F(\cdot,0)=c$ and $F(\cdot, 1)=\gamma$.  
The difference between a homotopy class and a free homotopy class is that
there is no fixed base point for a free homotopy class.

Standard arguments in Riemannian geometry then give:

\begin{Lem}
In each free homotopy class on a closed manifold, there is at least one curve that realizes the smallest energy.   This minimizing curve is a geodesic and is non-trivial if the homotopy class is non-trivial.
\end{Lem}

\begin{figure}[htbp]
\centering\includegraphics[totalheight=.35\textheight, width=.9\textwidth]{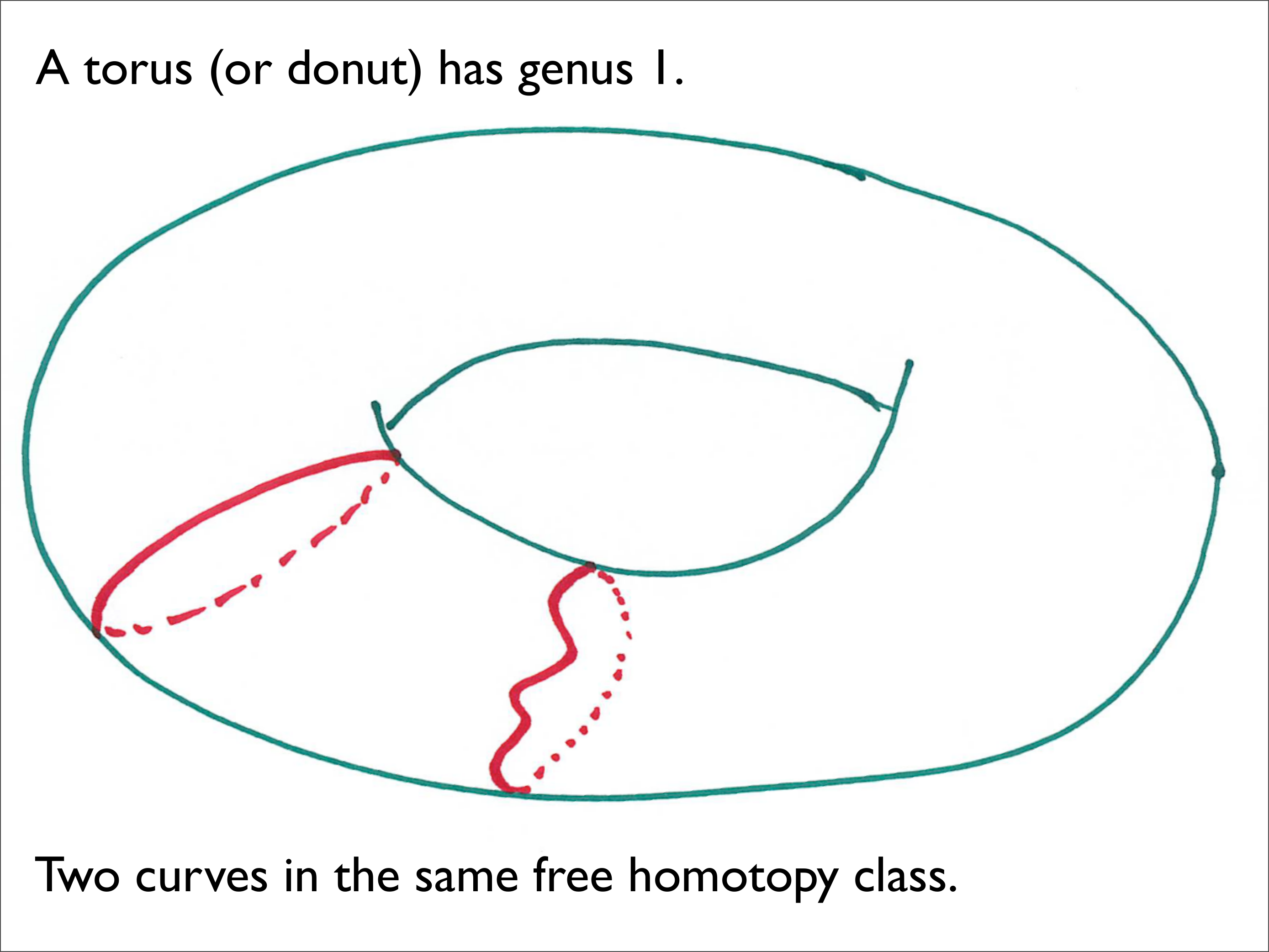}
\caption{Freely homotopic curves.}
\end{figure}
  
\section{Birkhoff: A closed geodesic on a two sphere}

In the 1910s, Birkhoff came up with an ingenious method of constructing non-trivial closed geodesics on a topological $2$-sphere.  Since   $\SS^2$ is simply-connected, this cannot be done by minimizing in a free homotopy class.  Birkhoff's instead used a min-max argument to find higher index critical points.  We will describe Birkhoff's idea and some related results in this section;  see \cite{B1},
\cite{B2} and  section $2$ in \cite{Cr} for more about  Birkhoff's ideas.

\subsection{Sweepouts and the width}
 
The starting point is   a min-max construction that uses a non-trivial homotopy class  of maps from $\SS^2$ to construct a geometric metric called the width.  Later, we will see that this invariant is realized as the length of a closed geodesic.   Of course, one has to assume that such a non-trivial homotopy class exists (it does on $\SS^2$, but not on   higher genus surfaces; fortunately, it is easy to construct minimizers on  higher genus surfaces).

Let $\Omega$ be the set of continuous maps $$\sigma :
\SS^1 \times [0,1] \to M$$ with the following three properties:
\begin{itemize}
\item 
 For each $t$ the map $\sigma (\cdot , t )$
is in $W^{1,2}$. 
\item The map  $t \to \sigma (\cdot , t )$
is continuous
 from $[0,1]$ to
$W^{1,2}$.
\item
$\sigma$ maps $\SS^1 \times \{ 0 \}$ and $\SS^1 \times \{ 1 \}$ to points.
\end{itemize}

 \begin{figure}[htbp]
\centering\includegraphics[totalheight=.4\textheight, width=1\textwidth]{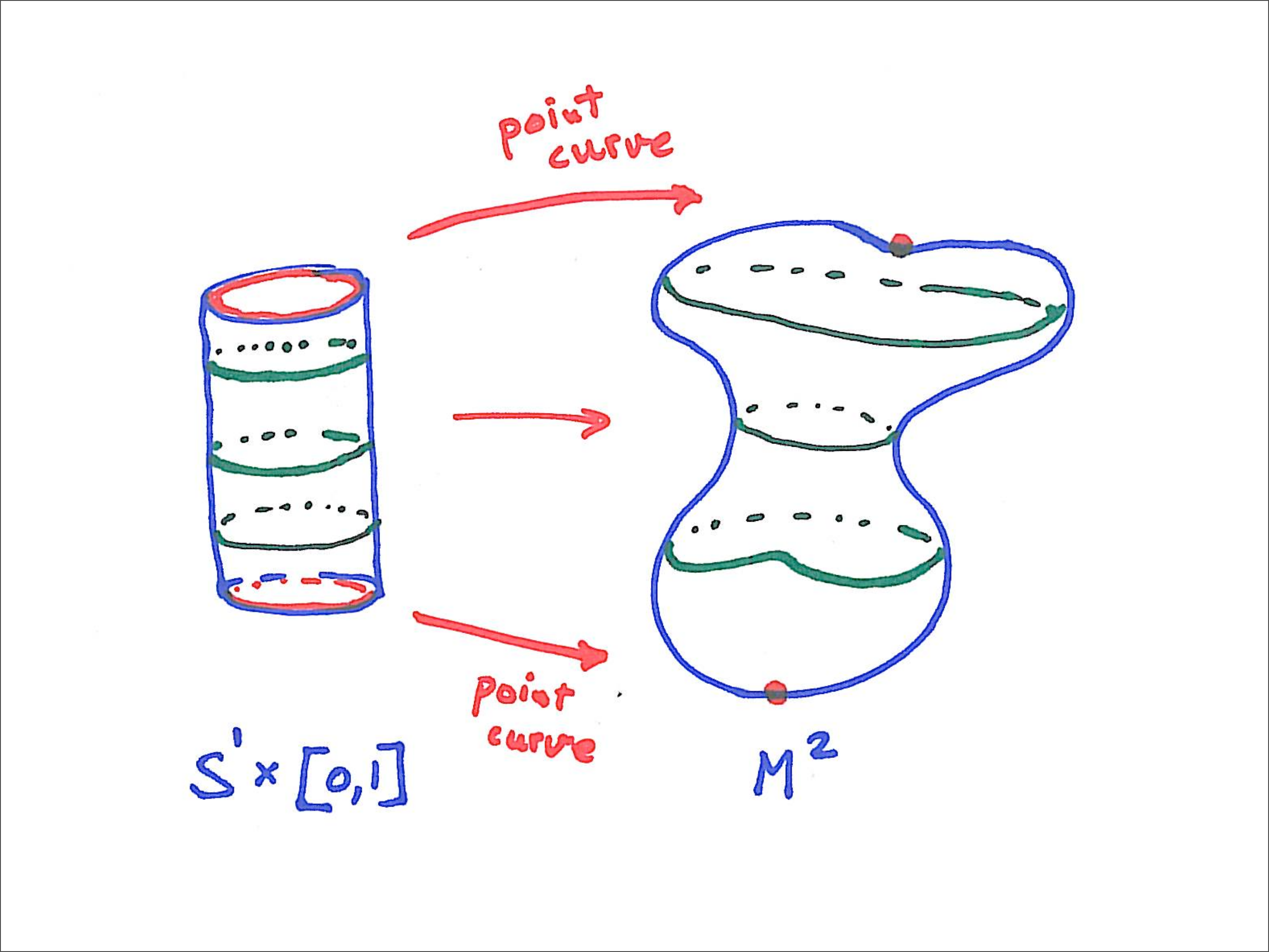}
\caption{A sweepout.}   
  \end{figure}

Given a map
    $\hat{\sigma} \in \Omega$, the homotopy class $\Omega_{\hat{\sigma}}$
 is defined to be the set of maps $\sigma \in \Omega$ that are homotopic to
    $\hat{\sigma}$ through maps in $\Omega$.
The width\index{width} $W = W (\hat{\sigma})$ associated to the homotopy class
 $\Omega_{\hat{\sigma}}$ is defined by taking $\inf$ of $\max$ of
the energy of each slice. That is,  set
\begin{equation}    \label{e:w}
    W = \inf_{  \sigma \in \Omega_{\hat{\sigma}}  } \,
       \,  \max_{ t \in [0, 1]} \,  \Energy \, (\sigma (\cdot , t ))
          \, ,
\end{equation}
where the energy is given by $$\Energy \, (\sigma (\cdot , t )) =
\int_{\SS^1} \, \left| \partial_x \sigma (x,t) \right|^2 \, dx \, .$$
The width is always non-negative and is positive if $\hat{\sigma}$
is in a non-trivial homotopy class.

A particularly
interesting example is when $M$ is a topological $2$-sphere and
the induced map from $\SS^2$ to $M$ has degree one. In this case,
the width is positive and realized by a non-trivial closed
geodesic.  To see that the width is positive on non-trivial
homotopy classes, observe that if the maximal energy of a slice is
sufficiently small, then each curve $\sigma (\cdot , t)$ is
contained in a convex geodesic ball in $M$. Hence,   a geodesic
homotopy connects $\sigma$ to a path of point curves, so $\sigma$
is homotopically trivial.

\subsection{Pulling the sweepout tight to obtain a closed geodesic}

 The key to finding the closed geodesic is to ``pull the sweepout tight'' using the Birkhoff curve shortening process (or BCSP).  The BCPS is 
 a kind of discrete gradient flow on the space of curves.  It is given by subdividing a curve and then replacing first the even segments by minimizing geodesics, then replacing the odd segments by minimizing geodesics, and finally reparameterizing the curve so it has constant speed.  
 
 It is not hard to see that the BCSP has the following properties:
 \begin{itemize}
 \item  It is continuous on the space of curves.
 \item Closed geodesics are fixed under the BCSP.
 \item If a curve is fixed under BCSP, then it is a geodesic.
 \end{itemize}
 It is possible to make each of these three properties quantitative. 
 
   \begin{figure}[htbp]
    \begin{minipage}[t]{0.5\textwidth}
    \centering\includegraphics[totalheight=.3\textheight, width=.9\textwidth]{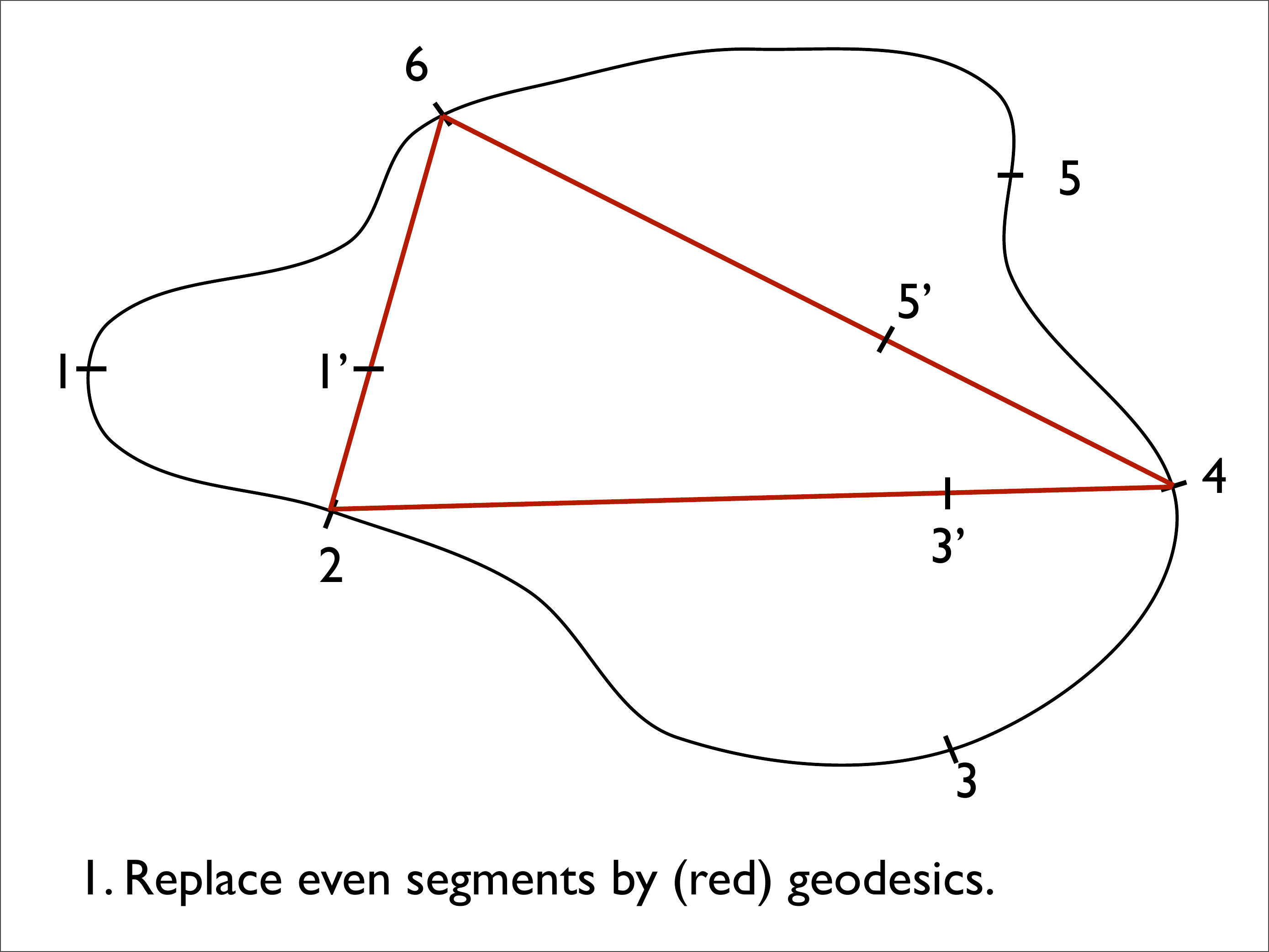}
    %\caption{The minimizing curve $\gamma$ in Hamilton's first isoperimetric quantity.}  
    \end{minipage}\begin{minipage}[t]{0.5\textwidth}
    \centering\includegraphics[totalheight=.3\textheight, width=.9\textwidth]{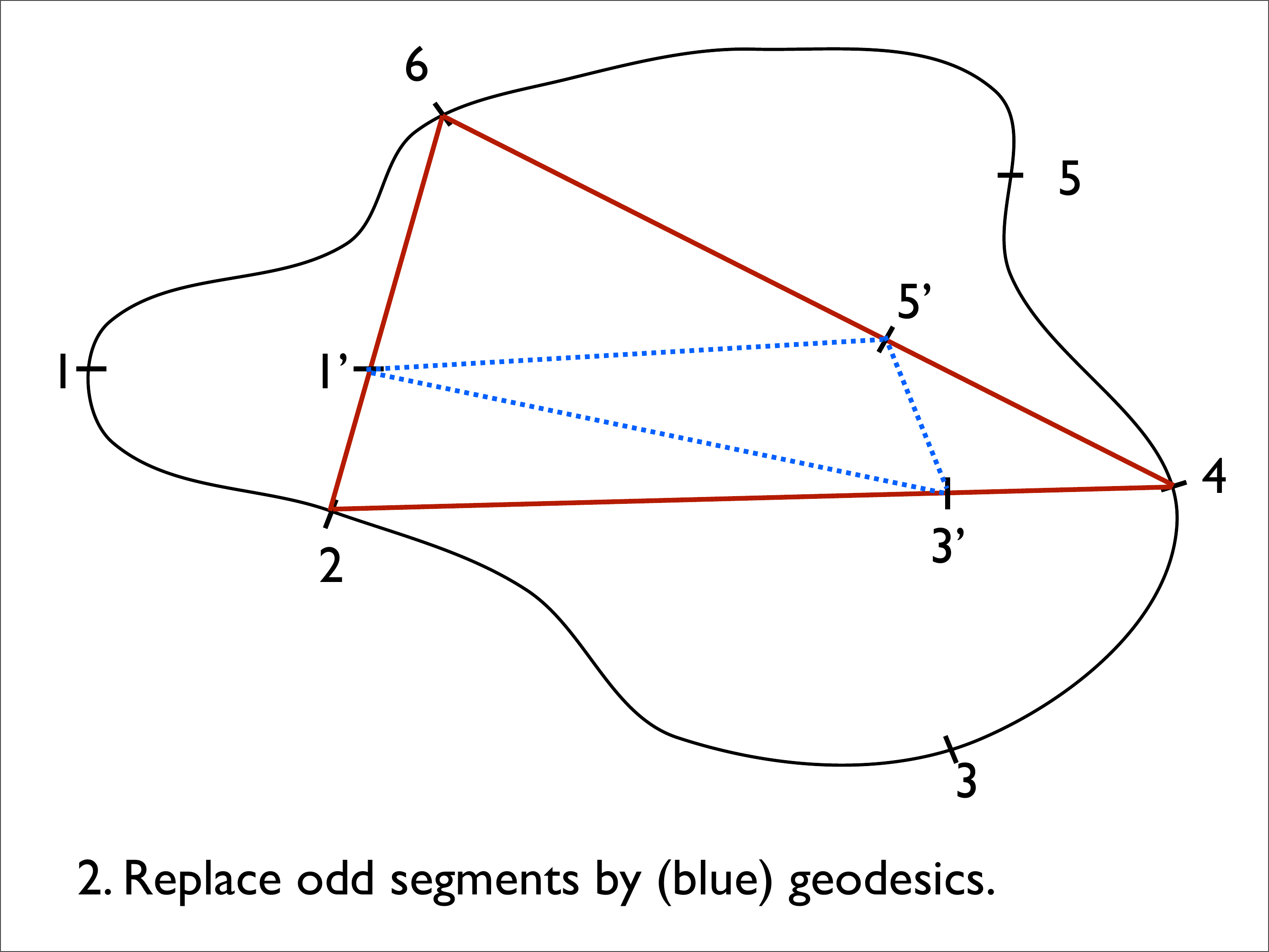}
    %\caption{Defining the length $L_0$ in Hamilton's second isoperimetric quantity.} 
    \end{minipage}
    \caption{Birkhoff's curve shortening process.}   
\end{figure}

 Using this map and more refined versions of these properties, we showed the existence of a sequence of tightened sweepouts:

\begin{Thm}   (Colding-Minicozzi, \cite{CM2})   \label{t:mm}
There exists a sequence of sweepouts $\gamma^j$ with the property that:  Given
  $\epsilon > 0$,  there is $\delta > 0$   so that
if $j
> 1/\delta$ and  
\begin{equation}    \label{e:close1}
    2\pi \, \Energy \, ( \gamma^j (\cdot , t_0))
        = \Length^2 \, ( \gamma^j (\cdot , t_0)) > 2\pi \, (W - \delta)  \, ,
\end{equation}
then for this $j$ we have $\dist \, \left(  \gamma^j (\cdot , t_0)
\, , \, G \right) < \epsilon$ where $G$ is the set of closed geodesics.
\end{Thm}

As an immediate consequence, we get the existence of non-trivial closed geodesics for any metric on $\SS^2$; this is due to Birkhoff.  See \cite{LzWl} for an alternative proof using the harmonic map heat flow.

 \begin{figure}[htbp]
\centering\includegraphics[totalheight=.35\textheight, width=.85\textwidth]{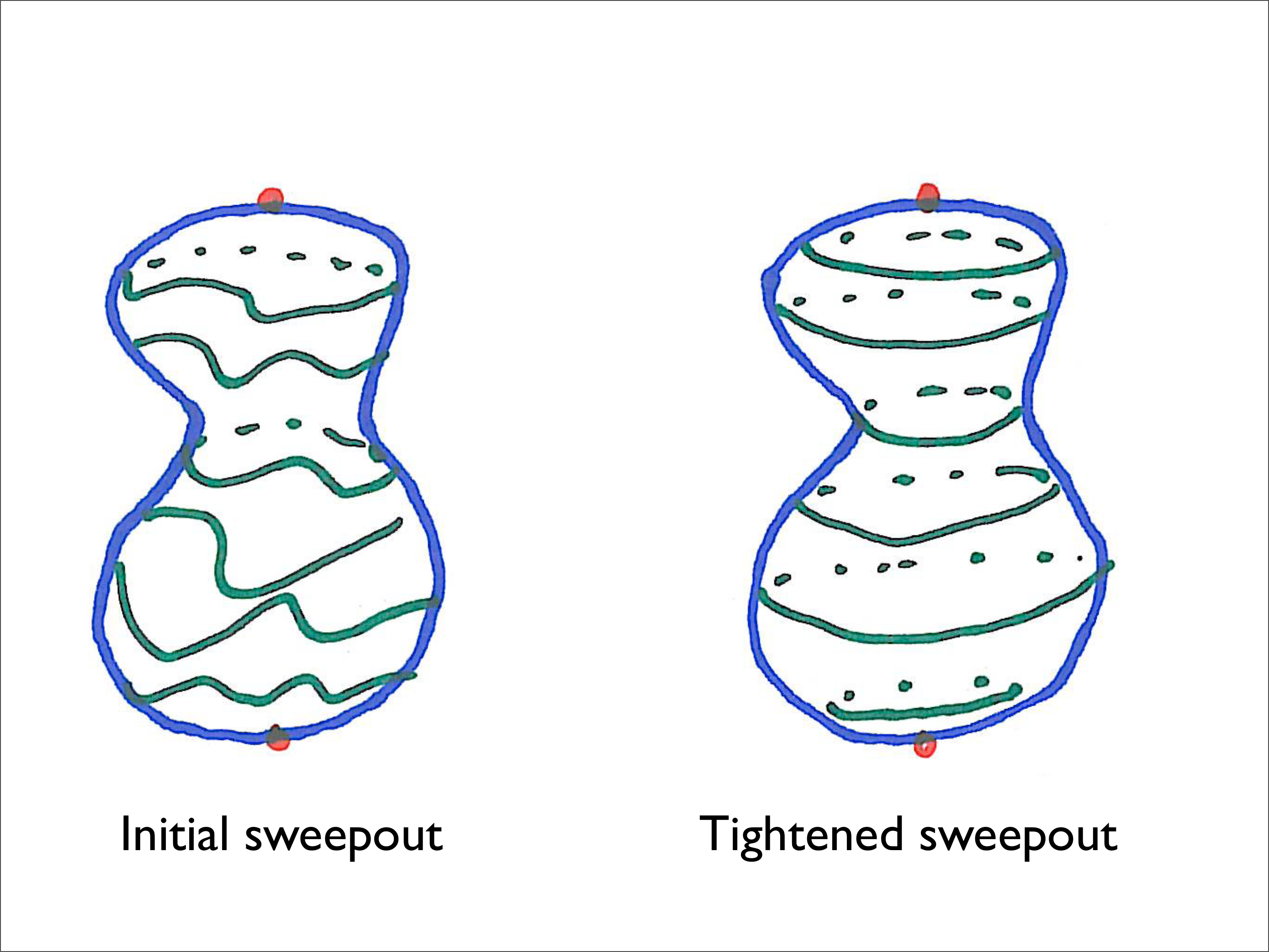}
\caption{Tightening the sweepout.}
\end{figure}

\subsection{Sweepouts by spheres}

We will now define a two-dimensional version of the width, where we sweepout by spheres instead of curves.
Let $\Omega$ be the set of continuous maps $$\sigma : \SS^2 \times
[0,1] \to M$$ so that:\index{sweepouts}
\begin{itemize}
\item 
 For each $t\in [0,1]$ the map $\sigma (\cdot , t )$
is in $C^0 \cap W^{1,2}$. 
\item The map  $t \to \sigma (\cdot , t )$ is
continuous
 from $[0,1]$ to
$C^0 \cap W^{1,2}$.
\item $\sigma$ maps $\SS^2 \times \{ 0
\}$ and $\SS^2 \times \{ 1 \}$ to points.   
\end{itemize}
Given a map
    $\beta \in \Omega$, the homotopy class $\Omega_{\beta}$
 is defined to be the set of maps $\sigma \in \Omega$ that are homotopic to
    $\beta$ through maps in $\Omega$.    We will call any such
     $\beta$  a
{\emph{sweepout}}.

The (energy) width\index{width} $W_E = W_E (\beta , M)$ associated to the homotopy class
 $\Omega_{\beta}$ is defined by taking the infimum of the maximum of
the energy of each slice.  That is,  set
\begin{equation}    \label{e:width}
    W_E = \inf_{  \sigma \in \Omega_{\beta}  } \,
       \,  \max_{ t \in [0, 1]} \,  \Energy \, (\sigma (\cdot , t ))
          \, ,
\end{equation}
where the energy is given by
\begin{equation}
    \Energy \, (\sigma
(\cdot , t )) = \frac{1}{2} \, \int_{\SS^2} \, \left| \nabla_x
\sigma (x,t) \right|^2 \, dx \, .
\end{equation}
 
The next result gives the existence of a sequence of good
sweepouts.

\begin{Thm}  (Colding-Minicozzi, \cite{CM19})  \label{t:existence}
Given a metric $g$ on $M$ and a map $\beta  \in \Omega$
representing a non-trivial class in $\pi_3 (M)$, there exists  a
sequence of sweepouts $\gamma^j \in \Omega_{\beta}$ with
$\max_{s \in [0,1]} \, \Energy (\gamma^j_s)\to W(g)$,
and so that given $\epsilon >
 0$, there exist $\bar{j}$ and $\delta > 0$
so that if $j > \bar{j}$ and
\begin{equation}    \label{e:eclose}
      \Area (\gamma^j (\cdot , s)) > W(g) - \delta \, ,
\end{equation}
then there are finitely many harmonic maps $u_i : \SS^2 \to M$
 with
\begin{equation}    \label{e:bclose}
    \ddd_V  \, (\gamma^j (\cdot , s) , \cup_i \{ u_i \}
     ) < \epsilon \, .
\end{equation}
\end{Thm}

In \eqr{e:bclose},  we have identified each map $u_i$ with the
varifold associated to the pair $(u_i, \SS^2)$ and then taken the
disjoint union of these $\SS^2$'s to get $\cup_i \{ u_i \}$.  The distance $\ddd_V$ in \eqr{e:bclose} is a weak measure-theoretic distance called    ``varifold distance''; see \cite{CM19} or   Chapter $3$ of \cite{CM14} for the definition.

\vskip2mm One immediate consequence of Theorem \ref{t:existence}
is that if $s_j$ is any sequence with $\Area (\gamma^j (\cdot ,
s_j))$ converging to the width $W(g)$ as $j \to \infty$, then a
subsequence of
 $\gamma^j (\cdot ,
s_j)$ converges to a collection of harmonic maps from $\SS^2$ to
$M$.  In particular, the sum of the areas of these maps is exactly
$W(g)$ and, since the maps are automatically conformal, the sum of
the energies is also $W(g)$.  
The existence of at least one
non-trivial harmonic map from $\SS^2$ to $M$ was first proven in
\cite{SaUh}, but they allowed for loss of energy in the limit; cf.
also \cite{St}.   Ruling out this possible energy loss in various settings 
is known as the ``energy identity''\index{energy identity}
and it can be rather delicate.
This energy loss was ruled out by Siu and Yau,
using also arguments of Meeks and Yau (see Chapter VIII in
\cite{ScYa2}). This was also proven later by Jost, \cite{Jo}.

\section{Curve shortening flow}

The Birkhoff curve shortening process was a kind of discrete gradient flow on the space of curves.  We turn next to a continuous gradient flow that is called the curve shortening flow.

Suppose that $\gamma_0$ again is a curve but this time we will think of it as an embedded submanifold in $\RR^2$ or, more generally, a surface $M^2$.
We can again look at variations $\gamma_t$ of the one-dimensional submanifold $\gamma_0$ and get for lengths:
\begin{equation}
\frac{d}{dt} \big|_{t=0} \Length (\gamma_t)=  \int_{\gamma_0} h \, \langle  \nn , V\rangle \, ,  \notag
\end{equation}
where $h$ is the (geodesic) curvature of the one-dimensional submanifold given by
\begin{equation}
 h = \langle \nabla_{e_1} \nn,e_1 \rangle\, ,  \notag
\end{equation}
where $e_1$ is a unit vector tangent to the curve $\gamma_0$ and $\nn$ is the unit normal to $\gamma_0$. It follows that
\begin{enumerate}
\item $\gamma_0$ is a critical point for length if and only if it is a geodesic (after being reparameterized to have constant speed).
\item The negative gradient flow for the length functional in $\RR^2$ is the {\emph{curve shortening flow}}
\begin{equation}
\partial_t x=   h\,\nn\,  .  \notag
\end{equation}
\end{enumerate}

 \begin{figure}[htbp]
\centering\includegraphics[totalheight=.35\textheight, width=1\textwidth]{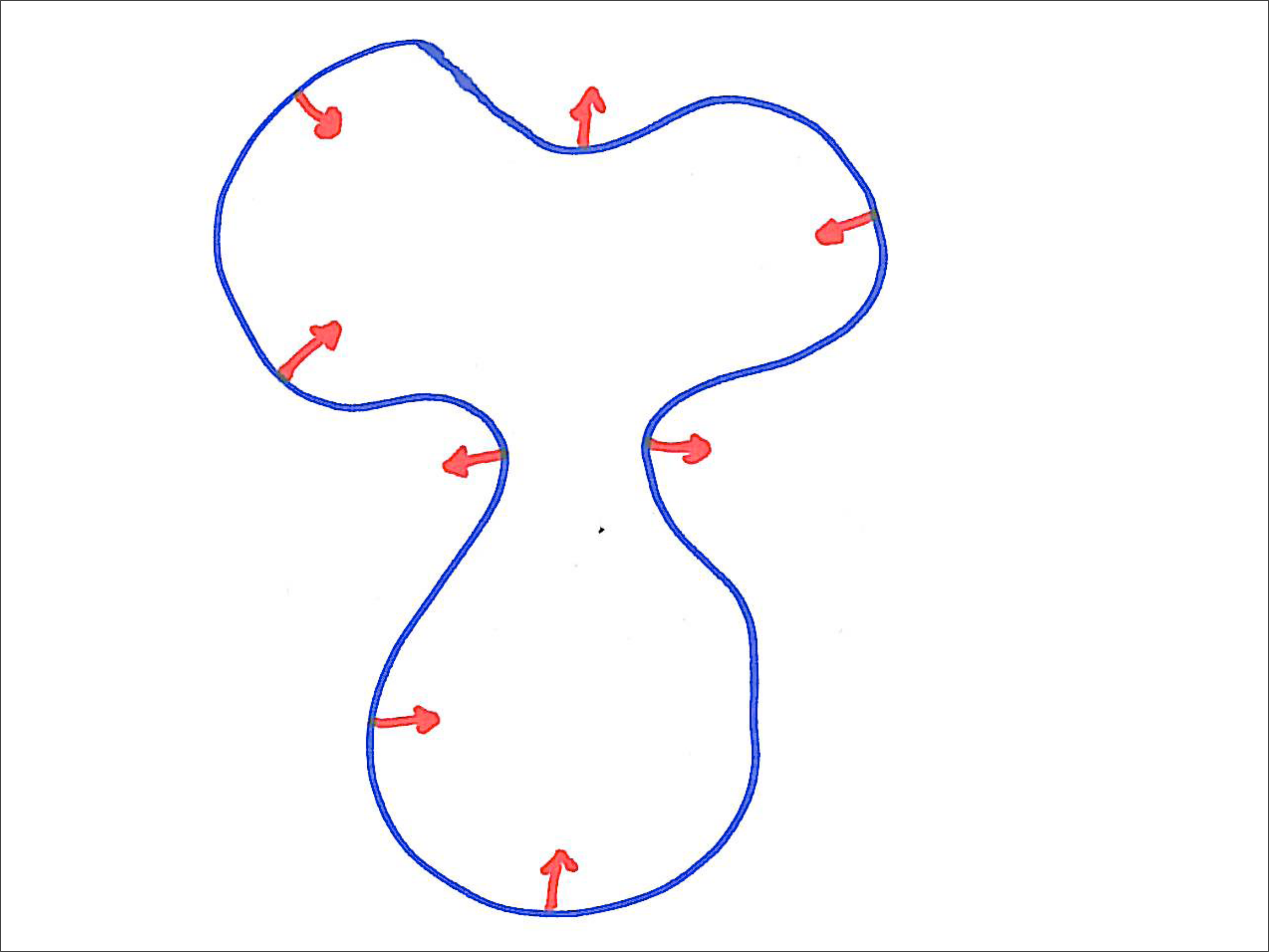}
\caption{Curve shortening flow: the curve evolves by its geodesic curvature.  The red arrows indicate direction of flow.}   
  \end{figure}

The simplest (non-trivial) solution of the curve shortening flow is given by a one-parameter family of concentric circles with
  radius $$r(t) = \sqrt{-2t}$$
   for $t$ in $(-\infty , 0)$.  This is an {\emph{ancient}} solution since it is defined for  all $t< 0$, it is {\emph{self-similar}} since the shape is preserved (i.e., we can think of it as a fixed circle moving under rigid motions of $\RR^2$), and it becomes {\emph{extinct}} at the origin in space and time.

 \begin{figure}[htbp]
\centering\includegraphics[totalheight=.35\textheight, width=.7\textwidth]{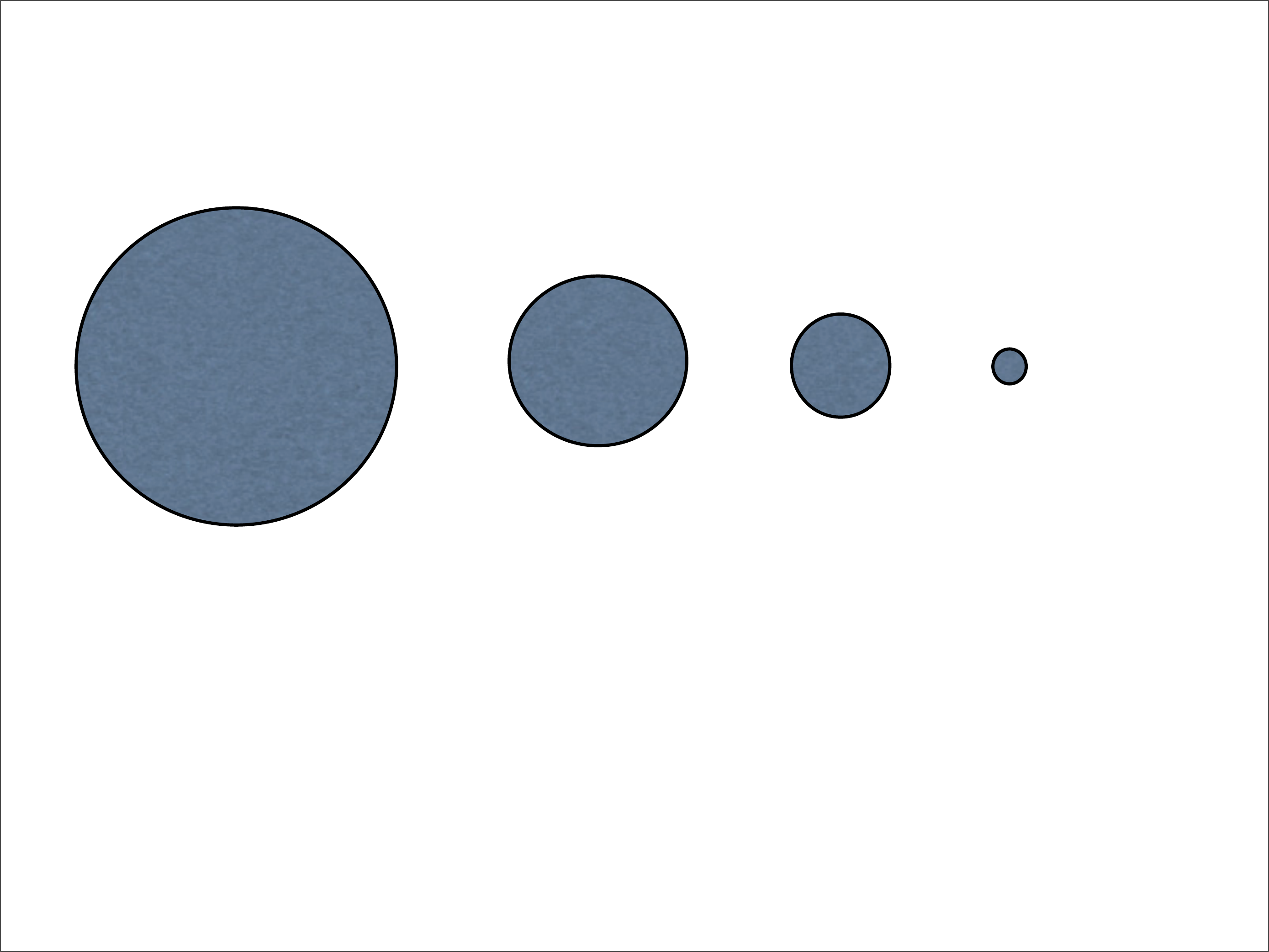}
\caption{Four snapshots in time of concentric circles shrinking under the curve shortening flow.}   
  \end{figure}

\subsection{Self-similar solutions}

A solution of the curve shortening flow is self-similar if the shape does not change with time.  The simplest example is a static solution, like a straight line, that does not change at all.  The next simplest is given by concentric shrinking circles, but there are many other interesting possibilities.  There are three types of self-similar solutions that are most frequently considered:
\begin{itemize}
\item Self-similar shrinkers.
\item Self-similar translators.
\item Self-similar expanders.
\end{itemize}

We will explain shrinkers first.
Suppose that
$c_t$ is a  one-parameter family of curves flowing by the curve shortening flow for $t< 0$.  
We say that $c_t$ is a 
{\emph{self-similar shrinker}} if $$c_t = \sqrt{-t} \, c_{-1}$$ for all $t< 0$.  For example, 
circles of radius $\sqrt{-2t}$ give such a solution.  In 1986, 
Abresch and Langer, \cite{AbLa}, classified such solutions and showed that the shrinking circles give the only embedded one (cf. Andrews, \cite{An}).  In 1987,  Epstein-Weinstein, \cite{EpW}, showed a similar classification 
and  analyzed  the dynamics of the curve shortening flow near a shrinker.

We say that $c_t$ is a {\emph{self-similar translator}} if there is a constant vector $V \in \RR^2$ so that
$$
c_t = c_0 + t \, V 
$$
for all $t \in \RR$.  These solutions are {\emph{eternal}} in that they are defined for all time.  It is easy to see that any translator must be non-compact.
Calabi discovered a self-similar translator in the plane that he named the {\emph{grim reaper}}. 
 Calabi's Grim Reaper is given as the graph of the 
  function 
\begin{equation}
u(x,\,t)=t-\log\sin x\, .  \notag
\end{equation}

{\emph{Self-similar expanders}} are similar to shrinkers, except that they move by expanding dilations.  In particular, the solutions are defined as $t$ goes to $+ \infty$.  It is not hard to see that expanders must be non-compact.

There are other possible types of self-similar solutions, where the solutions move by one-parameter families of rigid motions over time.  See Halldorsson, \cite{Hh}, for other self-similar solutions to the curve shortening flow.

\begin{figure}[htbp]
\centering\includegraphics[totalheight=.4\textheight, width=.8\textwidth]{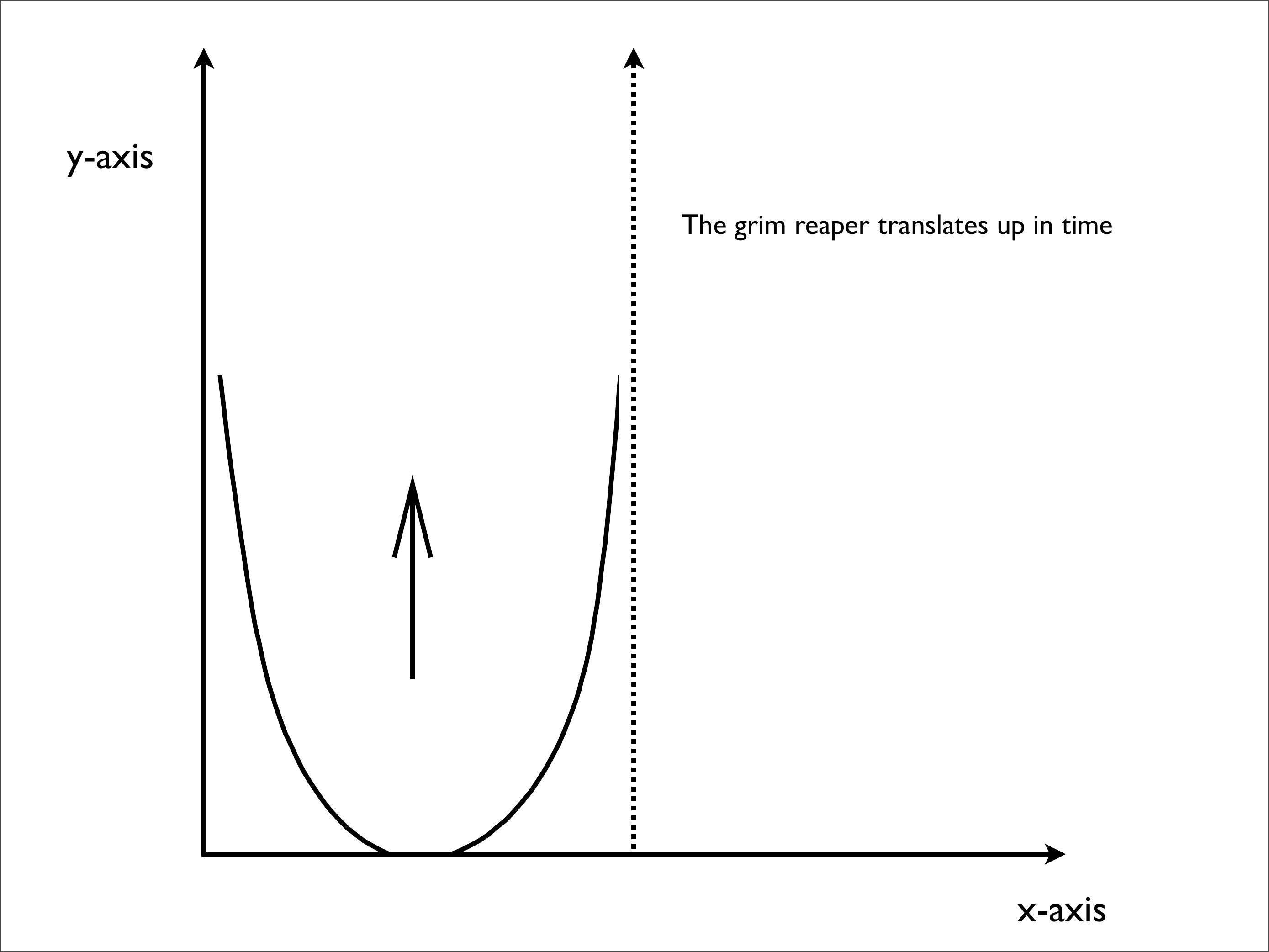}
\caption{Calabi's grim reaper moves by translations.}   
  \end{figure}

\subsection{Theorems of Gage-Hamilton and Grayson}

In 1986, building on earlier work of Gage,   \cite{Ga1} and \cite{Ga2},
Gage and  Hamilton classified closed convex solutions of the curve shortening flow:

\begin{Thm}
(Gage-Hamilton, \cite{GaH})
Under the curve shortening flow every simple closed convex curve remains smooth and convex and eventually becomes extinct in a ``round point''.
\end{Thm}

More precisely, they showed that the flow becomes extinct in a point and if the flow is rescaled to keep the enclosed area constant, then the resulting curves converge to a round circle.  They did this by tracking the isoperimetric ratio and showing that it was approaching the optimal ratio which is achieved by round circles.

In 1987, M. Grayson, \cite{G1}, showed that any simple closed curve eventually becomes convex under the flow:

\begin{Thm}
(Grayson, \cite{G1})
Any simple closed curve eventually becomes convex under the curve shortening flow.
Thus, by the result of Gage-Hamilton, it becomes extinct in a ``round point''.  
\end{Thm}

\begin{figure}[htbp]
\centering\includegraphics[totalheight=.35\textheight, width=.95\textwidth]{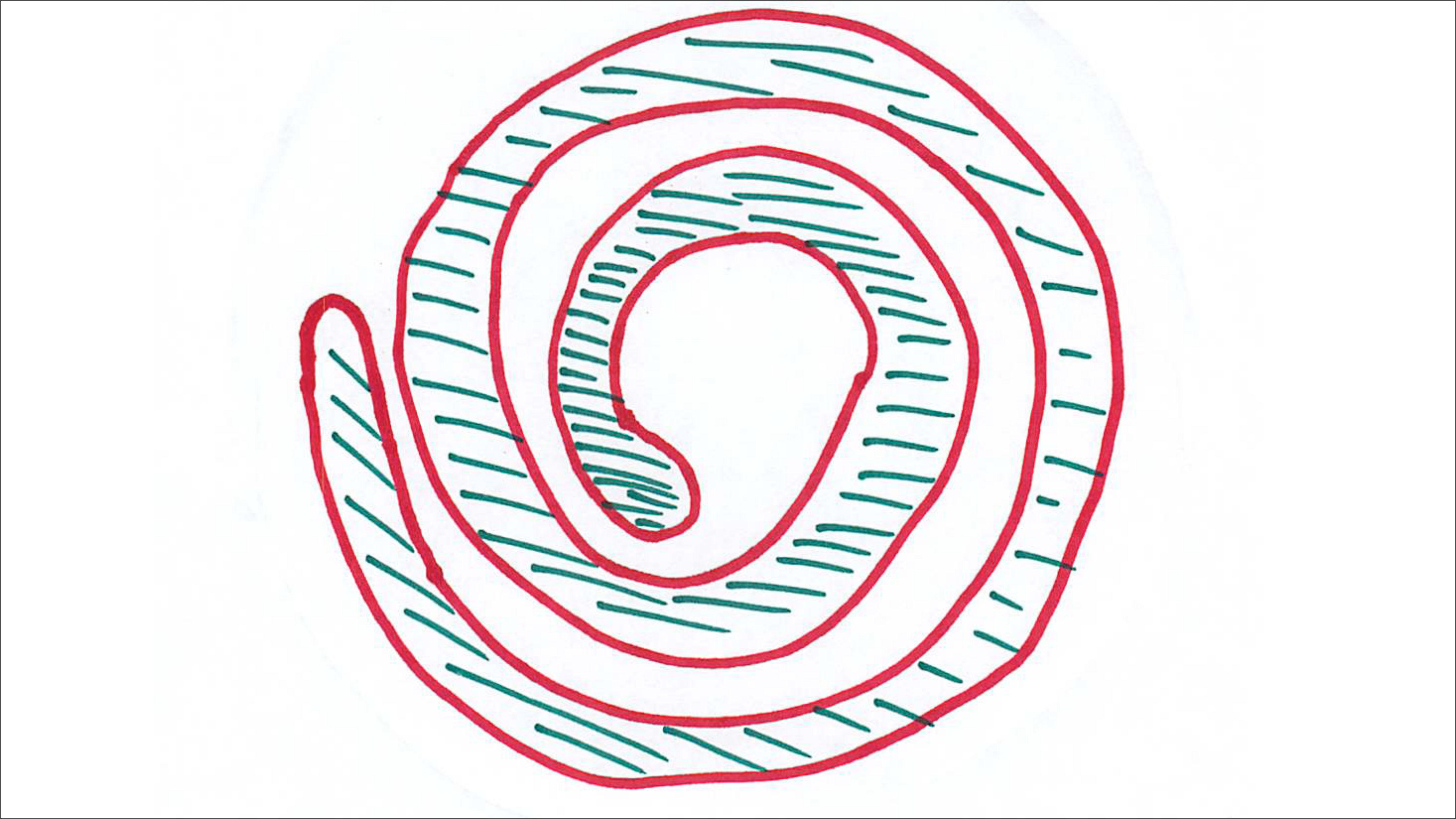}
\caption{The snake manages to unwind quickly enough to  become convex before extinction.}   
  \end{figure}

\subsection{Isoperimetric monotonicity under the curve shortening flow}

In 1995,  Hamilton, \cite{Ha1}, and Huisken, \cite{H5}, discovered two beautiful new ways to prove Grayson's theorem.  Both of these relied on proving monotonicity of various isoperimetric ratios under the curve shortening flow and using these to rule out singularities other than shrinking circles.  Recently, Andrews and Bryan, \cite{AnB} discovered another monotone quantity and used it to give a 
self-contained{\footnote{``Self-contained'' means avoiding the use of a blow up analysis.}} proof of Grayson's theorem.  We will describe two monotone quantities discovered by Hamilton.

\vskip2mm
 For both of Hamilton's quantities, we start with a simple closed curve
$$c:\SS^1\to \RR^2 \, .$$  
  The image of $c$  encloses a region in $\RR^2$.  Each simple curve $\gamma$ inside this region with boundary in the image of $c$ divides the region into two subdomains;  let $A_1$ and $A_2$ be the areas of these subdomains and let $L$ be the length of the dividing curve $\gamma$.
  
  \vskip2mm
  Hamilton's first quantity $I$ is defined to be
\begin{equation}
  I =\inf_{\gamma} \, \,  L^2\, \left(\frac{1}{A_1}+\frac{1}{A_2}\right)\, ,   \notag
\end{equation}
where the infimum is over all possible dividing curves $\gamma$.

\begin{figure}[htbp]
    \begin{minipage}[t]{0.5\textwidth}
    \centering\includegraphics[totalheight=.3\textheight, width=.9\textwidth]{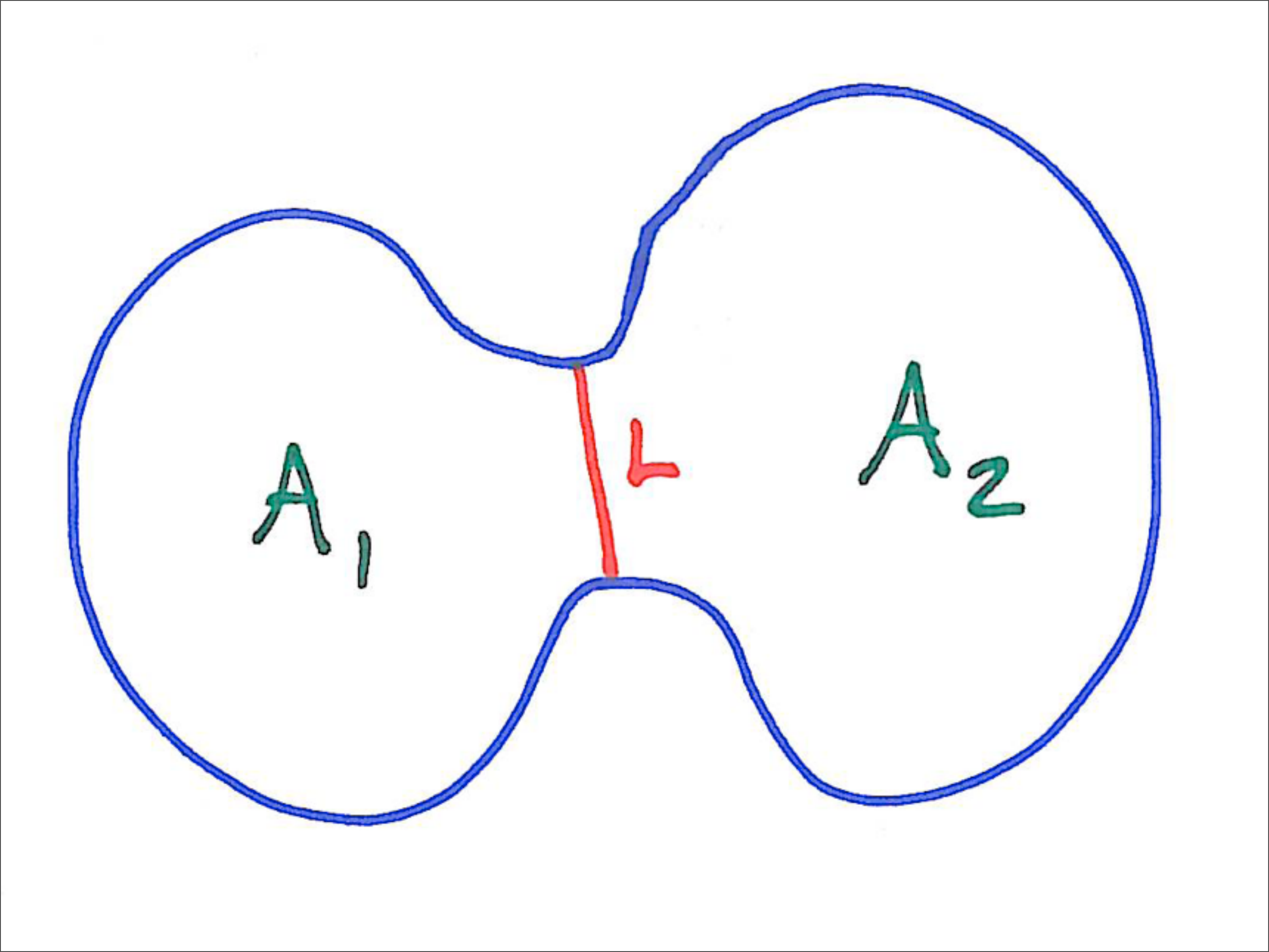}
    \caption{The minimizing curve $\gamma$ in Hamilton's first isoperimetric quantity.}  
    \end{minipage}\begin{minipage}[t]{0.5\textwidth}
    \centering\includegraphics[totalheight=.3\textheight, width=.9\textwidth]{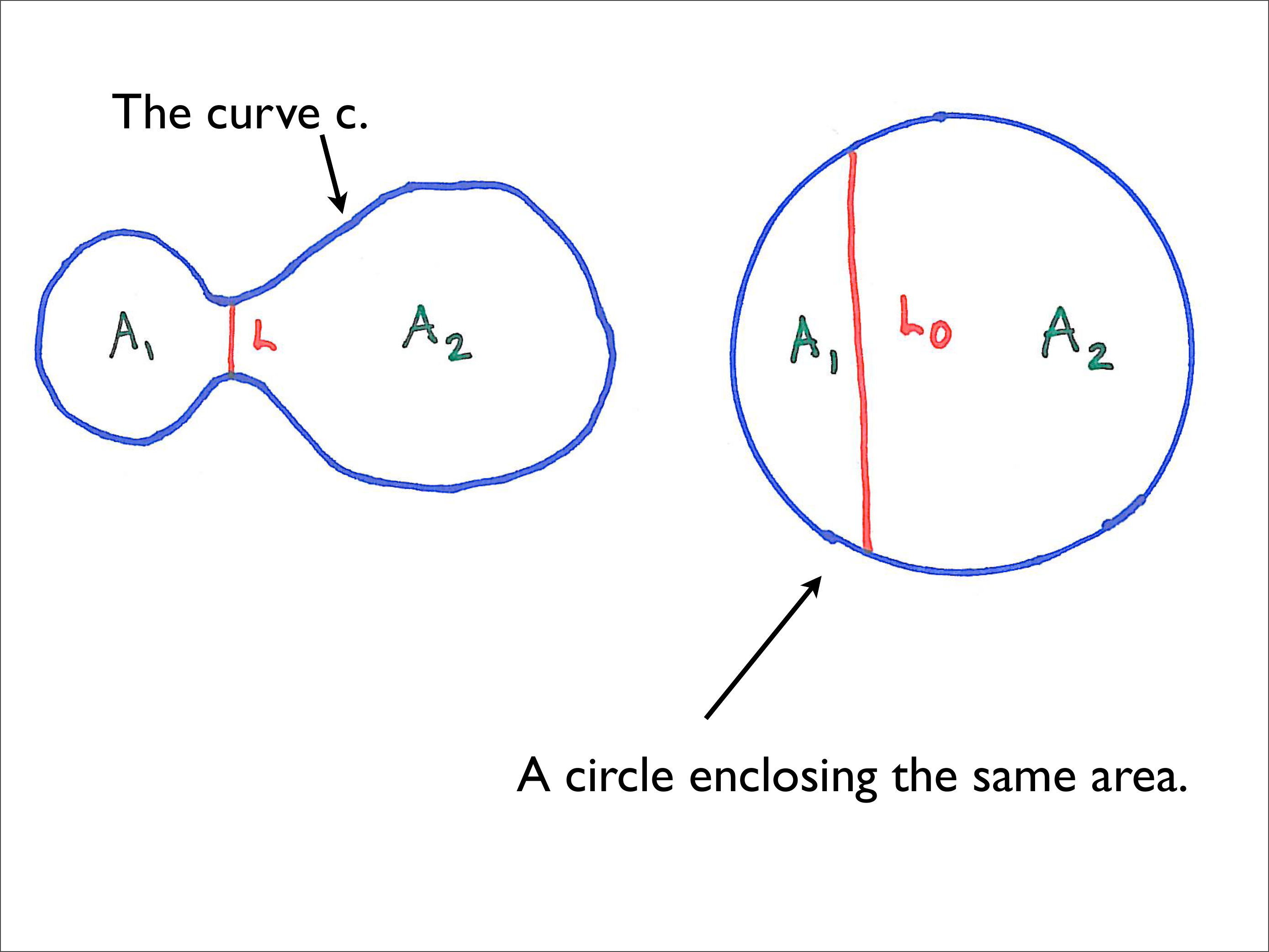}
    \caption{Defining the length $L_0$ in Hamilton's second isoperimetric quantity.} 
    \end{minipage}
\end{figure}

  \begin{Thm}
  (Hamilton, \cite{Ha1})
  Under the curve shortening flow,   $I$ increases  if $I \leq\pi$.
\end{Thm}

Hamilton's second quantity $J$ is defined to be
\begin{equation}
 J=  \inf_{\gamma} \,  \frac{L}{L_0} \, ,   \notag
\end{equation}
where the infimum is again taken over  all possible dividing curves $\gamma$ and the quantity $L_0 = L_0 (A_1 , A_2)$ is 
the length of the shortest curve which divides a circle of area $A_1+A_2$ into two pieces of area $A_1$ and $A_2$.

    \begin{Thm}
  (Hamilton, \cite{Ha1})
  Under the curve shortening flow,   $J$ always increases.
\end{Thm}

We will give a rough idea why these theorems are related to Grayson's theorem.  Grayson had to rule out  a singularity developing before the curve became convex.  As you approach a singularity, the geodesic curvature $h$ must be larger and larger.  Since the curve is compact, one can magnify the curve just before this singular time to get a new curve where the maximum of $|h|$ is one.  There is a blow up analysis that shows that this dilated curve must look like a circle unless it is very long and skinny (like the grim reaper).  Finally, a bound on any of these isoperimetric quantities rules out these long skinny curves.

\section{Minimal surfaces}

We turn next to higher dimensions and the variational properties of the area functional.  Critical points of the area functional are called {\emph{minimal surfaces}}.  In this section, we will give a rapid overview of some of the basic properties of minimal surfaces;
  see the book \cite{CM14} for more details.

\subsection{The first variation of area for surfaces}

Let $\Sigma_0$ be a  hypersurface   in $\RR^{n+1}$ and $\nn$ its unit normal.  Given a vector field 
$$V: \Sigma \to \RR^{n+1}$$
with compact support, we get a one-parameter family of hypersurfaces
\begin{equation}
	\Sigma_s = \{ x +s \, V(x) \, | \, x \in \Sigma_0 \} \, . \notag
\end{equation}
The first variation of area (or volume) is
\begin{equation}
 \frac{d}{ds} \, \big|_{s=0} \, \, \Vol \, (\Sigma_s)  = \int_{\Sigma_0} \dv_{\Sigma_0} V  \, , 
 \notag
\end{equation}
where the divergence $\dv_{\Sigma_0}$ is defined by
\begin{equation}
\dv_{\Sigma_0} \, V  = \sum_{i=1}^n \langle \nabla_{e_i} V , e_i \rangle \, ,  \notag
\end{equation}
where $e_i$ is an orthonormal frame for $\Sigma$.
The vector field $V$ can be decomposed into the part $V^T$ tangent to $\Sigma$ and the normal part $V^{\perp}$.  The divergence of the normal part $V^{\perp} =
\langle V , \nn \rangle \nn$ is
\begin{equation}
	\dv_{\Sigma_0} V^{\perp} = \langle \nabla_{e_i} \, \left( \langle V , \nn \rangle \nn \right) , e_i \rangle  = 
	\langle V , \nn \rangle \, \langle \nabla_{e_i} \,    \nn , e_i \rangle 
	= H \,  \langle V , \nn \rangle \, ,  \notag
\end{equation}
where  the mean curvature scalar $H$ is 
\begin{equation}
H=\text{div}_{\Sigma_0} (\nn)  = \sum_{i=1}^n \langle \nabla_{e_i} \nn , e_i \rangle \, .  \notag
\end{equation}
 With this normalization, $H$ is $n/R$ on the $n$-sphere of radius $R$.

\subsection{Minimal surfaces}

By Stokes' theorem, $\dv_{\Sigma_0} V^T$ integrates to zero.  Hence, since  $	\dv_{\Sigma_0} V^{\perp} 	= H \,  \langle V , \nn \rangle $, 
we can rewrite the first variation formula as
\begin{equation}
 \frac{d}{ds} \, \big|_{s=0} \, \, \Vol \, (\Sigma_s)  = \int_{\Sigma_0}  H \,  \langle V ,   \nn \rangle \, .
 \notag
\end{equation}
A hypersurface $\Sigma_0$ is {\emph{minimal}} when it is a critical point for the area functional, i.e., when the first variation is zero for every compactly supported vector field $V$.
By the first variation formula, this is equivalent to $H = 0$.

 \begin{figure}[htbp]
\centering\includegraphics[totalheight=.4\textheight, width=.8\textwidth]{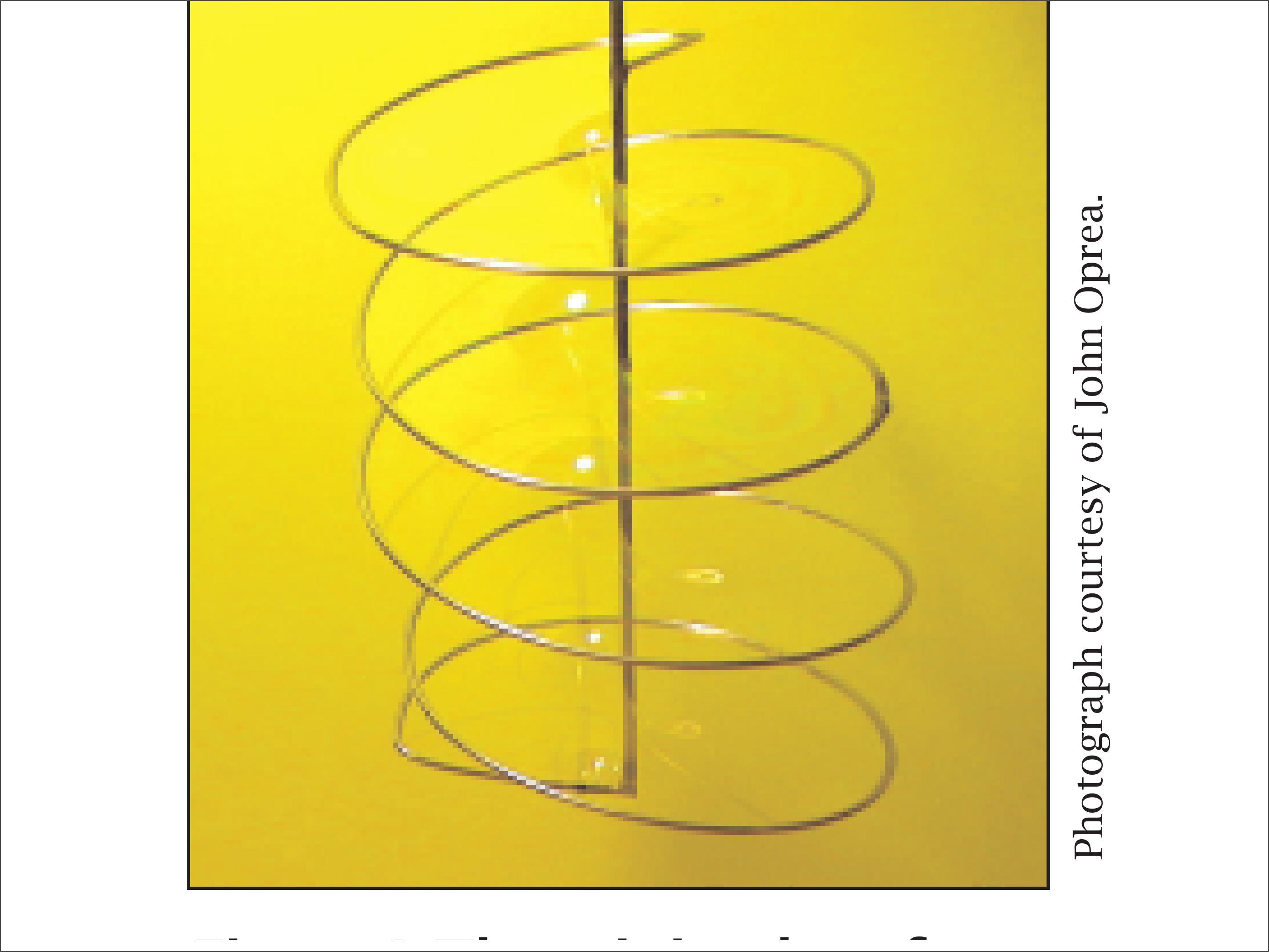}
\caption{The minimal surface called the helicoid is a double-spiral staircase.
This photo shows half of a helicoid as a soap film.}   
  \end{figure}

\subsection{Minimal graphs}

If $\Sigma$ is the graph of a function $u: \RR^n \to \RR$, then the upward-pointing unit normal is given by
\begin{equation}
	\nn = \frac{( - \nabla u , 1)}{ \sqrt{1 + \left| \nabla_{\RR^n} u  \right|^2 } } \, ,
\end{equation}
and the 
 mean curvature of $\Sigma$ is given by
\begin{equation}
	H   =  - \dv_{\RR^n} \, \, \left(  \frac{ \nabla_{\RR^n} u }{ \sqrt{1 + \left| \nabla_{\RR^n} u  \right|^2 } } \right)  \, . 
\end{equation}  	
Thus, minimal graphs are solutions of the nonlinear divergence-form PDE
\begin{equation}
	 \dv_{\RR^n} \, \, \left(  \frac{ \nabla_{\RR^n} u }{ \sqrt{1 + \left| \nabla_{\RR^n} u  \right|^2 } } \right)  = 0 \, .  
\end{equation}  
Every smooth hypersurface is locally graphical, so small pieces of a minimal surface satisfy this equation (over some plane).

In 1916, Bernstein proved that planes were the only entire solutions of the minimal surface equation:

\begin{Thm}
(Bernstein, \cite{Be}) Any minimal graph over all of $\RR^2$ must be flat (i.e., $u$ is an affine function).
\end{Thm}

Remarkably, this theorem holds for $n \leq 7$, but there are non-flat entire minimal graphs in dimensions $8$ and up.

The Bernstein theorem should be compared with the classical Liouville theorem for harmonic functions:

\begin{Thm}
(Liouville) A positive harmonic function on $\RR^n$ must be constant.
\end{Thm}

\subsection{Consequences of the first variation formula}

Suppose that $\Sigma \subset \RR^{n+1}$ is a hypersurface with normal $\nn$. 
 Given $f: \RR^{n+1} \to \RR$, the Laplacian on $\Sigma$ applied to $f$ is
\begin{equation}	\label{e:lapla}
	\Delta_{\Sigma} f \equiv \dv_{\Sigma}  (\nabla f)^T
	= \sum_{i=1} \Hess_f (e_i , e_i) - \langle \nabla f , \nn \rangle \, H \, ,   
\end{equation}
where $e_i$ is a frame for $\Sigma$ and $\Hess_f$ is the $\RR^{n+1}$ Hessian of $f$.  We will use this formula several times with different choices of $f$.

First, when $f$ is the $i$-th coordinate function $x_i$, \eqr{e:lapla} becomes
\begin{equation}
	\Delta_{\Sigma} x_i 
	=  - \langle \partial_i , \nn \rangle \, H \, . \notag 
\end{equation}
We see that:

\begin{Lem}
$\Sigma$ is minimal $\iff$ all coordinate functions are harmonic.
\end{Lem}

 Combining this with the maximum principle, we get Osserman's convex hull property, \cite{Os3}:
 
 \begin{Pro}	\label{p:cvx}
 If $\Sigma$ is compact and minimal, then $\Sigma$ is contained in the convex hull of $\partial \Sigma$.
 \end{Pro}
 
 \begin{proof}
 If not, then we could choose translate and rotate $\Sigma$ so that $\partial \Sigma \subset \{ x_1 < 0 \}$ but
 $\Sigma$ contains a point $p \in \{ x_1 > 0 \}$.  However, the function $x_1$ is harmonic on $\Sigma$, so the maximum principle implies that its maximum is on $\partial \Sigma$.  This contradiction proved the proposition.
 \end{proof}
 
 Applying \eqr{e:lapla} with $f=|x|^2$ and noting that 
the $\RR^{n+1}$ Hessian of $|x|^2$ is twice the identity and the gradient is $2x$, we see that
\begin{equation}
	\Delta_{\Sigma} |x|^2
	=  2\, n - 2 \, \langle x , \nn \rangle \, H  \, , \notag 
\end{equation}
when $\Sigma^n \subset \RR^{n+1}$ is a hypersurface.
When $\Sigma$ is minimal, this becomes
\begin{equation}
	\Delta_{\Sigma} |x|^2
	=  2\, n  \, . \notag 
\end{equation}
This identity is the key for the monotonicity formula:

\begin{Thm}	\label{t:monomin}
If $\Sigma \subset \RR^{n+1}$ is a minimal hypersurface, then 
\begin{equation}
	\frac{d}{dr} \, \, \frac{ \Vol (B_r \cap \Sigma)}{r^n} = 
	 \frac{1}{r^{n+1}} \, \int_{\partial B_r \cap \Sigma} \, \frac{ \left| x^{\perp} \right|^2 }{\left| x^T \right|} \geq 0 \, .  \notag
\end{equation}
\end{Thm}

Moreover,  the density ratio is constant  if and only if $x^{\perp} \equiv 0$; this is equivalent to   $\Sigma$ being  a cone with its vertex at the origin (i.e., $\Sigma$ is invariant with respect to dilations about $0$).

\subsection{Examples of minimal surfaces}

\subsubsection{The Catenoid}

The catenoid, shown in figure \ref{f:f2},  is the only non-flat minimal surface of revolution.  It was discovered by Euler in 1744
and
shown to be minimal by Meusnier (a student of Monge) in 1776.  
It is a complete embedded topological annulus (i.e., genus zero and two ends) and is given
  as the set where $x_1^2 + x_2^2 = \cosh^2 (x_3)$ in $\RR^3$.  It is easy to see that the catenoid has finite total curvature.

 \begin{figure}[htbp]
\centering\includegraphics[totalheight=.4\textheight, width=.85\textwidth]{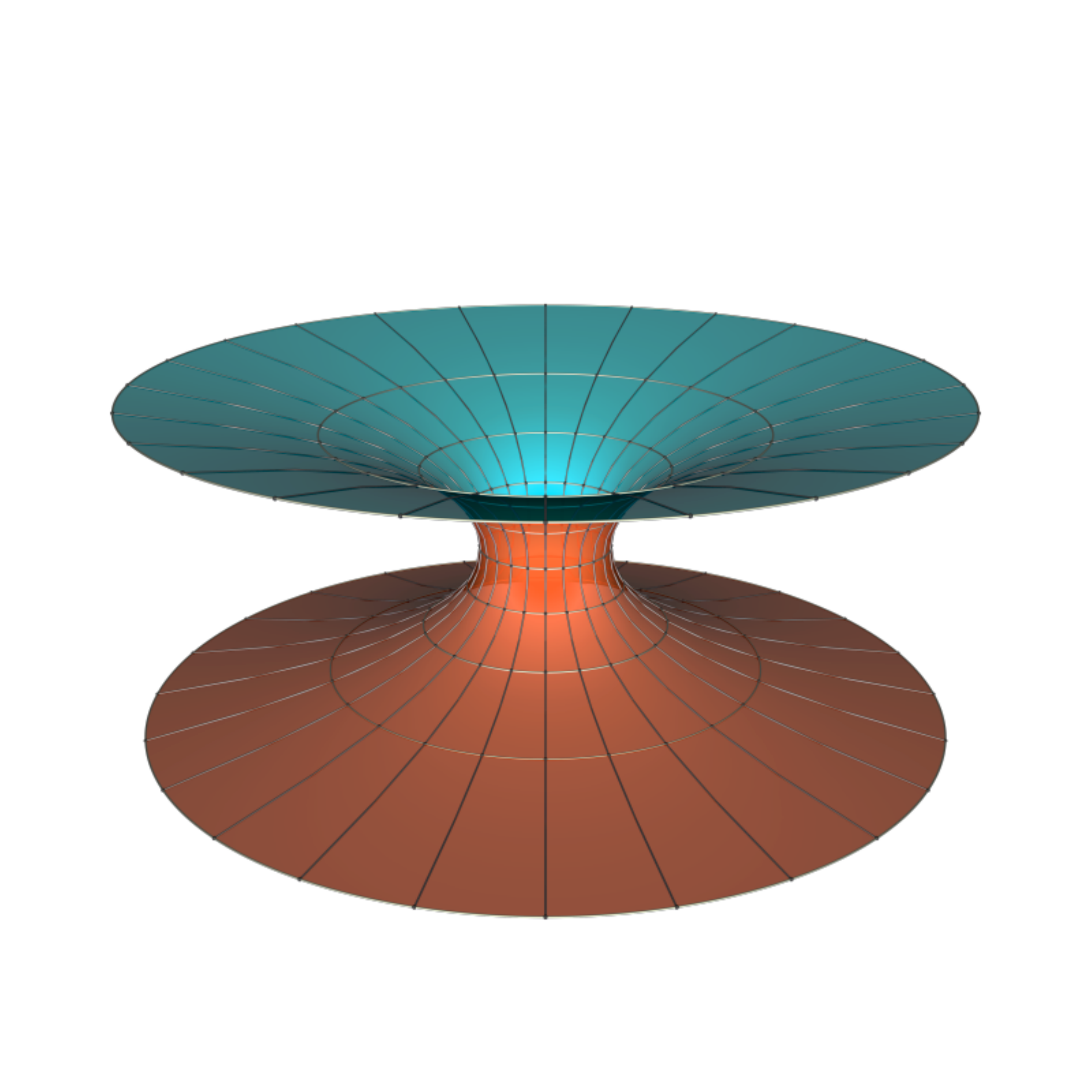}
\caption{The catenoid given by revolving $x_1= \cosh x_3$
around the $x_3$--axis. Credit: Matthias Weber, www.indiana.edu/~minimal.}    \label{f:f2}
  \end{figure}

\subsubsection{The Helicoid}  \label{e:xheli}

The helicoid (see figure \ref{f:f1}) is given as the set
\rm   $x_3 =\tan ^{-1} \left(
\frac{x_2}{x_1}\right)$; alternatively, it is given in parametric form by
\begin{equation}    \label{e:helicoid}
    (x_1,x_2,x_3)=(t\,\cos s,t\,\sin s,s) \, ,
\end{equation}
 where $s$, $t\in \RR$.   It was discovered by Meusnier (a student of Monge) in 1776.
  It is complete, embedded, singly-periodic and simply connected.
 
 The helicoid is a ruled surface since its
 intersections with horizontal planes $\{ x_3 = s \}$ are
 straight lines.  These lines lift and rotate with constant speed
 to form a double spiral staircase.  In 1842, Catalan showed that the helicoid is the only (non-flat) ruled minimal surface.
 A surface is said to be ``ruled'' if it can be parameterized by
\begin{equation}
	X(s,t) = \beta (t) + s \, \delta (t) {\text{ where }} s ,t \in \RR \, , 
\end{equation}
and $\beta$ and $\delta$ are curves in $\RR^3$.  The curve $\beta (t)$ is called the ``directrix'' of the surface, and a line having $\delta (t)$ as direction vector is called a ``ruling''.  For the standard helicoid, the $x_3$-axis is a directrix, and for each fixed $t$ the line $s \to (s\cos t , s \, \sin t , t )$ is a ruling.
 
 \begin{figure}[htbp]
    \centering\includegraphics[totalheight=.55\textheight, width=.85\textwidth]{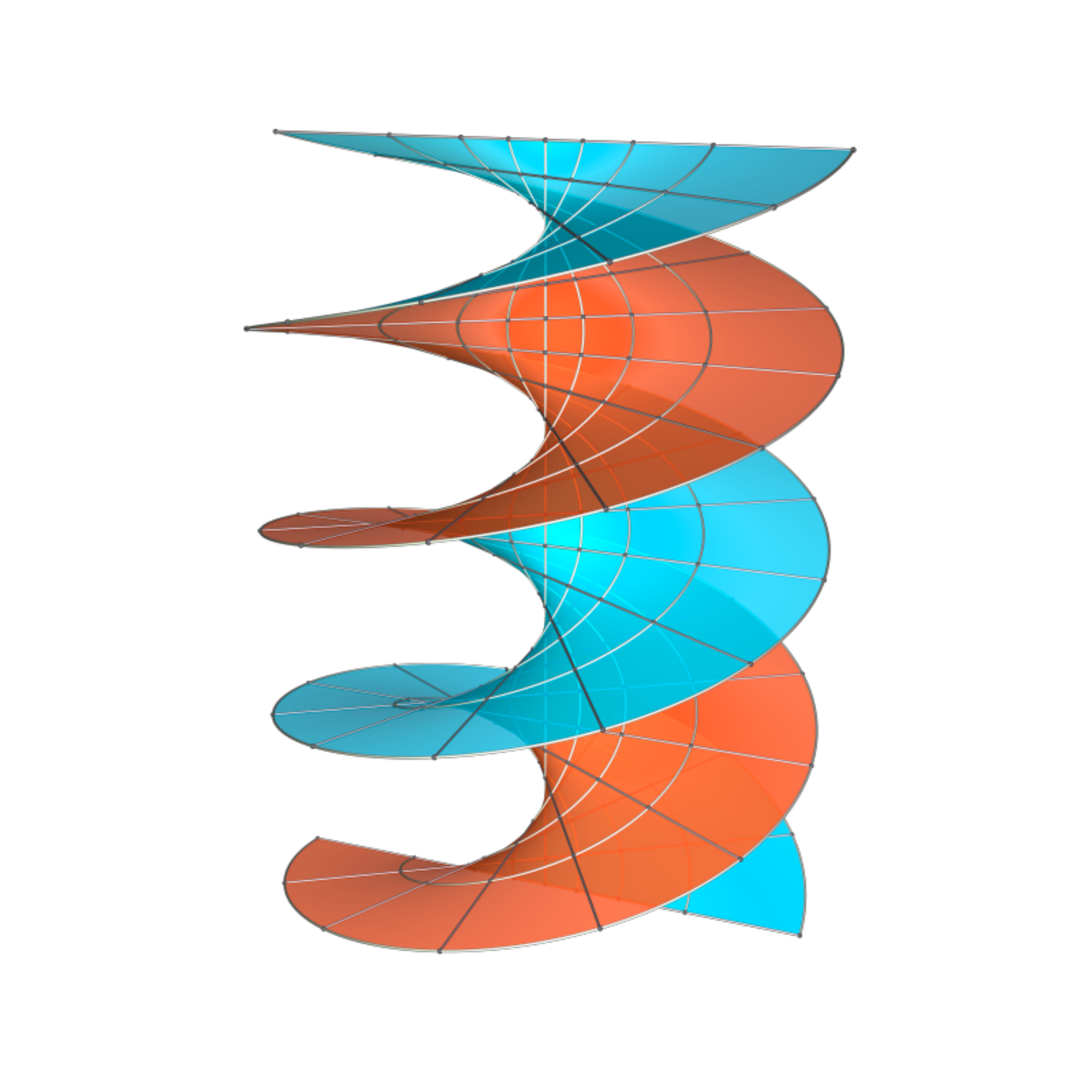}
    \caption{The helicoid, with the ruling pictured. Credit: Matthias Weber, www.indiana.edu/~minimal.}  \label{f:f1}
\end{figure}

\subsubsection{The Riemann Examples}

\begin{figure}[htbp]
    \begin{minipage}[t]{0.5\textwidth}
    \centering\includegraphics[totalheight=.4\textheight, width=.9\textwidth]{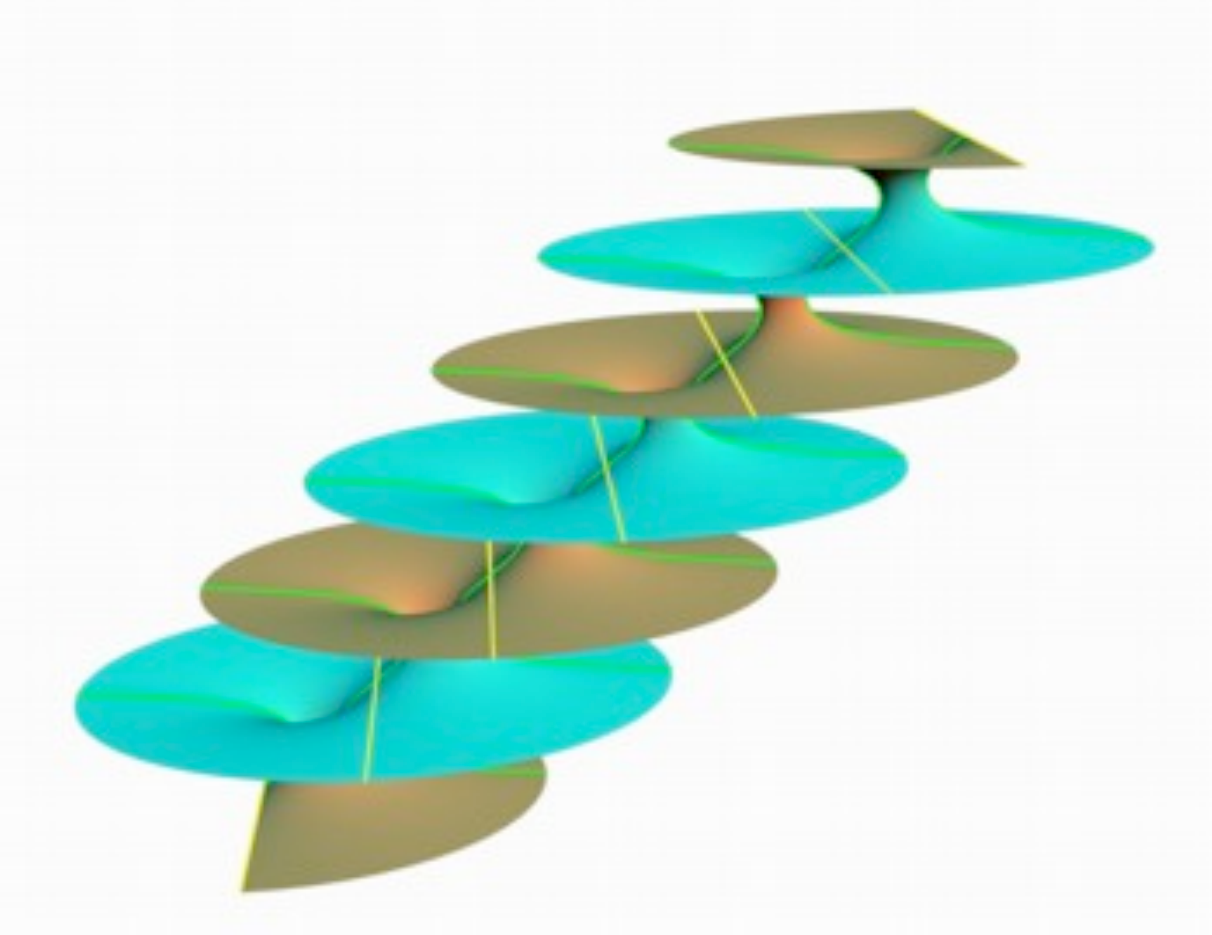}
     \end{minipage}\begin{minipage}[t]{0.5\textwidth}
    \centering\includegraphics[totalheight=.4\textheight, width=.9\textwidth]{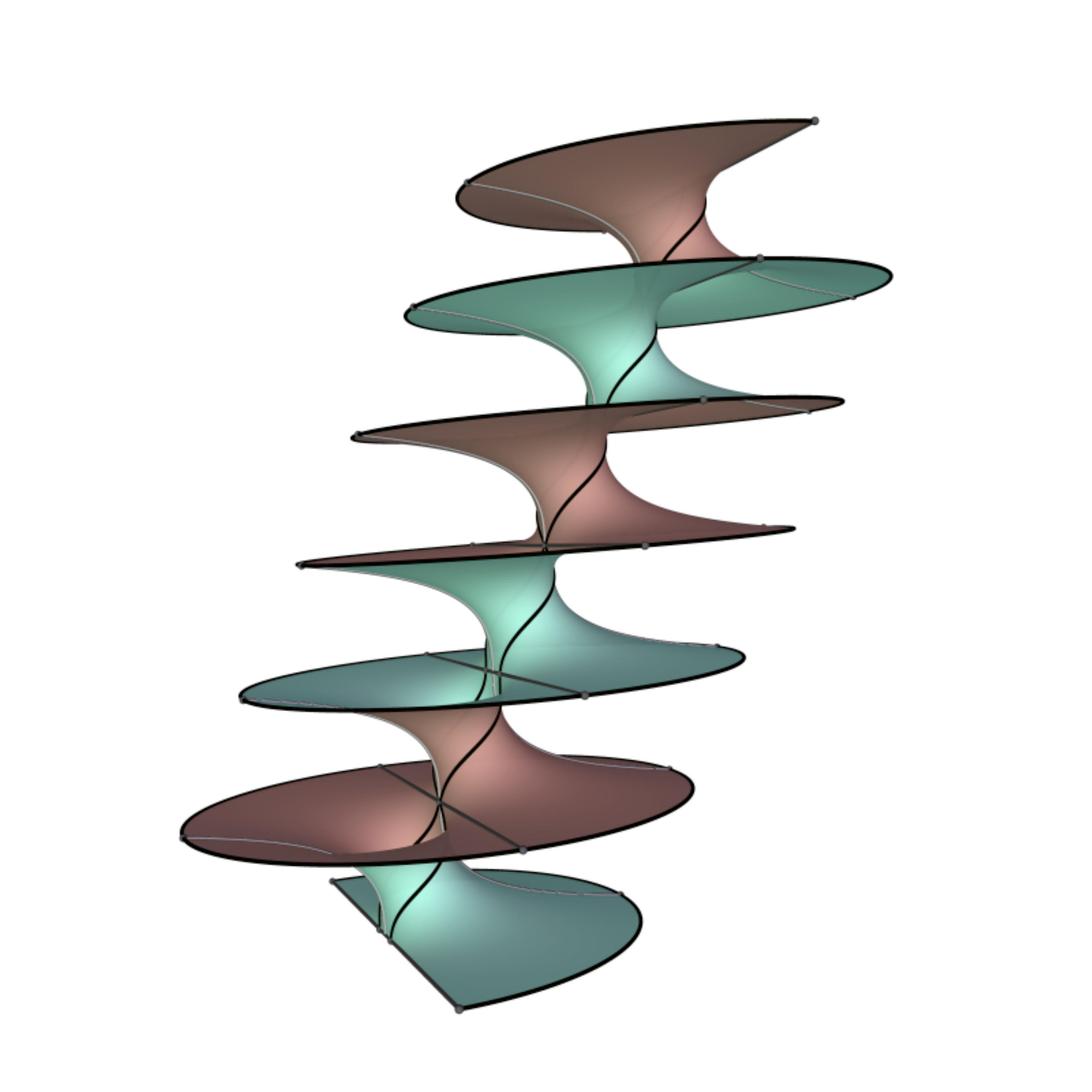}
        \end{minipage}
        \caption{Two of the Riemann examples.  The second one is starting to degenerate to helicoids.
        Credit: Matthias Weber, www.indiana.edu/~minimal.} 
   \end{figure}
Around 1860, 
  Riemann, \cite{Ri},  classified all minimal surfaces in $\RR^3$ that are foliated by circles and straight lines
  in horizontal planes.  He showed that the only such surfaces are the plane, the catenoid, the helicoid, and 
  a two-parameter family that is now known as the Riemann examples.\index{Riemann examples}  The surfaces that he
  discovered formed a family of complete embedded minimal surfaces that are singly-periodic and have genus zero.   Each of the surfaces has infinitely many parallel planar ends  connected by necks (``pairs of pants'').
 
Modulo rigid motions, this is a $2$ parameter family of minimal surfaces.   The parameters are:
\begin{itemize}
\item Neck size.
\item Angle between period vector and the ends.
\end{itemize}

If we keep the neck size fixed and allow the angle to become vertical (i.e., perpendicular to the planar ends), the family degenerates to a pair of oppositely oriented helicoids.  On the other hand, as the angle 
goes to zero, the family degenerates to a catenoid.

\subsubsection{The Genus One Helicoid}

In 1993, 
  Hoffman-Karcher-Wei gave numerical evidence for the existence of a complete embedded minimal surface with genus one that is asymptotic to a helicoid; they called it a ``genus one helicoid''.   In \cite{HoWW}, Hoffman, Weber and Wolf 
constructed such a surface as the limit of ``singly-periodic  genus one helicoids'', where each singly-periodic genus one helicoid was constructed via the Weierstrass representation.  Later, 
Hoffman and White constructed a genus one helicoid variationally in \cite{HoWh1}.

\begin{figure}[htbp]
    \begin{minipage}[t]{0.5\textwidth}
    \centering\includegraphics[totalheight=.3\textheight, width=1\textwidth]{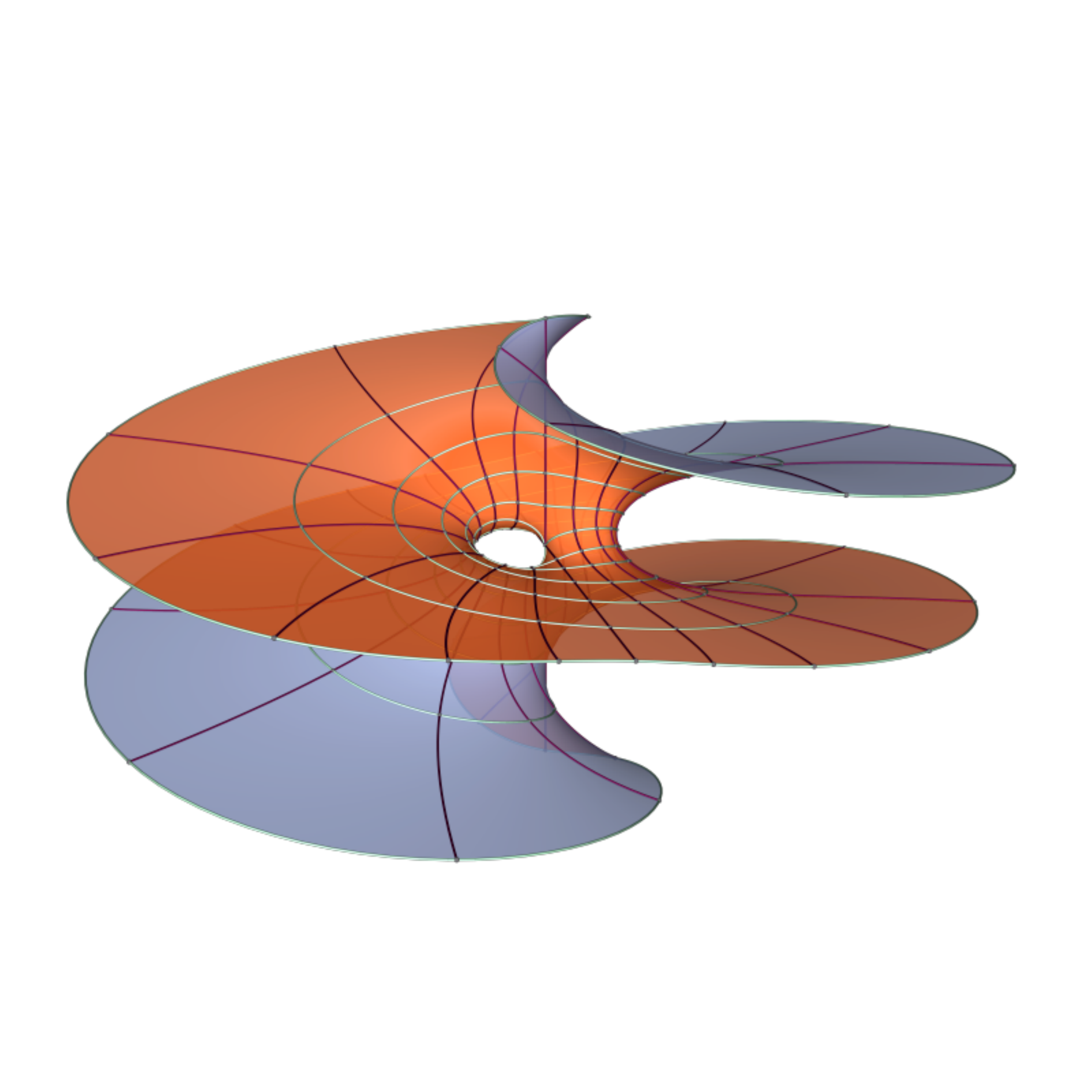}
    \caption{The genus one helicoid. }  
    \end{minipage}\begin{minipage}[t]{0.5\textwidth}
    \centering\includegraphics[totalheight=.3\textheight, width=1\textwidth]{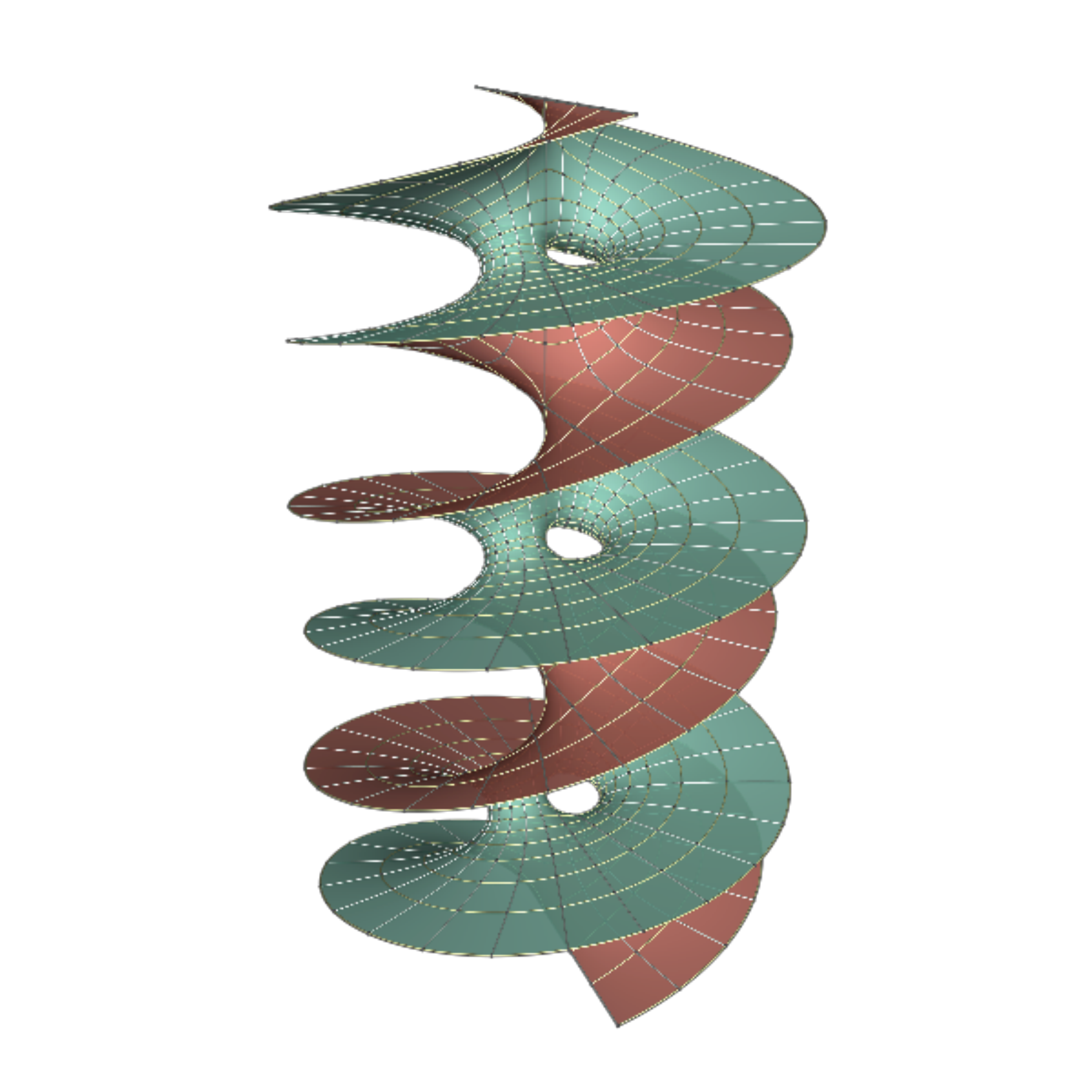}
    \caption{A periodic minimal surface asymptotic to the helicoid, whose fundamental domain has genus one.} 
    \end{minipage}
    Credit: Matthias Weber, www.indiana.edu/~minimal.
\end{figure}

\subsection{Second variation}

Minimal surfaces are critical points for the area functional, so it is natural to look at the second derivative of the area functional at a minimal surface.  This is called the second variation and it has played an important role in the subject since at least the work of Simons, \cite{Sim}, in 1968.

To make this precise,
let $\Sigma_0$ be a $2$-sided minimal hypersurface   in $\RR^{n+1}$, $\nn$ its unit normal, $f$ a function, and
define the normal variation
\begin{equation}
	\Sigma_s = \{ x +s \, f(x)\, \nn(x) \, | \, x \in \Sigma_0 \}  \notag \, .
\end{equation}
A calculation (see, e.g., \cite{CM14}) shows that
the second variation of area along this    one-parameter family of hypersurfaces 
is given by
\begin{equation}	\label{e:secvar1}
 \frac{d^2}{ds^2} \, \big|_{s=0} \, \, \Vol \, (\Sigma_s)  = \int_{\Sigma_0} |\nabla f|^2 - |A|^2 \, f^2  
  = - \int_{\Sigma_0} f \, (\Delta + |A|^2) \, f
\, , 
\end{equation}
where $A$ is the second fundamental form.

When $\Sigma_0$ is minimal in a Riemannian manifold $M$, the formula becomes
\begin{equation}
 \frac{d^2}{ds^2} \, \big|_{s=0} \, \, \Vol \, (\Sigma_s)  = \int_{\Sigma_0} |\nabla f|^2 - |A|^2 \, f^2 - \Ric_M (\nn , \nn) \, f^2   \, , 
 \notag
\end{equation}
where $\Ric_M$ is the Ricci curvature of $M$.

A minimal surface $\Sigma_0$ is {\emph{stable}} when it passes the second derivative test, i.e., when
\begin{equation}
 0 \leq \frac{d^2}{ds^2} \, \big|_{s=0} \, \, \Vol \, (\Sigma_s)  = - \int_{\Sigma_0} f \, (\Delta + |A|^2) \, f
\, , 
 \notag
\end{equation}
for {\emph{every}} compactly supported variation $\Sigma_s$.  Analytically, stability means that the  Jacobi operator $\Delta + |A|^2$ is non-negative.
 
There is a useful analytic criterion to determine stability:

\begin{Pro}	(Fischer-Colbrie and Schoen, \cite{FiSc})
A $2$-sided minimal hypersurface $\Sigma \subset \RR^3$ is stable
  if and only if there is a positive function $u$ with $\Delta \, u = - |A|^2 \, u$.
\end{Pro}

Since the normal part of a constant vector field automatically satisfies the Jacobi equation, we conclude that minimal graphs are stable.  The same argument implies that minimal multi-valued graphs 
are stable.{\footnote{A multi-valued graph is a surface that is locally a graph over a subset of the plane, but the projection down to the plane is not one to one.}}

Stability is a natural condition given the variational nature of minimal surfaces, but 
one of the reasons that stability is useful is the following curvature estimate of R. Schoen, \cite{Sc1}:

\begin{Thm}	\label{t:schoenstable}
(Schoen, \cite{Sc1})
If $\Sigma_0 \subset \RR^3$ is stable and $2$-sided and $\cB_R$ is a geodesic ball in $\Sigma_0$, then
\begin{equation}
	\sup_{\cB_{\frac{R}{2}}} \, \, |A| \leq \frac{C}{R} \, ,  \notag
\end{equation}
where $C$ is a fixed constant.
\end{Thm}

When $\Sigma_0$ is complete, we can let $R$ go to infinity and conclude that $\Sigma_0$ is a plane.  This  Bernstein theorem was proven independently by do Carmo-Peng, \cite{dCP}, and Fischer-Colbrie-Schoen, \cite{FiSc}.

 See \cite{CM15} for a different proof of Theorem \ref{t:schoenstable} and 
 a generalization    to surfaces that are stable for a parametric elliptic
  integrand.  The key point for getting the curvature estimate is to establish uniform area bounds just using stability:
  
  \begin{Thm}  (Colding-Minicozzi, \cite{CM15})    \label{c:fsapp}
If   $\Sigma^2 \subset \RR^3$ is stable and $2$-sided and 
$\cB_{r_0}$ is simply-connected, then
\begin{equation}
        \Area\, (\cB_{r_0}) \leq \frac{4\pi}{3} \, r_0^2 \, .
\end{equation}
\end{Thm}

 \vskip2mm
 The corresponding result is not known in higher dimensions, although Schoen, Simon and Yau proved  
 curvature estimates assuming an area bound   in low dimensions; see \cite{ScSiY}.   The counter-examples to the Bernstein problem in dimensions seven and up show that such a bound can only hold in low dimensions.
 However, 
R. Schoen has conjectured that the Bernstein theorem and curvature estimate  should be true also for stable hypersurfaces in $\RR^4$:

\begin{Con}	(Schoen)
If $\Sigma^3 \subset \RR^4$ is a complete immersed $2$-sided stable minimal hypersurface, then $\Sigma$ is flat.
\end{Con}

\begin{Con}	(Schoen)
If $\Sigma^3 \subset B_{r_0} = B_{r_0}(x) \subset M^4$ is an  immersed $2$-sided stable minimal hypersurface 
where $|K_M|\leq k^2$, $r_0<\rho_1 (\pi/k,k)$, and $\partial
\Sigma \subset \partial B_{r_0}$, then for some $C=C(k)$ and all
$0 < \sigma \leq r_0$,
\begin{equation}        \label{e:o11cs}
        \sup_{B_{r_0-\sigma}}  |A|^2
                \leq C \, \sigma^{-2}\, .
\end{equation}
\end{Con}

Any progress on these conjectures would be enormously important for the theory of minimal hypersurfaces in $\RR^4$.

\section{Classification of embedded minimal surfaces}

One of the most fundamental questions about minimal surfaces is to classify or describe the space of all complete embedded minimal surfaces in $\RR^3$.  We have already seen three results of this type:
\begin{enumerate}
\item Bernstein showed that a complete minimal graph must be a plane.
\item Catalan showed that a ruled minimal surface is either a plane or a helicoid.
\item A minimal surface of revolution must be a plane or a catenoid.
\end{enumerate}
Each of these theorems makes a rather strong hypothesis on the class of surfaces.  It would be more useful to have classifications under weaker hypotheses, such as just as the topological type of the surface.

The last decade has seen enormous progress on the classification of embedded minimal surfaces in $\RR^3$
by their topology.  The surfaces are generally divided into three cases,  according to the topology:

\begin{itemize}
\item Disks.
\item Planar domains - i.e., genus zero.
\item Positive genus.
\end{itemize}

The classification of complete surfaces has relied heavily upon breakthroughs on the local descriptions of pieces of embedded minimal surfaces with finite genus.  

\subsection{The topology of minimal surfaces}

The topology of a compact connected oriented surface without boundary is described by a single non-negative number: the genus.  The sphere has genus zero, the torus has genus one, and the connected sum of $k$-tori has genus $k$.  

The genus of an oriented surface with boundary is defined to be the genus of the compact surface that you get by gluing in a disk along each boundary component.  Since an annulus with two disks glued in becomes a sphere, the annulus has genus zero.
Thus, the topology of a connected oriented surface with boundary is described by two numbers: the genus and the number of boundary components.  

The last topological notion that we will need is {\emph{properness}}.  An immersed submanifold $\Sigma \subset M$ is proper when the intersection of $\Sigma$ with any compact set in $M$ is compact.  Clearly, every compact submanifold is automatically proper.

There are two important monotonicity properties for the topology of minimal surfaces in $\RR^3$; one does not use minimality and one does.  

\begin{Lem}
If $\Sigma$ has genus $k$ and  $\Sigma_0 \subset \Sigma$, then the genus of $\Sigma_0$ is at most $k$.
\end{Lem}

\begin{proof}
This follows immediately from the definition of the genus and does not use minimality.
\end{proof}

\begin{Lem}	\label{l:cvxhull}
If $\Sigma \subset \RR^3$ is a properly embedded minimal surface and 
$B_R (0) \cap \partial \Sigma = \emptyset$, then the inclusion of $B_R(0) \cap \Sigma$ into $\Sigma$ is an injection on the first homology group.
\end{Lem}

\begin{proof}
If not, then $B_R(0) \cap \Sigma$ contains a one-cycle $\gamma$ that does not bound a surface in 
$B_R(0) \cap \Sigma$ but does
bound a surface $\Gamma \subset \Sigma$.  However, $\Gamma$ is then a minimal surface that must leave $B_R(0)$ but with $\partial \Gamma \subset B_R(0)$, contradicting the convex hull property (Proposition \ref{p:cvx}).
\end{proof}

This has the following immediate corollary for disks:

\begin{Cor}
If $\Sigma$ is a properly embedded minimal disk, then each component of 
$B_R(0) \cap \Sigma$ is a disk.
\end{Cor}

\subsection{Multi-valued graphs}   \label{s:s1}

We will need the notion of a multi-valued graph
from \cite{CM6}--\cite{CM9}.  At a first approximation, a multi-valued graph is  locally a graph over a subset of the plane but the projection down is not one to one.  Thus, it shares many properties with minimal graphs, including stability, but includes new possibilities such as the helicoid minus the vertical axis.

To be precise, let
$D_r$ be the disk in the plane centered at the origin and of
radius $r$ and let $\cP$ be the universal cover of the punctured
plane $\CC\setminus \{0\}$ with global polar coordinates $(\rho,
\theta)$ so $\rho>0$ and $\theta\in \RR$. Given $0 \leq r\leq s$
and $\theta_1 \leq \theta_2$, define the ``rectangle''
$S_{r,s}^{\theta_1 , \theta_2} \subset \cP$ by
\begin{equation}
   S_{r,s}^{\theta_1 , \theta_2} = \{(\rho,\theta)\,|\,r\leq \rho\leq s\, ,\, \theta_1 \leq \theta\leq
\theta_2 \} \, .
\end{equation}
An $N$-valued graph of a function $u$ on the annulus $D_s\setminus
D_r$ is a single valued graph over  
\begin{equation}
   S_{r,s}^{-N\pi , N\pi} = \{(\rho,\theta)\,|\,r\leq \rho\leq s\, ,\, |\theta|\leq
N\,\pi\} \, .
\end{equation}
($\Sigma_{r,s}^{\theta_1 , \theta_2}$ will denote the subgraph of
$\Sigma$ over the smaller rectangle $S_{r,s}^{\theta_1 ,
\theta_2}$).
 The multi-valued graphs that we will
consider will never close up; in fact they will all be embedded.
Note that embedded corresponds to that the separation never
vanishes.  Here the separation $w$ is the difference in height
between consecutive sheets and is  therefore given by
\begin{equation}        \label{e:sepw}
w(\rho,\theta)=u(\rho,\theta+2\pi)-u(\rho,\theta)\, .
\end{equation}
In the case where $\Sigma$ is the helicoid [i.e., $\Sigma$ can be
parametrized by $(s\,\cos t,s\,\sin t,t)$ where $s,\,t\in \RR$],
then
\begin{equation}
    \Sigma\setminus x_3-\text{axis}=\Sigma_1\cup \Sigma_2 \, ,
\end{equation}
 where $\Sigma_1$, $\Sigma_2$ are $\infty$-valued graphs.
$\Sigma_1$ is the graph of the function $u_1(\rho,\theta)=\theta$
and $\Sigma_2$ is the graph of the function
$u_2(\rho,\theta)=\theta+\pi$. In either case the separation
$w=2\,\pi$.

\begin{figure}[htbp]
    \includegraphics[totalheight=.55\textheight, width=.85\textwidth]{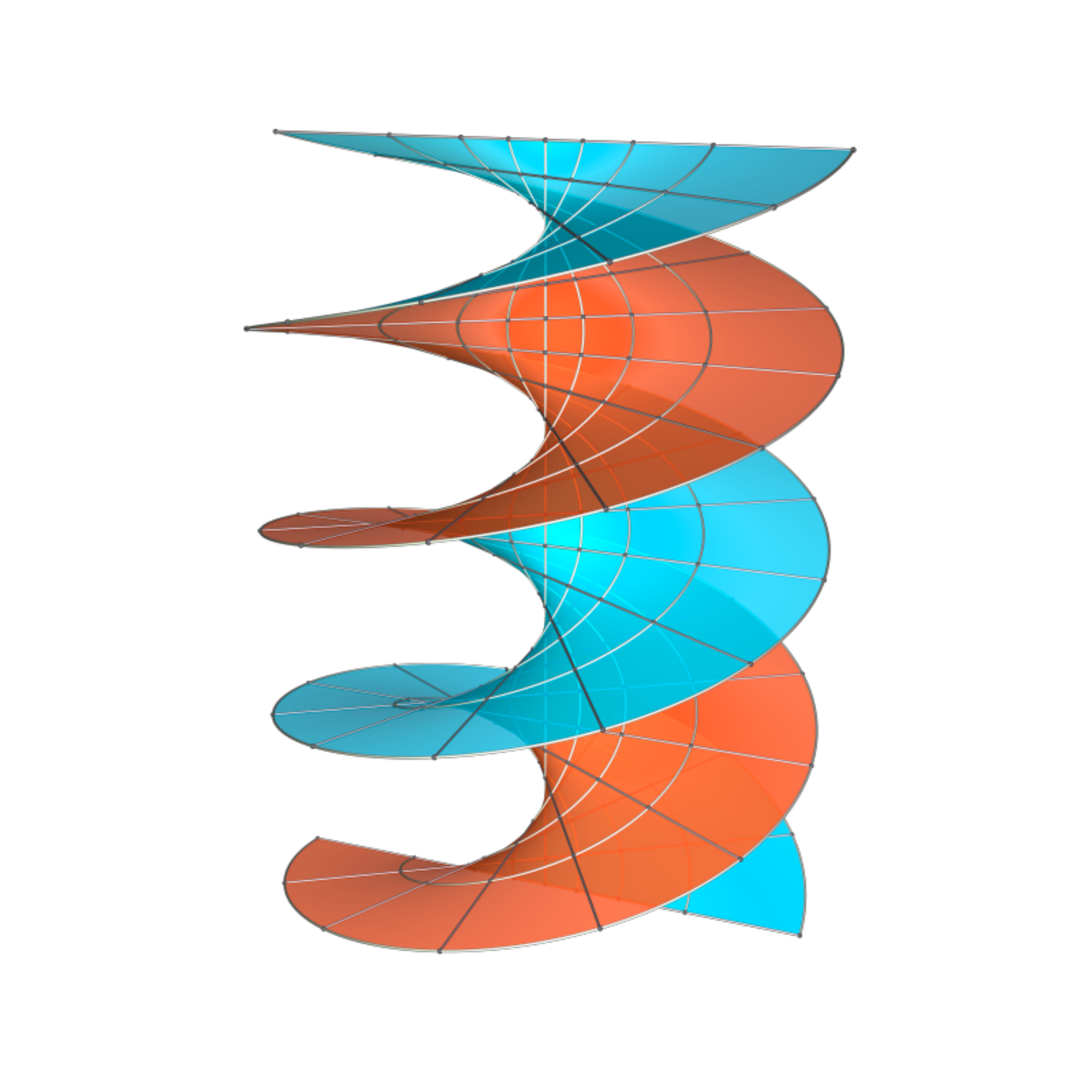}
    \caption{Multi--valued graphs.  The helicoid is obtained
    by gluing together two $\infty$--valued graphs along a line. 
    Credit: Matthias Weber, www.indiana.edu/~minimal.}   
\end{figure}

Note that for an embedded multi--valued graph, the sign of $w$
determines whether the multi--valued graph spirals in a
left--handed or right--handed manner, in other words, whether
upwards motion corresponds to turning in a clockwise direction or
in a counterclockwise direction.

\begin{figure}[htbp] \label{f:fsep}
    \centering\includegraphics[totalheight=.55\textheight, width=.95\textwidth]{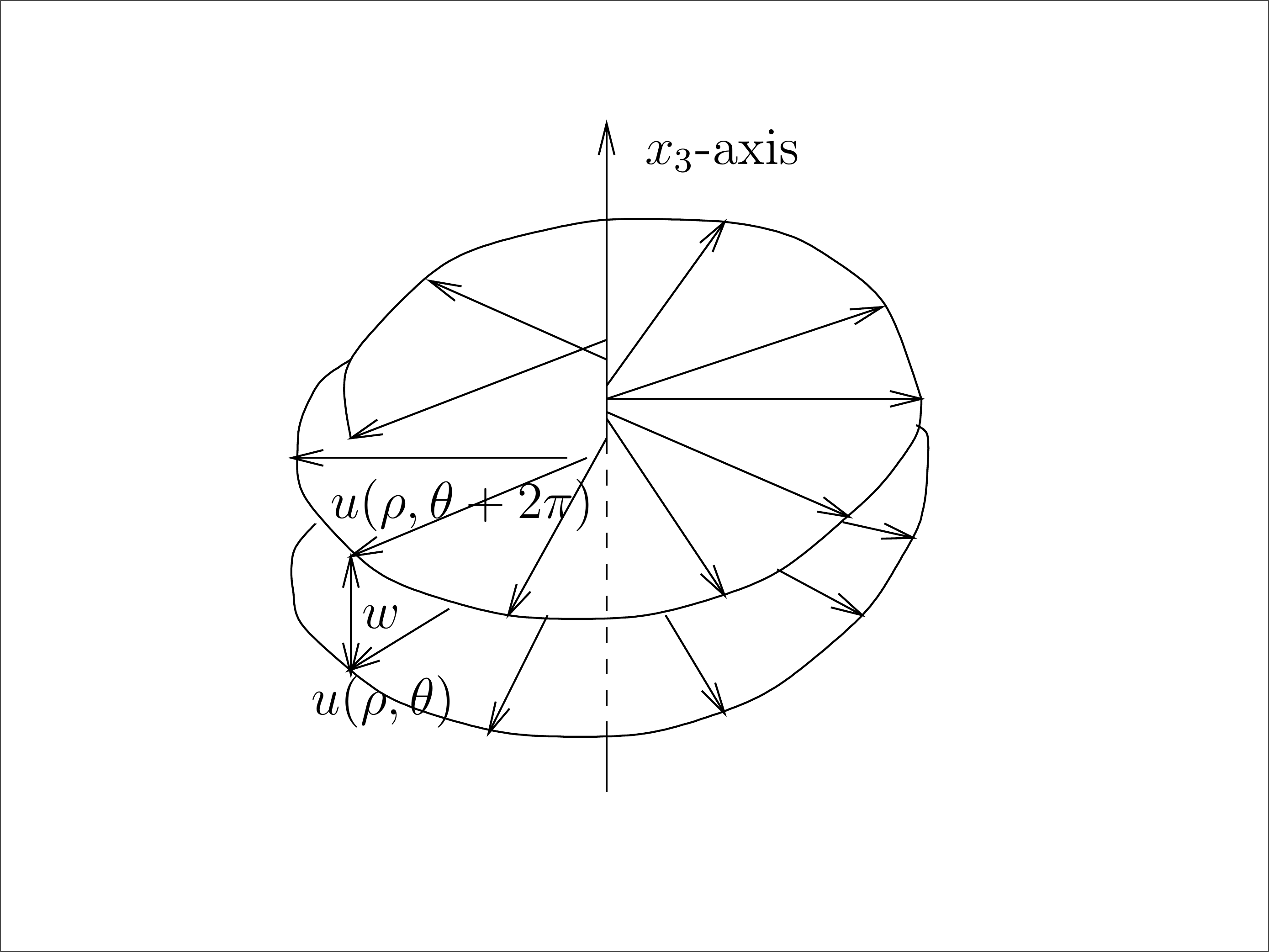}
    \caption{The separation $w$
for a multi--valued minimal graph.} 
\end{figure}

\subsection{Disks are double spiral staircases or graphs}

We will describe first the local classification of properly embedded minimal disks that follows from \cite{CM6}--\cite{CM9}.  This turns out to be the key step for understanding embedded minimal surfaces with finite genus since any of these can be decomposed into pieces that are either disks or pairs of pants.

There are two classical models for embedded minimal disks.  The first is a minimal graph over a simply-connected domain in $\RR^2$ (such as the plane itself), while the second is a double spiral staircase like the helicoid.  A double spiral staircase consists of two staircases
that spiral around one another so that two people
can pass each other without meeting.

  In \cite{CM6}--\cite{CM9} we showed that these are the only possibilities and, in fact, every
 embedded
minimal disk  is either
a minimal graph or can be approximated by a piece
of a rescaled helicoid.
It is graph when the curvature is small and is part of a helicoid when the curvature is above a certain threshold.

\begin{figure}[htbp]
\centering\includegraphics[totalheight=.55\textheight, width=.85\textwidth]{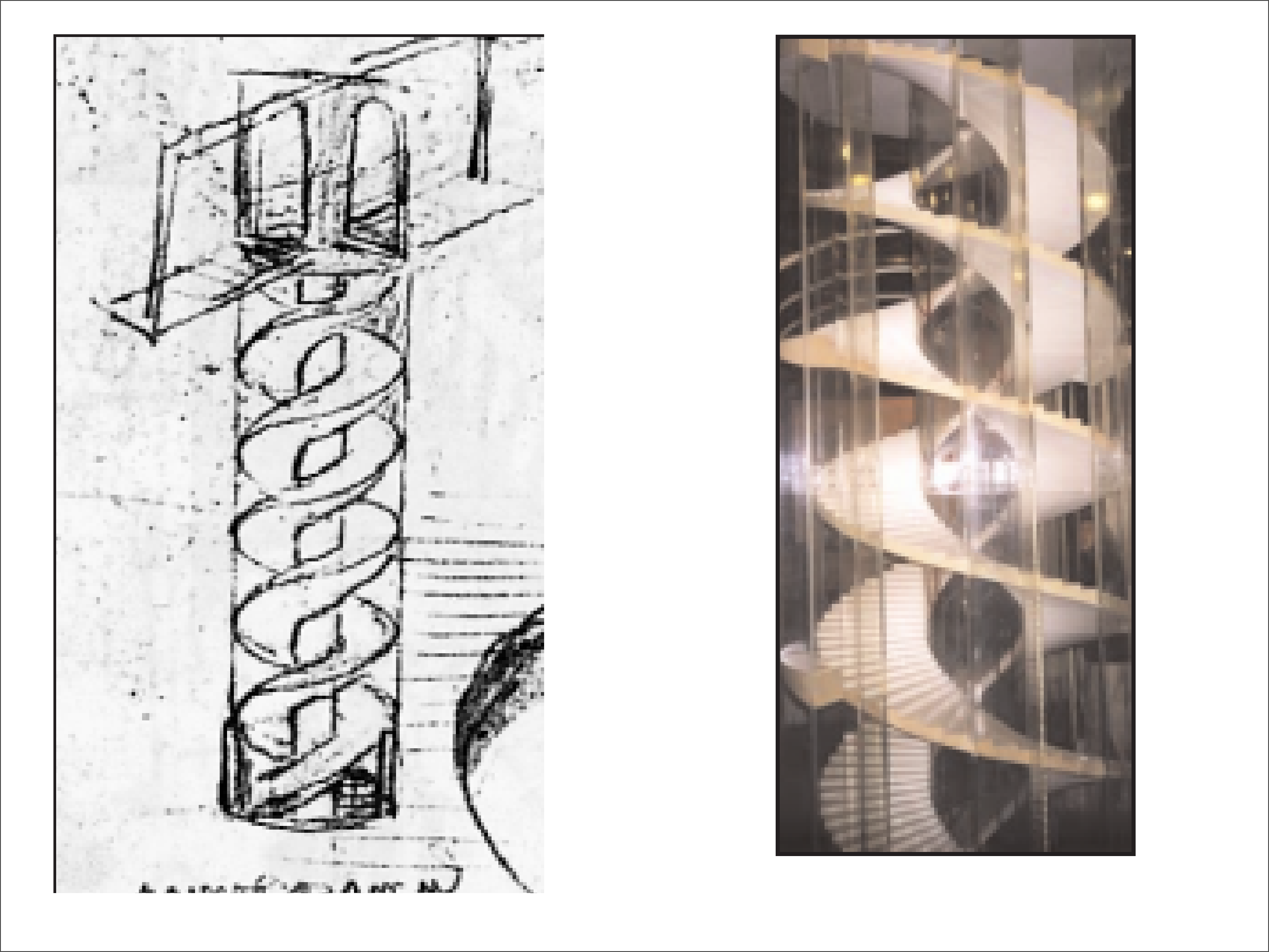}
\caption{Left: Drawing by Leonardo Da Vinci of a double spiral staircase from around 1490.
Right: Model of Da Vinci's double spiral staircase in Ch\^ateau de Chambord. Both figures reprinted from 
\cite{CM20}.}
\end{figure}

The main point in the proof is to show the double spiral staircase structure when the curvature is large.  The proof of this is long, but can be split into three main steps.

\centerline{\bf{Three main steps.}}

 \vskip1.5mm \noindent A. Fix an
integer $N$ (the ``large'' of the curvature in what follows will
depend on $N$). If an embedded minimal disk $\Sigma$ is not a
graph (or equivalently if the curvature is large at some point),
then it contains an $N$-valued minimal graph which initially is
shown to exist on the scale of $1/\max |A|$.  That is, the
$N$-valued graph is initially shown to be defined on an annulus
with both inner and outer radius inversely proportional to $\max
|A|$. 

\vskip1.5mm \noindent B. Such a potentially small $N$-valued
graph sitting inside $\Sigma$ can then be seen to extend as an
$N$-valued graph inside $\Sigma$ almost all the way to the
boundary.  That is, the small $N$-valued graph can be extended to
an $N$-valued graph defined on an annulus where the outer radius
of the annulus is proportional to $R$. Here $R$ is the radius of
the ball in $\RR^3$ that the boundary of $\Sigma$ is contained in.

\vskip1.5mm \noindent C.  The $N$-valued graph not only extends
horizontally (i.e., tangent to the initial sheets) but also
vertically (i.e., transversally to the sheets). That is, once
there are $N$ sheets there are many more and, in fact, the disk
$\Sigma$ consists of two multi-valued graphs glued together along
an axis.

 \begin{figure}[htbp]
\centering\includegraphics[totalheight=.4\textheight, width=.9\textwidth]{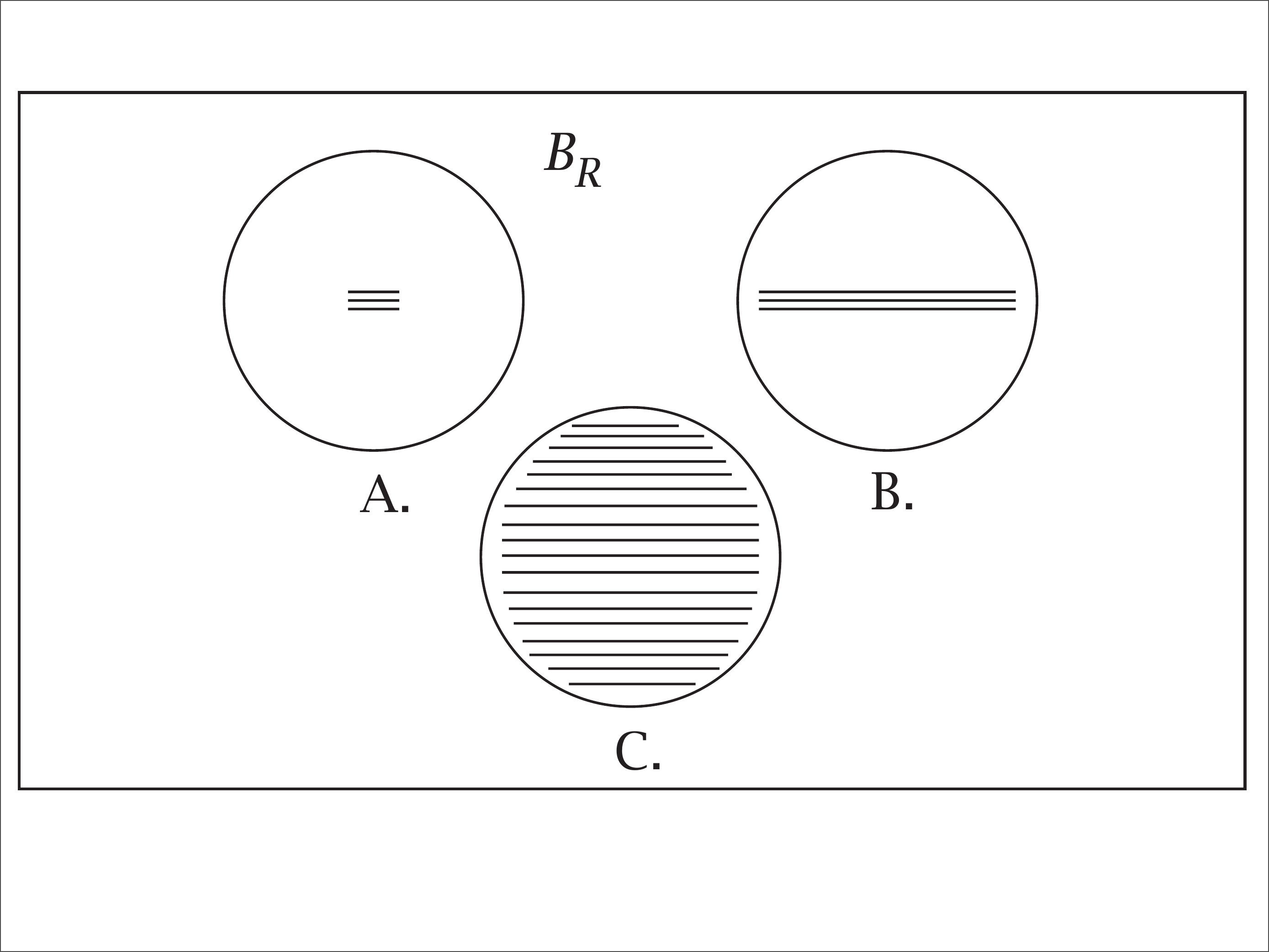}
\caption{Three main steps:  A. Finding a small $N$-valued graph in $\Sigma$.
B. Extending it in $\Sigma$ to a large $N$-valued graph.
C. Extend the number of sheets.}   
  \end{figure}

This general structure result for embedded minimal disks, and the methods used in its proof, give a compactness theorem for sequences of embedded minimal disks.  
This theorem is modelled on rescalings of the
helicoid and the precise statement is as follows (we state the
version for extrinsic balls; it was extended to intrinsic balls
in \cite{CM12}):

\begin{Thm} \label{t:t0.1}
(Theorem 0.1 in \cite{CM9}.)
 Let $\Sigma_i \subset B_{R_i}=B_{R_i}(0)\subset \RR^3$
be a sequence of embedded minimal {\underline{disks}} with
$\partial \Sigma_i\subset \partial B_{R_i}$ where $R_i\to \infty$.
If \begin{equation}
    \sup_{B_1\cap \Sigma_i}|A|^2\to \infty \, ,
    \end{equation}
 then there exists a subsequence, $\Sigma_j$, and
a Lipschitz curve $\cS:\RR\to \RR^3$ such that after a rotation of
$\RR^3$:
\begin{enumerate}
\item[\underline{1.}] $x_3(\cS(t))=t$.  (That is, $\cS$ is a graph
over the $x_3$-axis.)
\item[\underline{2.}]  Each $\Sigma_j$
consists of exactly two multi-valued graphs away from $\cS$ (which
spiral together).
\item[\underline{3.}] For each $1>\alpha>0$,
$\Sigma_j\setminus \cS$ converges in the $C^{\alpha}$-topology to
the foliation, $\cF=\{x_3=t\}_t$, of $\RR^3$.
\item[\underline{4.}] $\sup_{B_{r}(\cS (t))\cap
\Sigma_j}|A|^2\to\infty$ for all $r>0$, $t\in \RR$.  (The
curvatures blow up along $\cS$.)
\end{enumerate}
\end{Thm}

This theorem is sometimes referred to as {\emph{the lamination theorem}}.
Meeks  showed in \cite{Me2} that the Lipschitz curve $\cS$ is in fact a straight line perpendicular to the foliation.

\vskip2mm
The assumption that the radii $R_i$ go to infinity is used in several ways in the proof.  This guarantees that the leaves are planes (this uses the Bernstein theorem), but it also is used to show that the singularities are removable.  We will see in the next subsection that this is not always the case in the ``local case'' where the $R_i$ remain bounded.

\subsection{The local case}

In contrast to the global case of the previous subsection, there are local examples of sequence of minimal surfaces that do not converge to a foliation.  The first such example was constructed  in 
\cite{CM17}, where we constructed
 a sequence of embedded minimal disks in a unit ball in $\RR^3$ so that:
\begin{itemize}
\item Each contains the $x_3$-axis.
\item Each  is given by two multi-valued graphs over  $\{ x_3 =0 \} \setminus \{ 0 \}$.
\item The graphs spiral faster and faster near $\{ x_3 = 0 \}$.
\end{itemize}

The precise statement is:

\begin{Thm} (Colding-Minicozzi, \cite{CM17})  \label{t:tams}
There is  a sequence  of compact embedded minimal disks $0 \in
\Sigma_i \subset  B_1 \subset \RR^3$
 with $\partial \Sigma_i \subset
\partial B_1$ and
containing the vertical segment
$\{ (0,0,t) \, | \, |t|<1 \} \subset \Sigma_i$ so:
\begin{enumerate}
\item[(1)] $\lim_{i\to \infty} |A_{\Sigma_i}|^2 (0) = \infty$. 
\label{i:1}
\item[(2)]
$\sup_i \sup_{\Sigma_i \setminus B_{\delta} } |A_{\Sigma_i}|^2 < \infty$
for all $\delta > 0$.
\label{i:2}
\item[(3)]
$\Sigma_i \setminus \{ {\text{$x_3$-axis}} \} =
\Sigma_{1,i} \cup \Sigma_{2,i}$ for multi-valued graphs
$\Sigma_{1,i}$ and
$\Sigma_{2,i}$.
\label{i:3}
\item[(4)] 
$\Sigma_i  \setminus \{ x_3 = 0 \}$  converges to     two
embedded minimal disks $\Sigma^{\pm} \subset \{ \pm x_3 > 0 \}$
with $\overline{\Sigma^{\pm}} \setminus \Sigma^{\pm} = B_1 \cap \{
x_3 = 0\}$. Moreover, $\Sigma^{\pm} \setminus \{
{\text{$x_3$-axis}} \} = \Sigma_1^{\pm} \cup \Sigma_2^{\pm}$ for
multi-valued graphs $\Sigma_{1}^{\pm}$ and $\Sigma_{2}^{\pm}$ each of
which spirals into $\{ x_3 =
0 \}$; see fig. \ref{f:f2tams}.
\label{i:4}
\end{enumerate}
\end{Thm}

It follows from (4)
that $\Sigma_i \setminus \{ 0 \}$ converges to a lamination
of $B_1 \setminus \{ 0 \}$    (with leaves $\Sigma^-$, $\Sigma^+$,
and $B_1 \cap \{ x_3 = 0 \} \setminus \{ 0 \}$) which
 does not extend to a lamination of $B_1$.
Namely, $0$ is not a removable singularity.

\begin{figure}[htbp] 
    \centering\includegraphics[totalheight=.35\textheight, width=.8\textwidth]{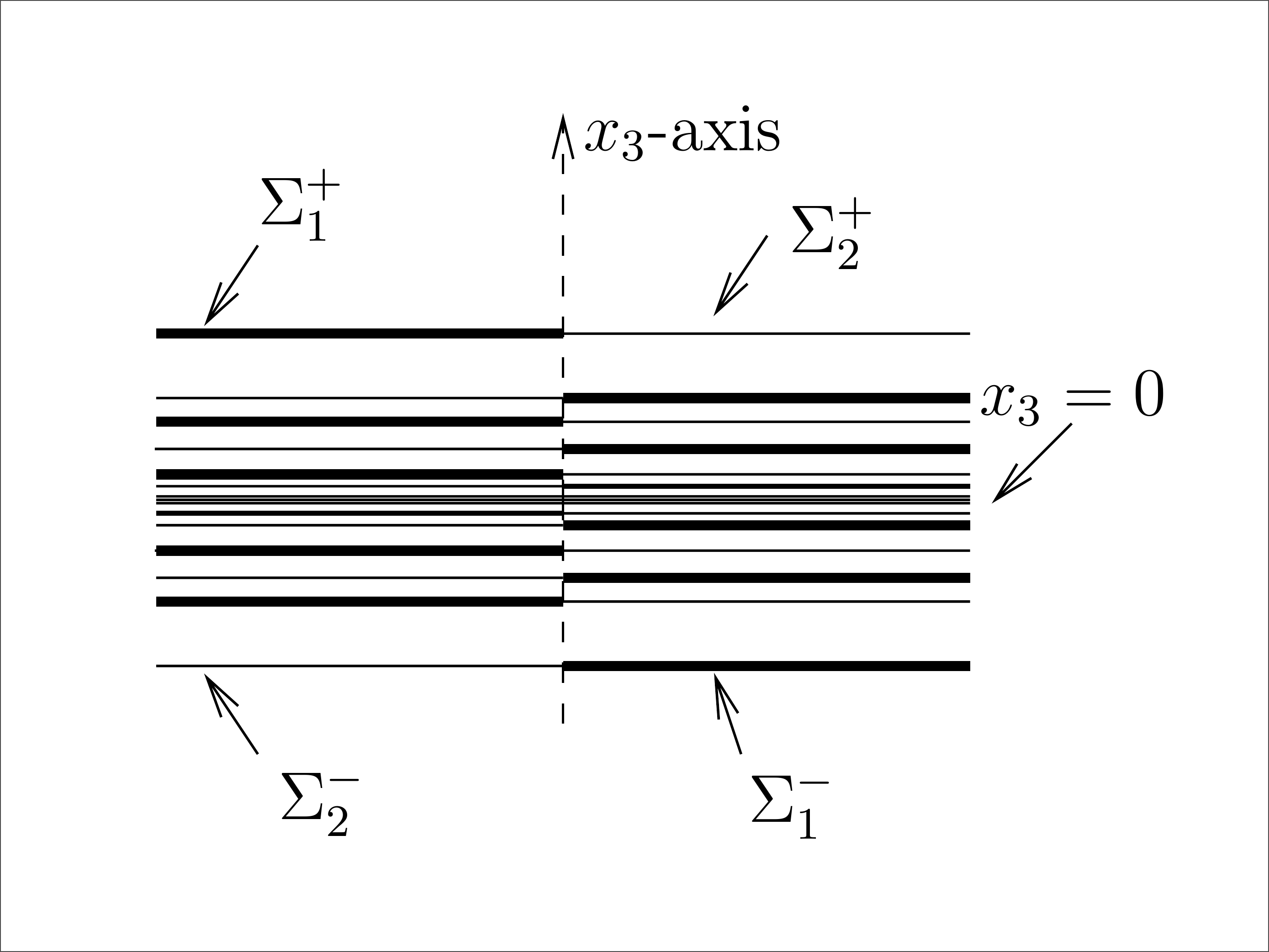}
    \caption{A schematic picture of the examples in Theorem \ref{t:tams}.  Removing the $x_3$-axis
    disconnects the surface into two multi-valued graphs; one of these is in bold.}  \label{f:f2tams}
\end{figure}

 There are now a number of 
 other interesting local examples of singular limit laminations:
 \begin{itemize}
 \item Meeks-Weber,  \cite{MeWe}: The Meeks-Weber bent helicoids are sequences of minimal surfaces defined in a tub about a curve that converge to a minimal foliation of the tube except along the curve itself and this curve is the singular set.
 \item Dean, \cite{De}: Dean generalized the construction of \cite{CM17} to get isolated singular points.
 \item Khan, \cite{Kh}: Khan used the Weierstrass representation to  construct
 sequences of spiraling multi-valued graphs where curvature blows up on half an interval.
 \item Hoffman-White, \cite{HoWh2}: They proved the definitive existence result for subsets of an axis by getting   an arbitrary closed subset as the singular set.  Their proof is variational, seizing on the fact that half of the helicoid is area-minimizing (and then using reflection to construct the other half).
 \item Kleene, \cite{Kl}: Gave different proof of Hoffman-White using the Weierstrass representation in the spirit of \cite{CM17}, \cite{De} and \cite{Kh}.
 \item Calle-Lee, \cite{CaL}: Constructed local helicoids in Riemannian manifolds.
 \end{itemize}

One of the most interesting questions is when does a minimal lamination have removable singularities? This is most interesting when for minimal limit laminations that arise from sequences of embedded minimal surfaces.   It is clear from these examples that this is a global question.  This question really has two separate cases depending on the topology of the leaves near the singularity.  When the leaves are simply connected, the only possibility is the spiraling and multi-valued graph structure proven in \cite{CM6}--\cite{CM9};
see \cite{CM18} for 
a flux argument to get removability in the global case. 
A different type of singularity occurs when the injectivity radius of the leaves goes to zero at a singularity; examples of this were constructed by Colding and De Lellis in \cite{CD}.
 The paper \cite{CM10} has a similar flux argument for the global case where the leaves are not simply-connected.

\subsection{The one-sided curvature estimate}

One of the key tools used to understand embedded 
minimal surfaces is the one-sided curvature estimate proven in \cite{CM9} using the structure theory developed in  \cite{CM6}--\cite{CM9}.  

The one-sided curvature estimate 
roughly states that an embedded minimal disk that lies on one-side of a plane, but comes close to the plane, has bounded curvature.  Alternatively, it says that if the curvature is large at the center of a ball, then the minimal disk propagates out in all directions so that it cannot be contained on one side of any plane that passes near the center of the ball.\index{one-sided curvature estimate}

\begin{Thm}	(Colding-Minicozzi, \cite{CM9}) \label{t:onesided}
There exists $\epsilon_0 > 0$ so that the following holds.
Let $y \in \RR^3$, $r_0 > 0 $ and
\begin{equation}    \label{e:two}
    \Sigma^2 \subset B_{2r_0}(y) \cap \{ x_3 > x_3(y) \}
        \subset \RR^3
\end{equation}
 be a compact embedded  minimal
disk with $\partial \Sigma \subset \partial  B_{2 \, r_0}(y)$.
 For any connected component $\Sigma'$ of $B_{ r_0}(y) \cap \Sigma$
with $B_{\epsilon_0 \, r_0 }(y) \cap \Sigma' \ne \emptyset$,
\begin{equation}    \label{e:osr}
    \sup_{\Sigma'}   |A_{\Sigma'}|^2
        \leq  r_0^{-2}  \, .
\end{equation}
\end{Thm}

 \begin{figure}[htbp]
 \centering\includegraphics[totalheight=.35\textheight, width=.75\textwidth]{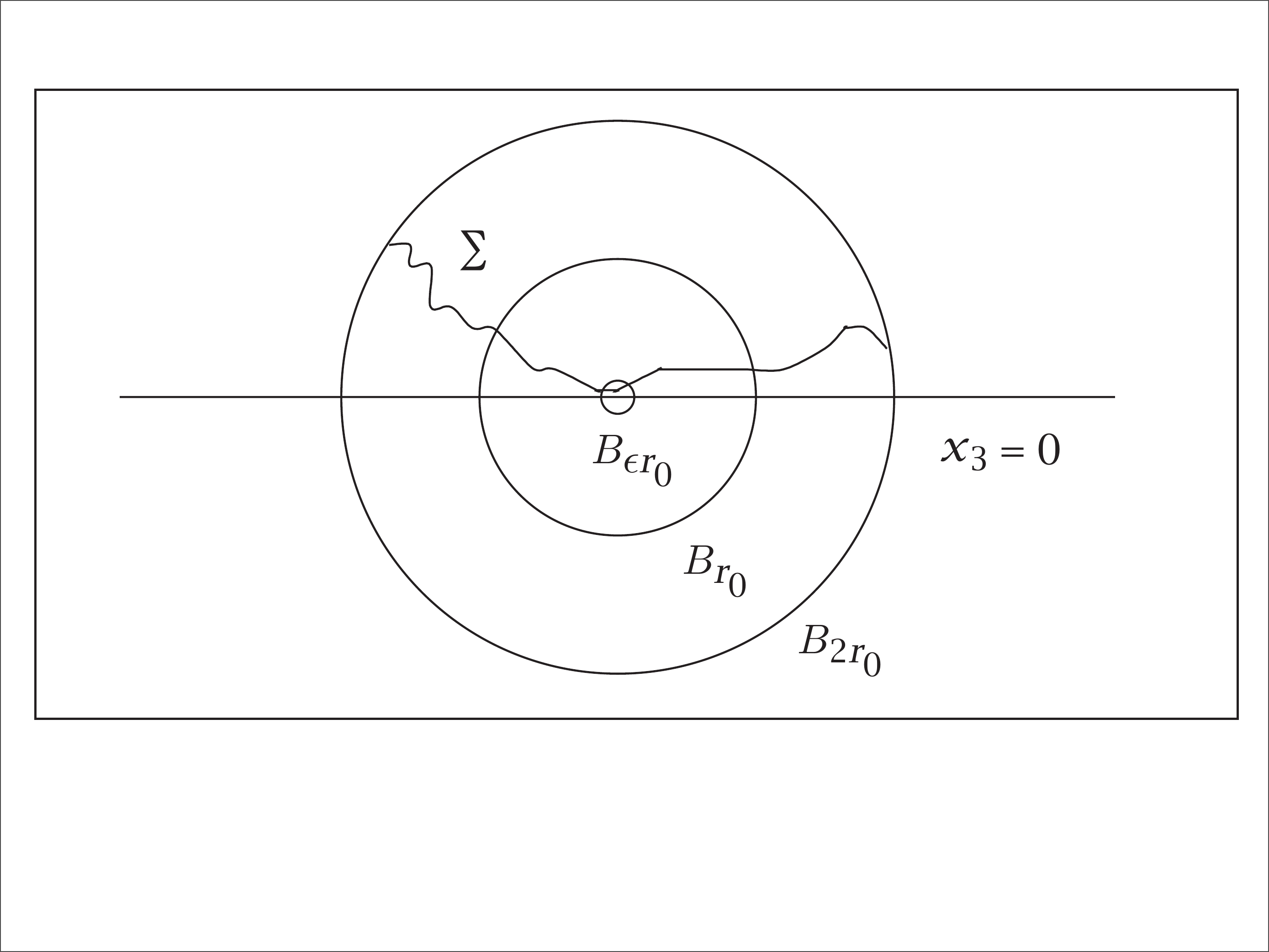}
    \caption{The one-sided curvature estimate.} 
\end{figure}

The example of a rescaled catenoid shows that simply-connected and embedded are both essential hypotheses for the one-sided curvature estimate.   More precisely, the height of the catenoid grows logarithmically in the distance to the axis of rotation.  In particular, the intersection of the catenoid with $B_{r_0}$ lies in a slab of thickness $\approx \log r_0$ and the ratio of 
\begin{equation}
	\frac{\log r_0}{r_0} \to 0 {\text{ as }} r_0 \to \infty \, .
\end{equation}
Thus, after dilating the catenoid by $\frac{1}{j}$,  we get a sequence of minimal surfaces in the unit ball that converges as sets to $\{ x_3 = 0\}$ as $j \to \infty$.  
  However, the catenoid is not flat, so these rescaled catenoids have $|A| \to \infty$ and \eqr{e:osr} does not apply for $j$ large.
  
   \begin{figure}[htbp]
\centering\includegraphics[totalheight=.35\textheight, width=.75\textwidth]{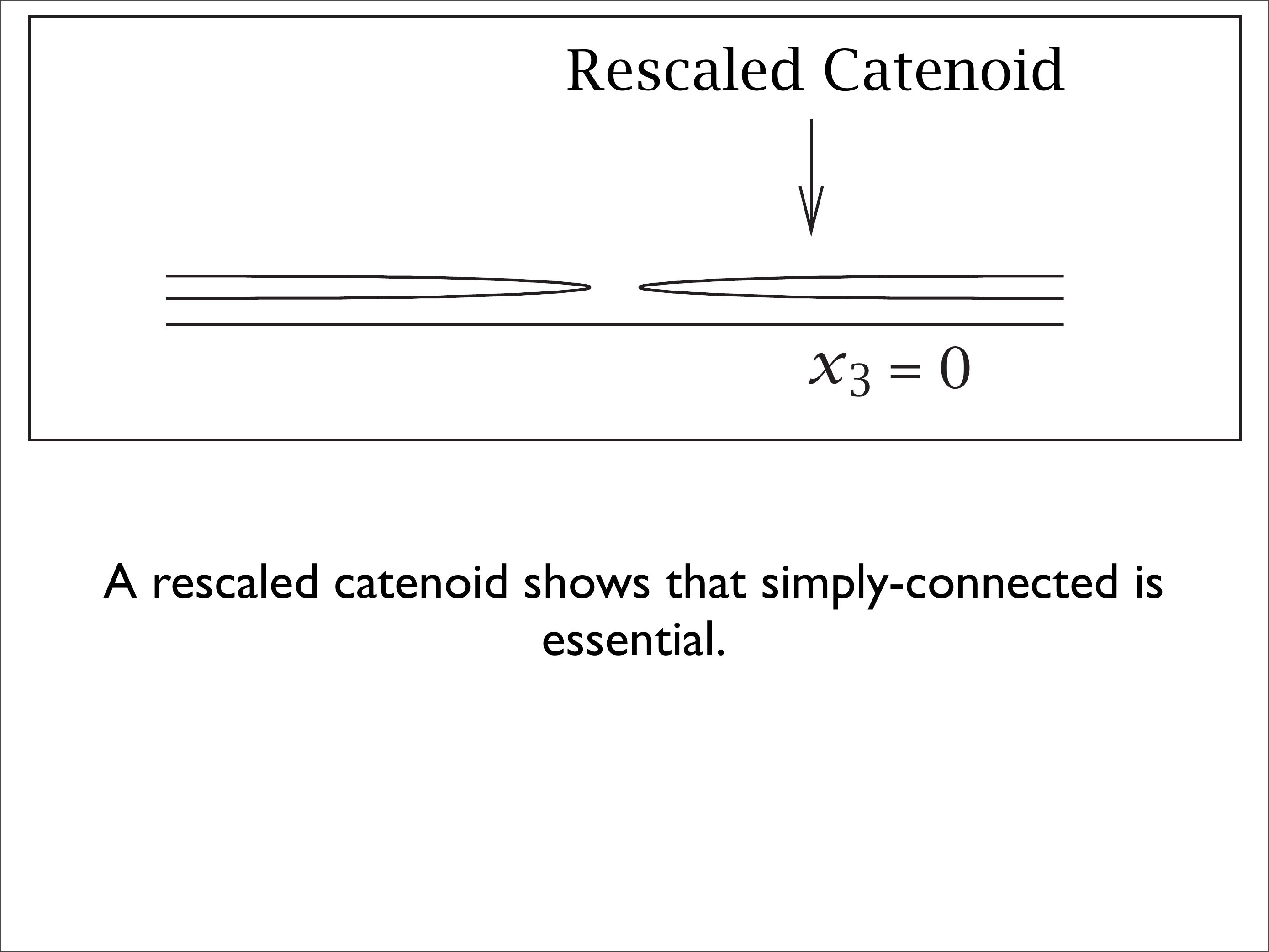}
    \caption{Rescaled catenoids.} 
\end{figure}

\subsection{Genus zero have pair of pants decomposition}

Thus far, we have concentrated on the case of embedded minimal disks.  We turn next to the general case of embedded minimal planar domains (i.e., where the surfaces have genus zero but may not be simply connected).  The new possibilities are illustrated by Riemann's family of examples where small necks connect large graphical regions that are asymptotic to planes.  Cutting along these small necks, one can decompose the Riemann examples into ``pairs of pants'' that are topologically disks with two sub-disks removed (think of the outer boundary as the waist and the two inner boundaries corresponding to the legs).

One of the main results from \cite{CM10} is that a general embedded minimal planar domain has a similar 
 pair of pants decomposition:

\begin{Thm}  
	\label{t:cm7c}
 Any nonsimply connected embedded minimal planar
domain with a small neck can be cut along a collection of short curves.
After the cutting, we are left with graphical pieces that are defined over
a disk with either one or two subdisks removed (a topological disk with
two subdisks removed is called a pair of pants).\index{pair of pants} Moreover, if for some
point the curvature is large, then all of the necks are very small.
\end{Thm}

 \begin{figure}[htbp]
\centering\includegraphics[totalheight=.4\textheight, width=1\textwidth]{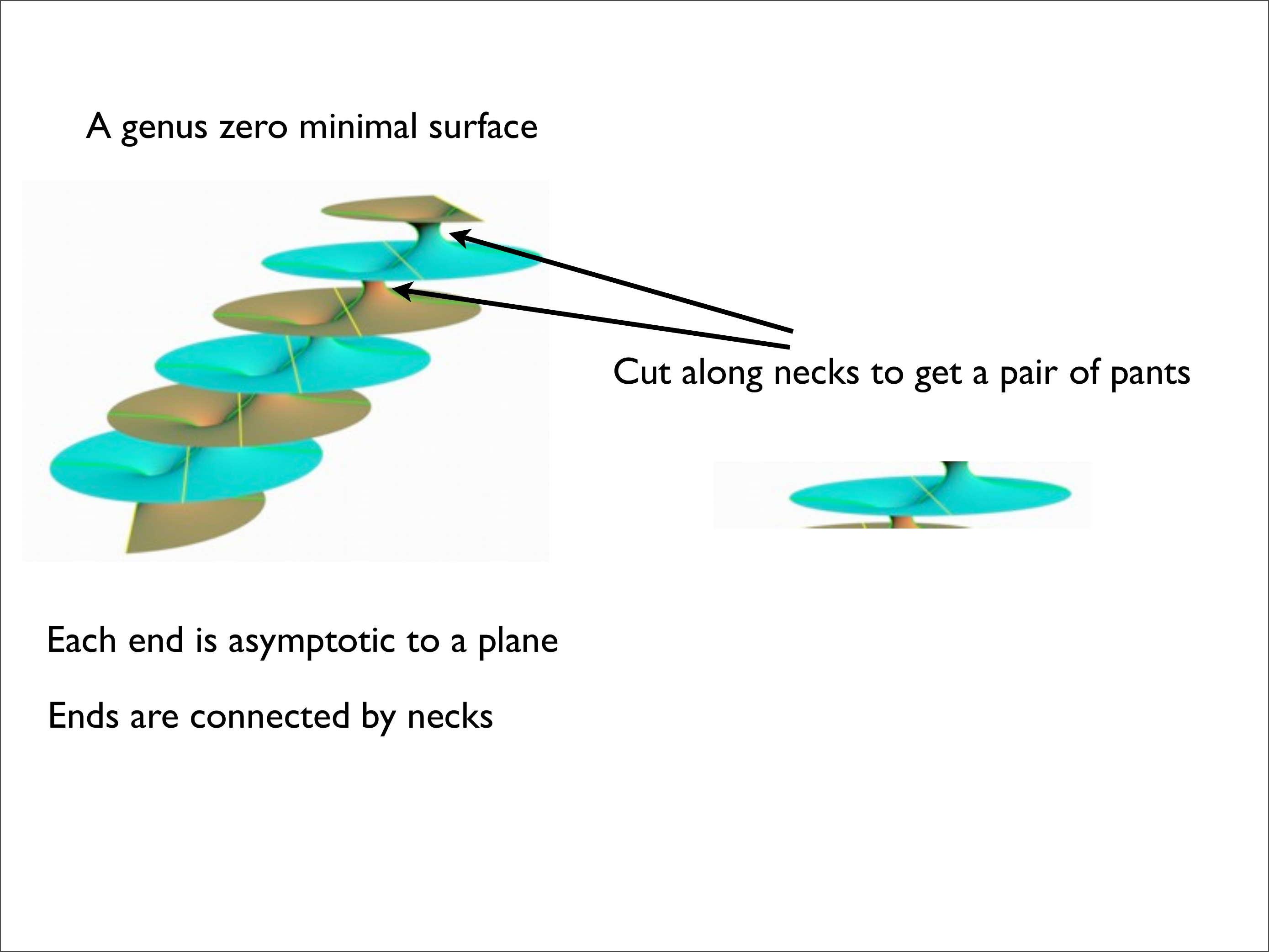}
\caption{The pair of pants decomposition. Credit: Matthias Weber, www.indiana.edu/~minimal.}   
  \end{figure}

The following compactness result is a consequence:

\begin{Cor}	(Colding-Minicozzi, \cite{CM10})	\label{t:cm7d}
 A sequence of embedded minimal planar domains that are
not ULSC, but with curvatures blowing up, has a subsequence that
converges to a collection of flat parallel planes.
\end{Cor}

In the next section, we will turn to finer structure and compactness theorems for sequences of planar domains.

\subsection{Compactness theorems for planar domains}

In order to describe the results for sequence of planar domains, it will be useful to 
divide things into two cases depending on whether or not the topology is concentrating a points.
  To distinguish between these cases, we
will say that a sequence of surfaces $\Sigma_i^2\subset \RR^3$ is
{\it{uniformly locally simply connected}} (or ULSC)  if for each
compact subset $K$ of $\RR^3$, there exists a constant $r_0 > 0$
(depending on $K$) so that for every $x \in K$, all $r \leq
r_0$, and every surface
$\Sigma_i$
\begin{equation}    \label{e:ulsc2}
 {\text{each connected component of }} B_{r}(x) \cap \Sigma_i {\text{ is
 a disk.}}
\end{equation}
  For instance, a sequence of rescaled catenoids
   where the necks shrink to zero is not ULSC, whereas a
sequence of rescaled helicoids is.

Another way of locally distinguishing sequences where the topology
does not concentrate from sequences where it does comes from
analyzing the singular set. The singular set $\cS$ is defined to
be the set of points where the curvature is blowing up.  That is,
a point $y$ in $\RR^3$ is in $\cS$ for a sequence $\Sigma_i$ if
\begin{equation}
    \sup_{B_{r}(y)\cap
\Sigma_i}|A|^2\to\infty {\text{ as $i \to \infty$ for all $r>0$}}
.
\end{equation}
We will show that for embedded minimal surfaces $\cS$ consists of
two types of points. The first type is roughly modelled on
rescaled helicoids
 and the second on
rescaled catenoids:
\begin{itemize}
\item
A point $y$ in $\RR^3$ is in $\cSu$ if the curvature for the
sequence $\Sigma_i$ blows up at $y$ and the sequence is ULSC in a
neighborhood of $y$.
\item
A point $y$ in $\RR^3$ is in $\cSt$ if the sequence is not ULSC in
any neighborhood of $y$. In this case, a sequence of closed
non-contractible curves $\gamma_i \subset \Sigma_i$ converges to
$y$.
\end{itemize}
The sets $\cSt$ and $\cSu$ are obviously disjoint and
  the curvature blows up at both, so   $\cSt \cup \cSu \subset \cS$.   An easy
argument will later show that, after passing to a subsequence, we
can assume that
\begin{equation}    \label{e:}
    \cS=  \cSt \cup \cSu \, .
\end{equation}
Note that $\cSt = \emptyset$ is equivalent to that the sequence is
ULSC as is the case for sequences of rescaled helicoids.  On the
other hand, $\cSu = \emptyset$ for sequences of rescaled
catenoids.   

We will show that every sequence $\Sigma_i$ has a subsequence that
is either    ULSC or for which $\cSu$ is empty. This is the
 next ``no mixing'' theorem.  We will see later that these two different cases give
 two very different structures.

\begin{Thm}   (Colding-Minicozzi, \cite{CM10})	  \label{c:main}
If $\Sigma_i \subset B_{R_i}=B_{R_i}(0)\subset \RR^3$ is a
sequence
 of compact embedded minimal planar domains
  with $\partial
\Sigma_i\subset
\partial B_{R_i}$ where $R_i\to \infty$, then there is a subsequence with either
$\cSu = \emptyset$ or $\cSt = \emptyset$.
\end{Thm}

In view of Theorem \ref{c:main} and the earlier results for disks,
it is natural to first analyze sequences that are ULSC, so where
$\cSt=\emptyset$, and second analyze sequences where $\cSu$ is
empty.  We will do this next.

Common for both the ULSC case and the case where $\cSu$ is empty
is that the limits are always  laminations by flat parallel planes
and  the singular sets are always  closed subsets contained in the
union of the planes.  This is the content of the next theorem:

\begin{Thm} (Colding-Minicozzi, \cite{CM10})	  \label{t:tab}
 Let $\Sigma_i \subset B_{R_i}=B_{R_i}(0)\subset \RR^3$
be a sequence of compact embedded minimal planar
domains  with $\partial
\Sigma_i\subset \partial B_{R_i}$ where $R_i\to \infty$. If
\begin{equation}
\sup_{B_1\cap \Sigma_i}|A|^2\to \infty \, ,
 \end{equation}
   then there exists a subsequence $\Sigma_j$,
     a lamination $\cL=\{x_3=t\}_{ \{ t \in \cI \} }$ of $\RR^3$
 by parallel
 planes (where $\cI \subset \RR$ is a closed set), and a closed nonempty set
 $\cS$ in the union of the leaves of $\cL$ such that
 after a rotation of $\RR^3$:
 \begin{enumerate}
\item[(A)] For each $1>\alpha>0$, $\Sigma_j\setminus \cS$
converges in the $C^{\alpha}$-topology to the lamination $\cL
\setminus \cS$.
\item[(B)]  $\sup_{B_{r}(x)\cap \Sigma_j}|A|^2\to\infty$ as $j \to \infty$ for
all $r>0$ and $x\in \cS$.  (The curvatures blow up along $\cS$.)
\end{enumerate}
\end{Thm}

Loosely speaking, our next result  shows that when the sequence is
ULSC (but not simply connected), a subsequence converges to a
foliation by parallel planes away from two lines $\cS_1$ and
$\cS_2$. The lines $\cS_1$ and $\cS_2$ are
disjoint and orthogonal to the leaves of the foliation and the two
lines are precisely the points where the curvature is blowing up.
This is similar to the case of disks, except that we get two
singular curves for non-disks as opposed to just one singular
curve for disks.

\begin{Thm} (Colding-Minicozzi, \cite{CM10})	 \label{t:t5.1}
 Let  a sequence $\Sigma_i$, limit lamination $\cL$, and singular
 set $\cS$
be  as in Theorem \ref{t:tab}.  Suppose that each
$B_R(0) \cap \Sigma_i$ is not simply-connected. If every $\Sigma_i$ is ULSC and
\begin{equation}
\sup_{B_1\cap \Sigma_i}|A|^2\to \infty \, ,
 \end{equation}
 then the limit lamination $\cL$ is the foliation $\cF = \{ x_3 = t\}_t$
  and the singular set $\cS$ is the union of two  disjoint
lines $\cS_1$ and $\cS_2$
 such that:
 \begin{enumerate}
\item[($C_{ulsc}$)]
Away from $\cS_1 \cup \cS_2$, each $\Sigma_j$ consists of exactly
two multi-valued graphs spiraling together.
 Near $\cS_1$ and $\cS_2$, the pair of multi-valued graphs form double
spiral staircases
with opposite orientations at $\cS_1$ and $\cS_2$. Thus, circling
only $\cS_1$ or only $\cS_2$  results in going either up or down,
while a path circling both $\cS_1$ and $\cS_2$ closes up.
\item[($D_{ulsc}$)] $\cS_1$ and $\cS_2$
are  orthogonal to the leaves of the foliation.
\end{enumerate}
\end{Thm}

\begin{Thm} (Colding-Minicozzi, \cite{CM10})	 \label{t:t5.2a}
 Let  a sequence $\Sigma_i$, limit lamination $\cL$, and singular
 set $\cS$
be  as in Theorem \ref{t:tab}. If
$\cSu = \emptyset$ and
\begin{equation}
\sup_{B_1\cap \Sigma_i}|A|^2\to \infty \, ,
 \end{equation}
   then $\cS=\cSt$ by \eqr{e:} and
 \begin{enumerate}
\item[($C_{neck}$)]
Each point $y$ in $\cS$ comes with a sequence of
{\underline{graphs}} in $\Sigma_j$ that converge  to the plane $\{
x_3 = x_3 (y) \}$. The convergence is in the $C^{\infty}$ topology
away from the point $y$ and possibly also one other point  in $\{
x_3 = x_3 (y) \} \cap \cS$.  If the convergence is   away from one
point, then these graphs are defined over annuli; if the
convergence is away from two points, then the graphs are defined
over disks with two subdisks removed.
\end{enumerate}
\end{Thm}

\subsection{Uniqueness of complete examples: Catenoid}

The catenoid is the only (non-flat) minimal surface of revolution, but there are a number of other ways to uniquely characterize it.  We will discuss this next.
In this subsection, $\Sigma$ will always be {\bf{complete}}, {\bf{minimal}}, and {\bf{embedded}} in $\RR^3$.  

\vskip2mm
The first modern results assumed that $\Sigma$ has finite total curvature{\footnote{A minimal surface $\Sigma$ has finite total curvature
if $\int_{\Sigma} |A|^2 < \infty$.}}:

\begin{Thm}
(Schoen, \cite{Sc2}) The catenoid is the unique $\Sigma$ with finite total curvature and two ends.
\end{Thm}

It follows from Schoen's result that an embedded finite total curvature minimal surface with two ends cannot have positive genus.  

\begin{Thm}
(Lopez and Ros, \cite{LRo}) The catenoid is the unique (non-flat) $\Sigma$ with finite total curvature and genus zero.
\end{Thm}

The main point of the  Lopez-Ros theorem is to show that a
(finite total curvature) genus zero minimal surface with more than two ends cannot be embedded.

\vskip2mm
A major breakthrough came in 1997 with Collin's proof of the generalized Nittsche conjecture:

\begin{Thm}	\label{t:gnc}
(Collin, \cite{Co}) If $\Sigma$ is proper, has finite topology and at least two ends, then it has finite total curvature.
\end{Thm}

 \cite{CM11} gives an alternative proof of Collin's theorem using the one-sided curvature estimate.  The assumption that $\Sigma$ has at least two ends rules out the possibility of the helicoid (which has infinite total curvature)

\vskip2mm
Finally, in 2008, Colding-Minicozzi, \cite{CM12}, showed that embeddedness and finite topology together imply properness, thus removing the assumption of properness.  The final result is:

\begin{Thm}
(Schoen, Lopez-Ros, Collin, Colding-Minicozzi)
The catenoid is only complete embedded minimal surface with finite topology and either:
\begin{itemize}
\item Exactly two ends, or
\item Genus zero and more than one end.
\end{itemize}
\end{Thm}

 There  are local versions of these global uniqueness results for the catenoid.  The starting point is a  local version of Collin's result that follows from the argument of \cite{CM11}, using the one-sided curvature estimate, although it is not recorded there.  This was done in \cite{CM16}, where the following local version was proven:
 
 \begin{Thm}
 (Colding-Minicozzi, \cite{CM16})
 There exist $\epsilon > 0$ and $C_1 , C_2 , C_3 > 1$ so that:  If $\Sigma \subset B_R \subset \RR^3$ is an embedded minimal annulus with $\partial \Sigma \subset \partial B_R$ and $\pi_1 (B_{\epsilon R}\cap \Sigma) \ne 0$, then there is a simple closed geodesic $\gamma \subset \Sigma$ of length $\ell$ so that:
 \begin{itemize}
 \item The curve $\gamma$ splits the connected component of $B_{R/C_1} \cap \Sigma$ containing it into annuli $\Sigma^+$ and $\Sigma^-$, each with $\int |A|^2 \leq 5 \pi$.
 \item Each of $\Sigma^{\pm} \setminus \cT_{C_2 \ell}(\gamma)$ is a graph with gradient $\leq 1$.
 \item $\ell \, \log (R/\ell) \leq C_3 \, h$ where the separation $h$ is given by
 $$
 h \equiv \min \, \left\{ |x^+ - x^-| \, \big| \, x^{\pm} \in \partial B_{R/C_1} \cap \Sigma^{\pm} \right\} \, .
 $$
 \end{itemize}
 \end{Thm}
 
 Here $\cT_s (S)$ denotes the (intrinsic in $\Sigma$) tubular neighborhood of radius $s$ about the set $S \subset \Sigma$.

\subsection{Uniqueness of complete examples: Helicoid}

In this subsection, $\Sigma$ will always be {\bf{complete}}, {\bf{minimal}}, and {\bf{embedded}} in $\RR^3$.

Using the lamination theorem and one-sided curvature estimate from \cite{CM6}--\cite{CM9},  
Meeks-Rosenberg proved the uniqueness of the helicoid in 2005:

\begin{Thm}	\label{t:MR}
(Meeks-Rosenberg, \cite{MeR2})
The helicoid is the unique (non-flat) proper, simply-connected $\Sigma$.
\end{Thm}

By  \cite{CM12}, the assumption of properness can be removed in Theorem \ref{t:MR}.

\vskip2mm
Again using \cite{CM6}--\cite{CM9},  Meeks-Rosenberg and Bernstein-Breiner studied the ends of finite genus embedded minimal surfaces, showing that these are asymptotic to helicoids.  The Bernstein-Breiner theorem gives:

\begin{Thm}
(Bernstein-Breiner, \cite{BB2})
 Any (non-flat) finite genus $\Sigma$ with one end is asymptotic to a helicoid.
\end{Thm}

In particular, any finite genus  embedded minimal surface with one end must be conformal to a punctured Riemann surface and one gets rather good control on the Weierstrass data (it has an essential singularity at the puncture, but of the same type that the helicoid does).  It would be very interesting to get a finer description of the moduli space of such examples.  One natural result in this direction is a recent theorem of Bernstein and Breiner (that proves a conjecture of Bobenko, \cite{Bo}):

\begin{Thm}
(Bernstein-Breiner, \cite{BB3})
Let $\Sigma$ be an embedded genus one helicoid in $\RR^3$.  Then there is a line $\ell$ so that rotation by $180$ degrees about $\ell$ is an orientation preserving isometry of $\Sigma$.
\end{Thm}

\subsection{Uniqueness of complete examples: Riemann examples}

Using  \cite{CM6}--\cite{CM10} and Colding-De Lellis-Minicozzi, \cite{CDM}, Meeks-Perez-Ros recently showed that
the Riemann examples are the unique $\Sigma$'s with genus zero and infinitely many ends.

\begin{Thm}	(Meeks-Perez-Ros, \cite{MePRs4})
The Riemann examples are the unique complete properly embedded minimal planar domains with infinitely many ends.
\end{Thm}

This theorem completes the classification of the genus zero properly embedded minimal surfaces.  Remarkably, it turned out that the classical examples discovered in the 1700's and 1800's were the only ones.  A number of central questions remain, including the structure of the moduli space of finite genus properly embedded minimal surfaces and the systematic construction of examples.

\subsection{Calabi-Yau conjectures}		 
 
 Recall that an immersed submanifold in $\RR^n$  is proper if the pre-image of any
compact subset of $\RR^n$ is compact in the surface. This property has played an
important role in the theory of minimal submanifolds and many of the classical
theorems in the subject assume that the submanifold is proper.
It is easy to see that any compact submanifold is automatically proper. On the
other hand, there is no reason to expect a general immersion  to
be proper. For example, the non-compact curve parametrized in polar coordinates
by
\begin{align}
	\rho (t) &= \pi + \arctan (t) \, , \notag \\
	  \theta (t) &= t \notag
\end{align}
spirals infinitely between the circles of radius $\pi/2$ and $3 \pi /2$. However, it was long
thought that a minimal immersion (or embedding) should be better behaved. This
principle was captured by the Calabi-Yau conjectures. Their original form
was given in 1965 in \cite{Ce} where E. Calabi made the following two conjectures about
minimal surfaces (see also S.S. Chern, page 212 of \cite{Cs} and S.T. Yau's 1982 problem
list, \cite{Ya3}):

\begin{Con} \label{con:1}
``Prove that a complete  minimal hypersurface in $\RR^n$ must be
unbounded.''
\end{Con}

Calabi continued:
``It is known that there are no compact minimal submanifolds of
$\RR^n$ (or of any simply connected complete Riemannian manifold
with sectional curvature $\leq 0$). A more ambitious conjecture
is'': 

\begin{Con} \label{con:2}
``A complete [non-flat] minimal hypersurface in $\RR^n$ has an
unbounded projection in every $(n-2)$--dimensional flat
subspace.''
\end{Con}

The  immersed  versions of these conjectures were shown to be false
by examples of  Jorge and Xavier, \cite{JXa2},  and
 N. Nadirashvili, \cite{Na}.  The latter constructed a complete immersion of a minimal disk into the unit ball in $\RR^3$, showing
that Conjecture \ref{con:1} also failed for immersed surfaces; cf.
  \cite{MaMo1}, \cite{LMaMo1}, \cite{LMaMo2}.

It is clear from the definition of proper that a proper minimal surface in $\RR^3$
must be unbounded, so the examples of Nadirashvili are not proper. 

The strong halfspace theorem of D.
Hoffman and W. Meeks shows that properness also prevented a minimal surface from being contained in a slab, 
or even a half-space:

\begin{Thm}	(Hoffman-Meeks, \cite{HoMe})  \label{t:home}
A complete connected properly immersed minimal surface
contained in $\{ x_3 > 0 \} \subset \RR^3$  must be a horizontal plane 
$\{ x_3 = {\text{Constant}} \}$.
\end{Thm}

In \cite{CM12}, it was shown that the  Calabi-Yau Conjectures were true for  embedded surfaces.  
We will describe this more precisely below.

The main result of \cite{CM12} is an effective version of properness for disks, giving a
chord-arc bound. Obviously, intrinsic distances are larger than extrinsic distances,
so the significance of a chord-arc bound is the reverse inequality, i.e., a bound on
intrinsic distances from above by extrinsic distances. This is accomplished in the
next theorem:

\begin{Thm}    (Colding-Minicozzi, \cite{CM12})  \label{t:cy1}
There exists a  constant  $C > 0$  so that if $\Sigma \subset
\RR^3$ is an embedded minimal disk, $B^{\Sigma}_{2R}=B^{\Sigma}_{2R}(0)$ is an
intrinsic ball 
 in $\Sigma \setminus
\partial \Sigma$ of radius $2R$, and if
$ \sup_{B^{\Sigma}_{r_0}}|A|^2>r_0^{-2}$ where $R>  r_0$,
 then for $x \in B^{\Sigma}_R$
 \begin{equation}   \label{e:t1cy1}
    C \, \dist_{\Sigma}(x,0)<   |x| + r_0   \, .
 \end{equation}
\end{Thm}

The assumption of a lower curvature bound,
$\sup_{B^{\Sigma}_{r_0}}|A|^2>r_0^{-2}$, in the theorem is a necessary
normalization for a chord-arc bound.  This can easily be seen by
rescaling and translating the helicoid.

 Properness of a complete embedded minimal disk   is an immediate
consequence of Theorem \ref{t:cy1}.  Namely, by \eqr{e:t1cy1}, as
intrinsic distances go to infinity, so do extrinsic distances.
Precisely, if $\Sigma$ is flat, and hence a plane, then obviously
$\Sigma$ is proper and if it is non-flat, then
$\sup_{B^{\Sigma}_{r_0}}|A|^2>r_0^{-2}$ for some $r_0>0$ and hence
$\Sigma$ is proper by \eqr{e:t1cy1}.

A  consequence of Theorem \ref{t:cy1} together with
the one-sided curvature estimate   is the following  version of that estimate
for intrinsic balls:

\begin{Cor} (Colding-Minicozzi, \cite{CM12})   \label{t:one-sided}
There exists $\epsilon>0$, so that if
\begin{equation}
    \Sigma \subset \{x_3>0\} \subset \RR^3
\end{equation}
 is an embedded
minimal disk,   $B^{\Sigma}_{2R} (x) \subset \Sigma
\setminus
\partial \Sigma$, and $|x|<\epsilon\,R$,  then
\begin{equation}        \label{e:graph}
\sup_{ B^{\Sigma}_{R}(x) } |A_{\Sigma}|^2 \leq R^{-2} \, .
\end{equation}
\end{Cor}

 As a corollary of this intrinsic one-sided curvature estimate we get that
the second, and ``more ambitious'', of Calabi's conjectures is
also true for {\underline{embedded}} minimal disks.   

In fact, \cite{CM12} proved both of Calabi's conjectures and properness also for embedded surfaces
with finite topology:

\begin{Thm}   (Colding-Minicozzi, \cite{CM12})     \label{c:cy1}
The plane is the only complete embedded minimal surface with finite topology in $\RR^3$
 in a halfspace.
\end{Thm}

\begin{Thm}   (Colding-Minicozzi, \cite{CM12})   \label{c:cy2}
A complete embedded minimal surface with finite topology in $\RR^3$
 must be proper.
\end{Thm}

There have been several properness results for Riemannian three-manifolds.  
W. Meeks and H. Rosenberg, \cite{MeR3}, generalized this to get a local version in Riemannian three-manifolds;
they also extended it to  embedded minimal surfaces with finite genus and
positive injectivity radius
in $\RR^3$.  In \cite{Cb}, B. Coskunuzer  proved properness for area minimizing disks in hyperbolic three-space, assuming that there is at least one $C^1$ point in the boundary at infinity.  Finally, Daniel, Meeks and Rosenberg proved a version for Lie Groups in \cite{DMR}.

There has been extensive  work on both properness   and the halfspace property    assuming various {\underline{curvature}}
{\underline{bounds}}. Jorge and Xavier, \cite{JXa1} and
\cite{JXa2}, showed that there cannot exist a complete immersed
minimal surface with {\underline{bounded}} {\underline{curvature}}
in $\cap_i \{ x_i > 0 \}$; later Xavier proved that the plane is
the only such surface in a halfspace, \cite{Xa}. Recently, G.P.
Bessa, Jorge and G. Oliveira-Filho, \cite{BJO}, and H. Rosenberg,
\cite{Ro}, have shown that if complete  embedded minimal surface
has bounded curvature, then it must be proper. This properness was
extended to embedded minimal surfaces with locally bounded
curvature and finite topology by Meeks and Rosenberg in
\cite{MeR2}; finite topology was subsequently replaced by finite
genus in \cite{MePRs1} by Meeks, J. Perez and A. Ros.

Inspired by Nadirashvili's examples, F. Martin and S. Morales
 constructed  in \cite{MaMo2} a complete bounded minimal
immersion which is proper in the (open) unit ball.  That is, the
preimages of compact subsets of the (open) unit ball are compact
in the surface and the image of the surface accumulates on the
boundary of the unit ball. They extended this  in
 \cite{MaMo3} to show that any convex, possibly noncompact or
nonsmooth, region of $\RR^3$ admits a proper complete minimal
immersion of the unit disk.   There are a number of interesting related results,
including \cite{AFM}, \cite{MaMeNa}, and \cite{ANa}.

\section{Mean curvature flow}  

We will now turn to the gradient flow for volume, i.e., mean curvature flow (or MCF).  This is the higher dimensional analog of the curve shortening flow.   In this section, we will give a rapid overview of the subject.

A one-parameter family of hypersurfaces $\Sigma_t \subset \RR^{n+1}$ flows by mean curvature if
\begin{equation}
	\partial_t \, x = - H \, \nn \, , \notag 
\end{equation}
where $\nn$ is the unit normal and $H$ is the mean curvature.  
Since the first variation formula gives that
\begin{equation}
 \frac{d}{dt} \,   \, \Vol \, (\Sigma_t)  = \int_{\Sigma_t}  \langle \partial_t \, x , H \, \nn \rangle  \, , 
 \notag
\end{equation}
we see that mean curvature flow is the (negative) gradient flow for volume and
\begin{equation}
	\frac{d}{dt} \,   \Vol (\Sigma_t) = - \int_{\Sigma_t} \, H^2  \, . \notag
\end{equation}
Minimal surfaces (where $H=0$) are fixed points for this flow.  The next simplest example is given by concentric round $n$-dimensional spheres of radius $\sqrt{-2nt}$ for $t< 0$.

 \begin{figure}[htbp]
\centering\includegraphics[totalheight=.4\textheight, width=1\textwidth]{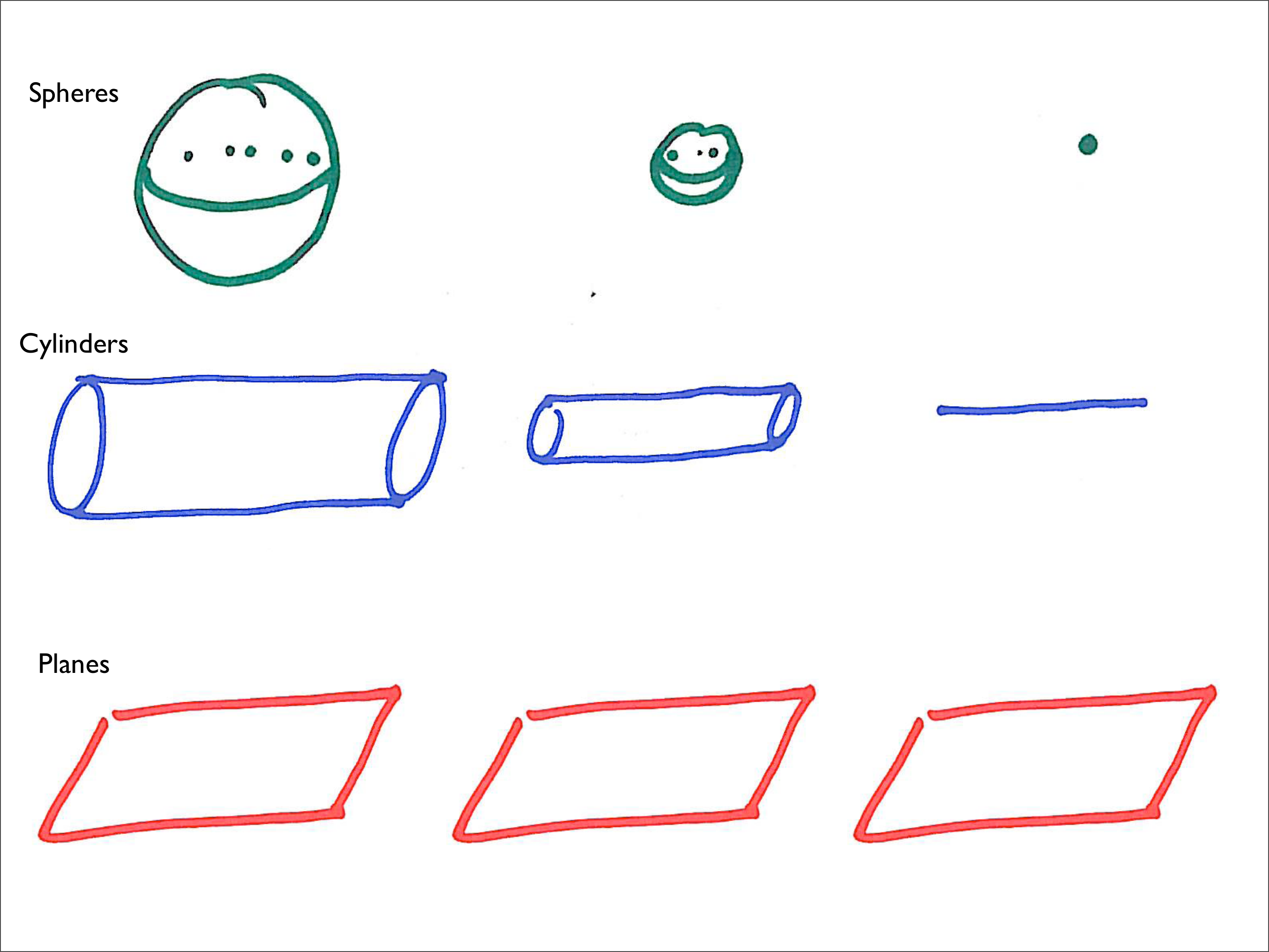}
\caption{Cylinders, spheres and planes are self-similar solutions of the mean curvature flow.
The shape is preserved, but the scale changes with time.}   
  \end{figure}

\subsection{MCF of graphs}

We saw earlier that the mean curvature of the graph of a function $u: \RR^n \to \RR$ is given by
\begin{equation}
	H   =  - \dv_{\RR^n} \, \, \left(  \frac{ \nabla_{\RR^n} u }{ \sqrt{1 + \left| \nabla_{\RR^n} u  \right|^2 } } \right)  \, .  \notag 
\end{equation}  	
Thus, the graph of $u(x,t)$ flows by MCF if  
\begin{equation}
	 \frac{\partial u}{\partial t} =  \left( 1 + \left| \nabla_{\RR^n} u  \right|^2 \right)^{\frac{1}{2} }  \, 
	 \dv_{\RR^n} \, \, \left(  \frac{ \nabla_{\RR^n} u }{ \sqrt{1 + \left| \nabla_{\RR^n} u  \right|^2 } } \right)   \, .  \notag 
\end{equation}  
The factor $\left( 1 + \left| \nabla_{\RR^n} u  \right|^2 \right)^{\frac{1}{2} } $ on the right compensates for the fact that the $x_{n+1}$ direction is not normal to the graph.

We already saw that Calabi's grim reaper 
   $
u(x,\,t)=t-\log\sin x  $ gives an example of graphs flowing by mean curvature; see \cite{AW} and \cite{CSS} for similar examples and stability results.

\subsection{Self-similar shrinkers}

Let $\Sigma_t \subset \RR^{n+1}$ be a  one-parameter family of 
hypersurfaces flowing by MCF for $t< 0$.
$\Sigma_t$ is said to be a self-similar shrinker if $$\Sigma_t = \sqrt{-t} \, \Sigma_{-1}$$ for all $t< 0$.
We saw that spheres of radius $\sqrt{-2nt}$ give such a solution.
Unlike the case of curves, numerical evidence suggests that a complete classification of embedded self-shrinkers 
is impossible.  In spite of all the numerical evidence, there are very few rigorous examples of self-shrinkers.

In 1992, 
 Angenent, \cite{A}, constructed a self-similar shrinking donut in $\RR^3$.
The shrinking donut was given by rotating a simple closed curve around an axis.  See
Kleene and M\"oller, \cite{KlMo}, for a classification of self-shrinkers of rotation (including ones with boundary).
 
 \begin{figure}[htbp]
\centering\includegraphics[totalheight=.4\textheight, width=.75\textwidth]{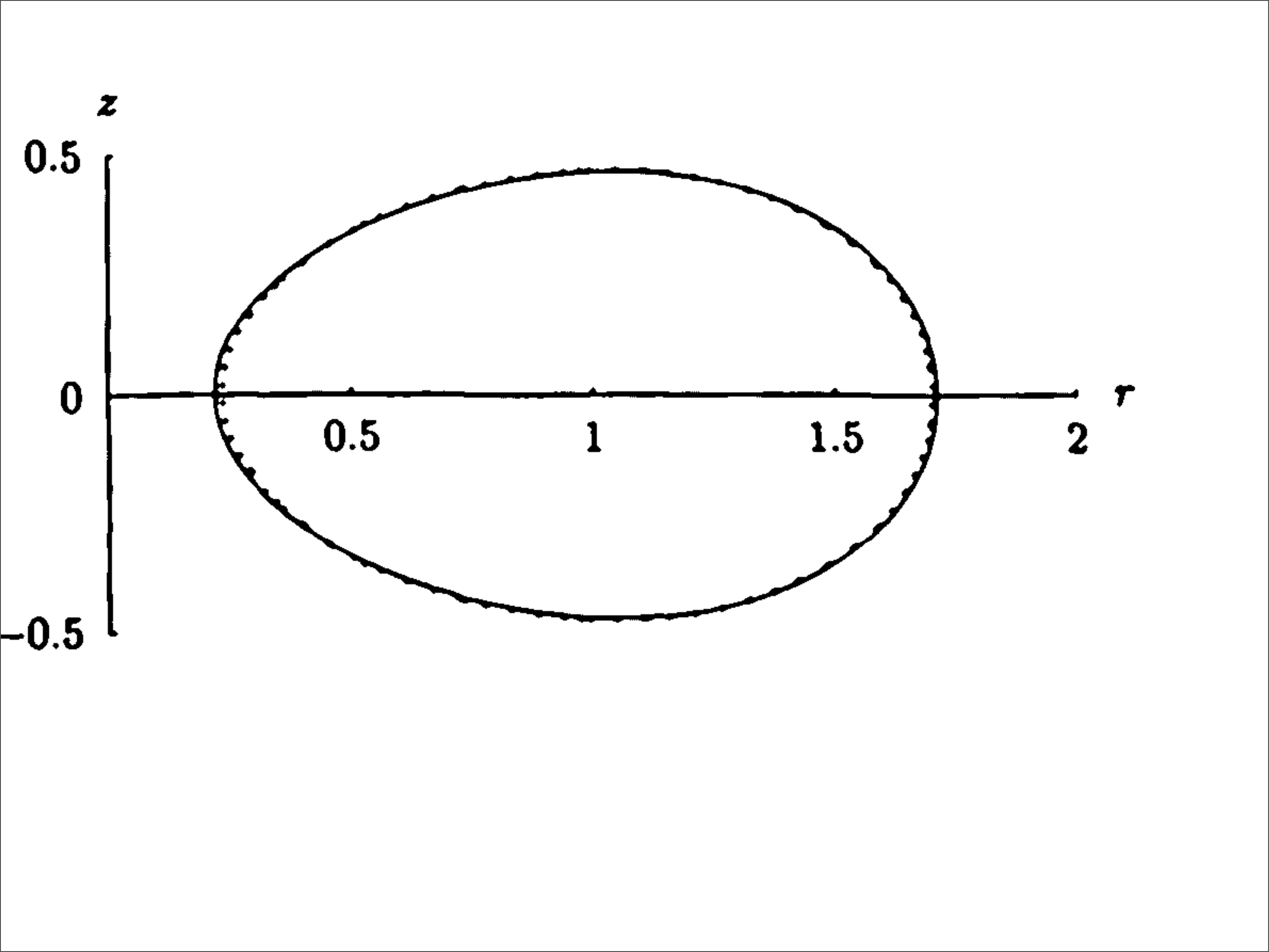}
\caption{Angenent's shrinking donut from numerical simulations of D. Chopp, \cite{Ch}.  The vertical $z$-axis is the axis of rotation
and the horizontal $r$-axis is a line of reflection symmetry.}   
  \end{figure}

\subsection{The maximum principle}

The parabolic maximum principle plays an important role in mean curvature flow.  
For instance, it is used   to prove the following key facts:
\begin{enumerate}
\item If two closed hypersurfaces are disjoint, then they remain disjoint under MCF.
\item A closed embedded hypersurface remains embedded under MCF.
\item If a closed hypersurface is convex, then it remains convex under MCF.
\item Likewise, mean convexity (i.e, $H>0$) is preserved under MCF.
\end{enumerate}

In 1989, Grayson, \cite{G2}, showed that his result for curves does not extend to surfaces.  In particular, he showed that 
a dumbbell with a sufficiently long and narrow bar will develop a pinching singularity before extinction.  A later proof was given by Angenent, \cite{A}, using the shrinking donut and the avoidance property (1).
Figures \ref{f:figdb1} to \ref{f:figdb2} show $8$ snapshots in time of the evolution of a dumbbell; the figures were created by computer simulation by U. Mayer.
%{\footnote{See the website  www.math.utah.edu/~mayer/math/MCF/dumbbell2_js.html.}}

% \begin{figure}[htbp]
    %\setlength{\captionindent}{20pt}  
%\centering\includegraphics[totalheight=.4\textheight, width=1\textwidth]{dumbbell}
%\caption{Grayson's dumbbell first develops a pinching singularity, separating into two components.}   
 % \end{figure}

\begin{figure}[htbp]
    \begin{minipage}[t]{0.5\textwidth}
    \centering\includegraphics[totalheight=.12\textheight, width=.95\textwidth]{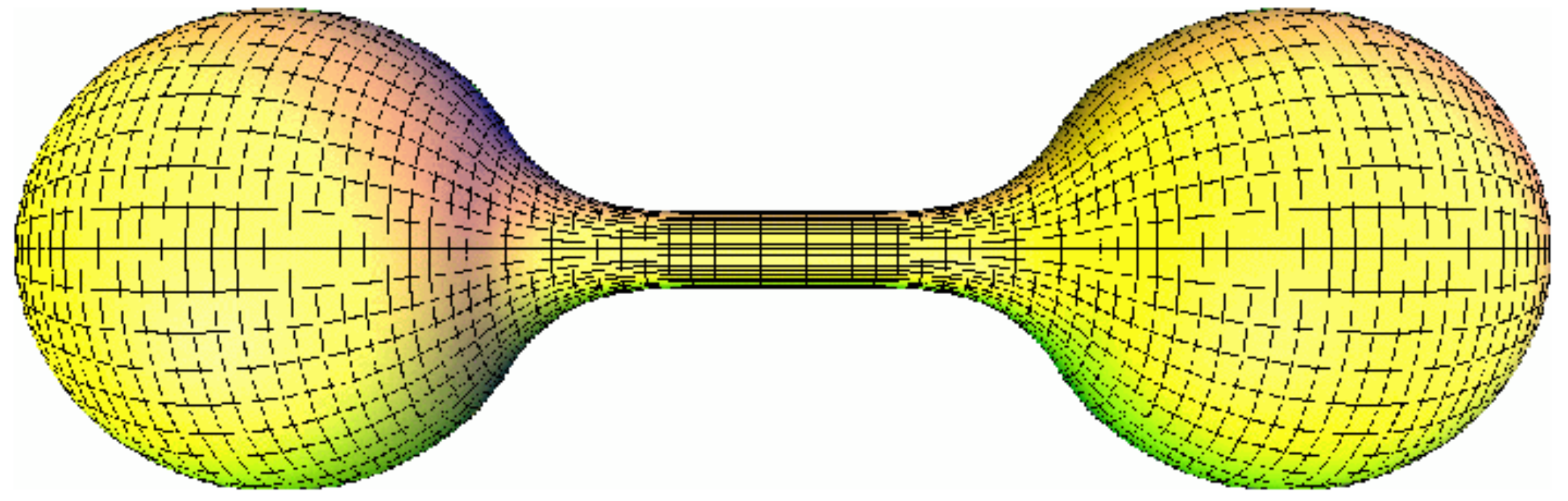}
   %\caption{Grayson's dumbbell; step 0.}  
    \end{minipage}\begin{minipage}[t]{0.5\textwidth}
    \centering\includegraphics[totalheight=.12\textheight, width=.95\textwidth]{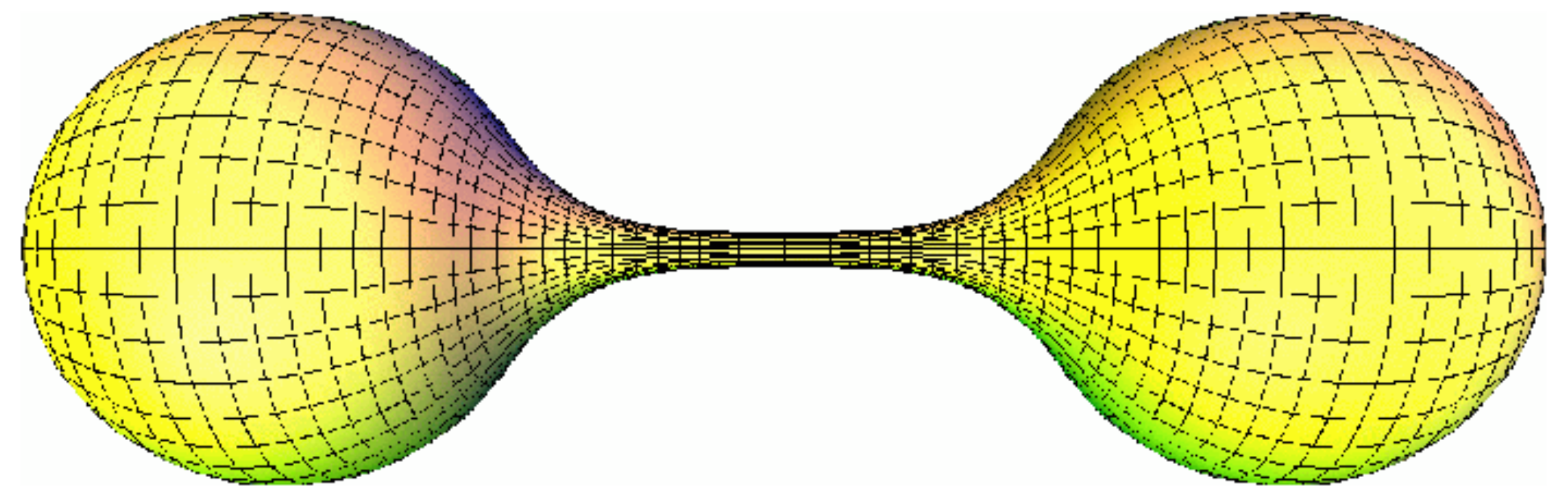}
    %\caption{The dumbbell; step 1.} 
    \end{minipage}
    \caption{Grayson's dumbbell; initial surface and step 1.}   \label{f:figdb1}
\end{figure}

\begin{figure}[htbp]
    \begin{minipage}[t]{0.5\textwidth}
    \centering\includegraphics[totalheight=.12\textheight, width=.95\textwidth]{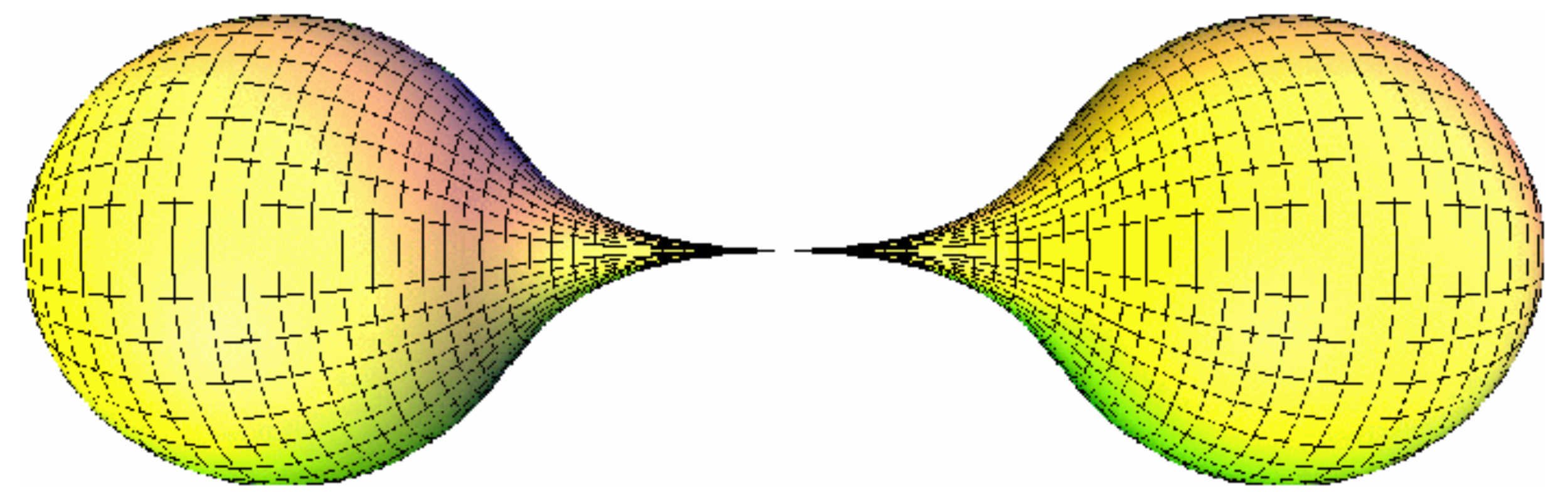}
   %\caption{The dumbbell; step 2.}   
    \end{minipage}\begin{minipage}[t]{0.5\textwidth}
    \centering\includegraphics[totalheight=.12\textheight, width=.95\textwidth]{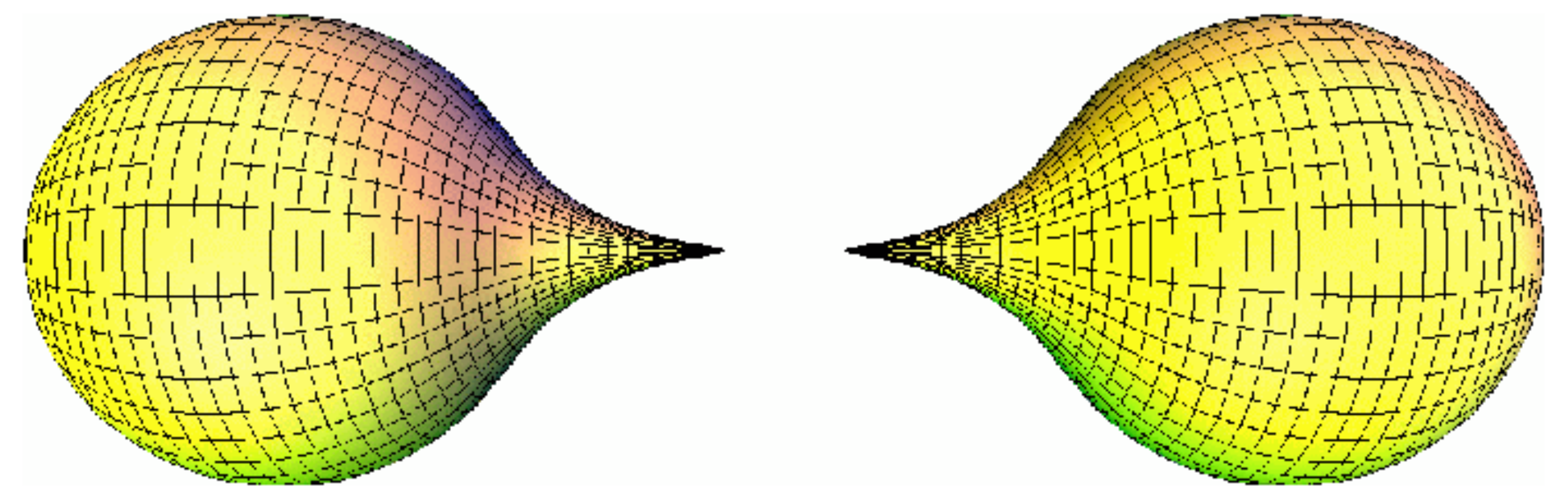}
    %\caption{The dumbbell; step 3.} 
    \end{minipage}
    \caption{The dumbbell; steps 2 and 3.}   
\end{figure}

\begin{figure}[htbp]
    \begin{minipage}[t]{0.5\textwidth}
    \centering\includegraphics[totalheight=.12\textheight, width=.95\textwidth]{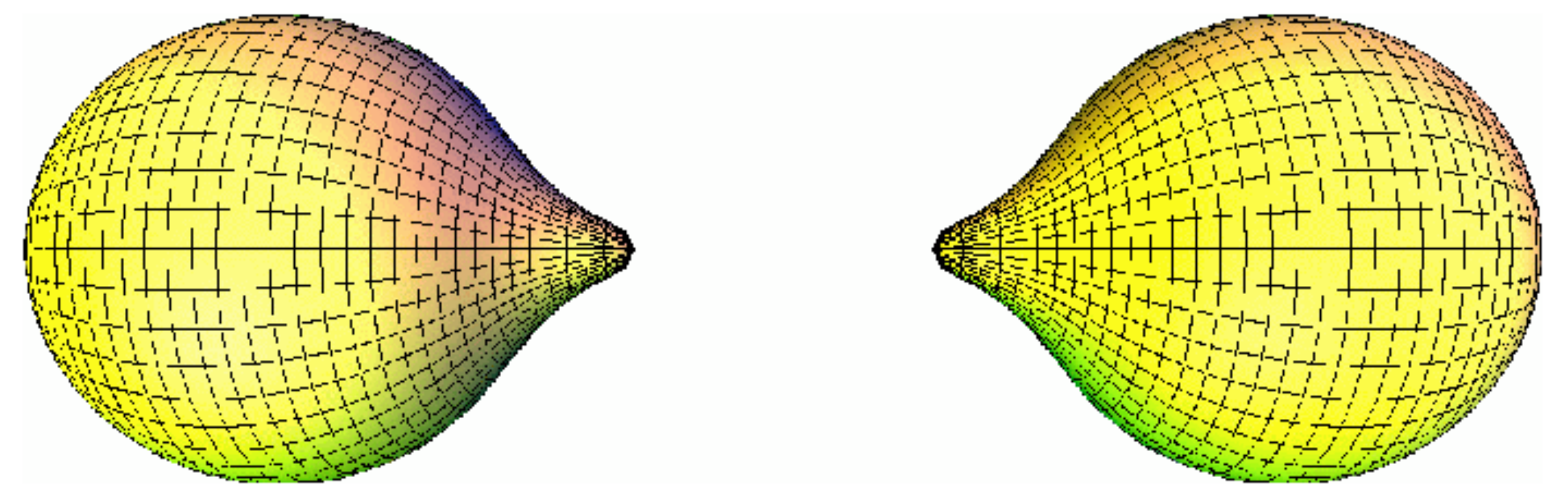}
   %\caption{The dumbbell; step 4.} 
    \end{minipage}\begin{minipage}[t]{0.5\textwidth}
    \centering\includegraphics[totalheight=.12\textheight, width=.95\textwidth]{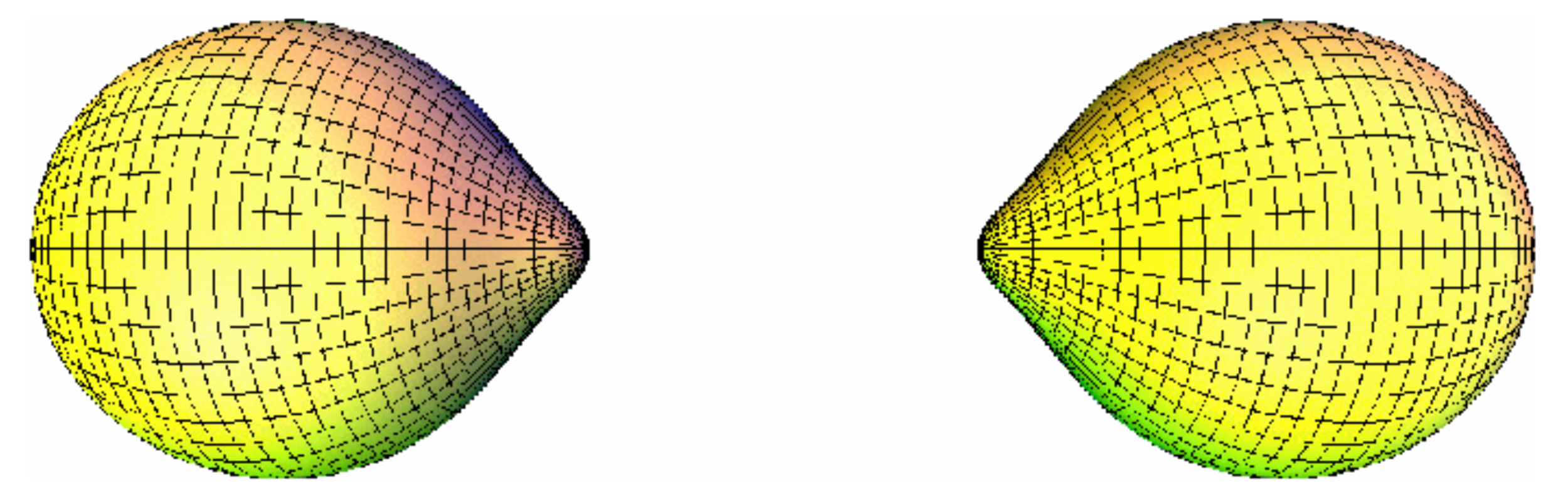}
   %\caption{The dumbbell; step 5.} 
    \end{minipage}
    \caption{The dumbbell; steps 4 and 5.}  
\end{figure}

\begin{figure}[htbp]
    \begin{minipage}[t]{0.5\textwidth}
    \centering\includegraphics[totalheight=.12\textheight, width=.95\textwidth]{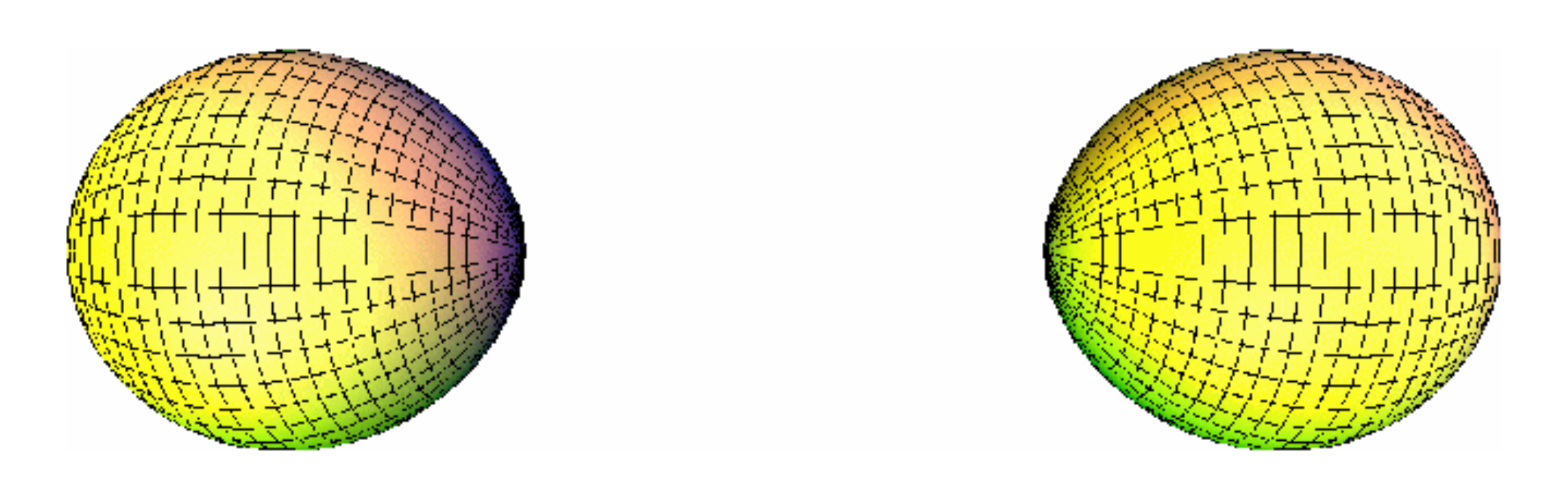}
   %\caption{The dumbbell; step 4.} 
    \end{minipage}\begin{minipage}[t]{0.5\textwidth}
    \centering\includegraphics[totalheight=.12\textheight, width=.95\textwidth]{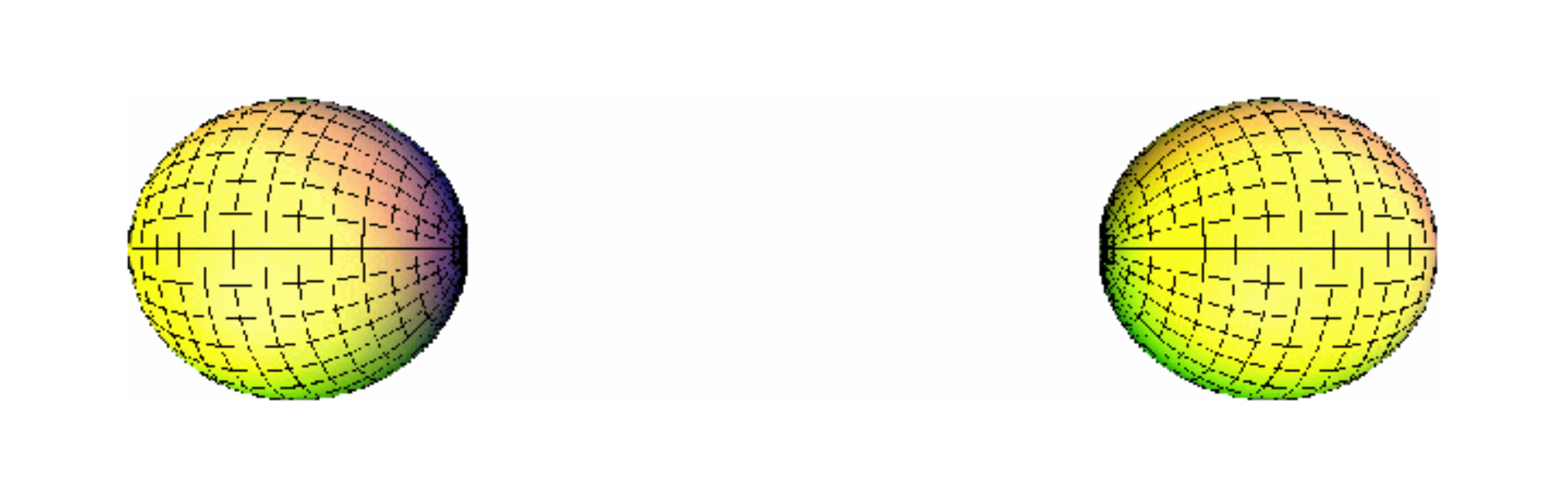}
   %\caption{The dumbbell; step 5.} 
    \end{minipage}
    \caption{The dumbbell; steps 6 and 7.}   \label{f:figdb2}
\end{figure}

\subsection{The self-shrinker equation}

%here on September 24, 2010

A MCF $M_t$ is   a self-similar shrinker if $M_{t}=\sqrt{-t}M_{-1}$ for $t<0$.  
This is equivalent to that
$\Sigma=M_{-1}$ satisfies the equation{\footnote{This equation differs by a factor of two from Huisken's definition of a self-shrinker; this is because Huisken works with the time $-1/2$ slice.}}
\begin{equation}  
H=\frac{\langle x,\nn \rangle}{2}\, .   \notag
\end{equation}
That is:  $M_t=\sqrt{-t}M_{-1}$  $\iff$  $M_{-1}$ satisfies $H=\frac{\langle x,\nn \rangle}{2}$.

\begin{figure}[htbp]
    \begin{minipage}[t]{0.5\textwidth}
    \centering\includegraphics[totalheight=.3\textheight, width=1\textwidth]{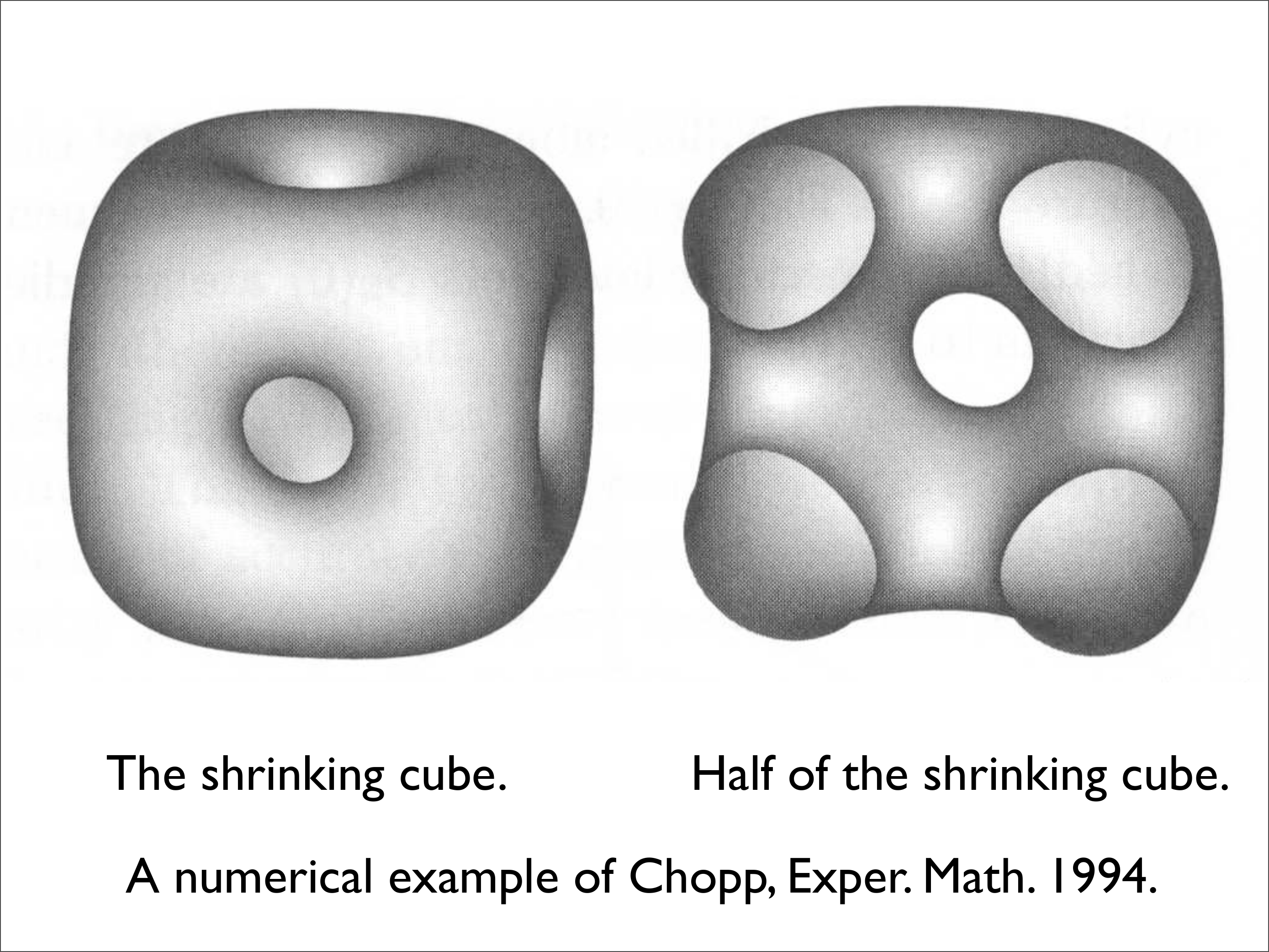}
    \caption{A numerical example of a closed shrinker from Chopp, \cite{Ch}.}  
    \end{minipage}\begin{minipage}[t]{0.5\textwidth}
    \centering\includegraphics[totalheight=.3\textheight, width=1\textwidth]{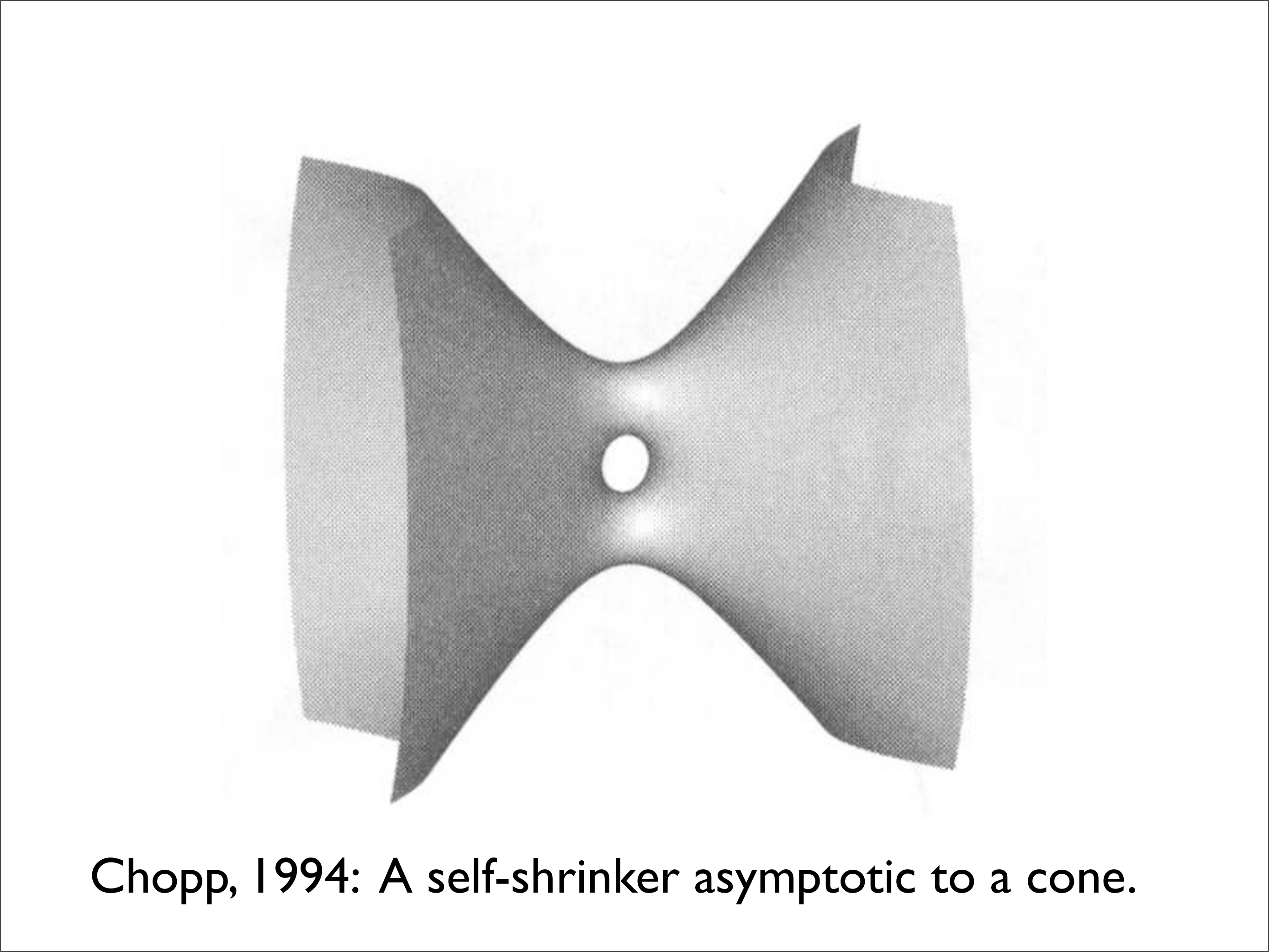}
    \caption{A non-compact numerical example from Chopp, \cite{Ch}.} 
    \end{minipage}
\end{figure}

The self-shrinker equation arises variationally in two closely related ways: as minimal surfaces for a conformally changed metric and as critical points for a weighted area functional.  We return to the second later, but state the first now:

\begin{Lem}
  $\Sigma$ is a self-shrinker
 $\iff$
 $\Sigma$ is a minimal surface in the metric 
 \begin{equation}
 	g_{ij}=\e^{-\frac{|x|^2}{2n}}\,\delta_{ij} \, .  \notag
\end{equation}
\end{Lem}
 
The proof follows immediately from the first variation.
Unfortunately, this metric is not complete (the distance to infinity is finite) and
 the curvature blows up exponentially.

\subsection{Huisken's theorem about MCF of convex hypersurfaces}

In 1984, Huisken, \cite{H1}, showed that convexity is preserved under MCF and showed that the surfaces become round:

\begin{Thm}
(Huisken, \cite{H1})  Under MCF, every closed convex hypersurface in $\RR^{n+1}$  remains convex and eventually becomes extinct in a ``round point''.
\end{Thm}

 This is exactly analogous to the result of 
Gage-Hamilton for convex curves, but it's interesting to note that 
Huisken's proof works only for $n>1$.   Namely, he shows that the hypersurfaces become closer and closer to being umbilic and that the limiting shapes are umbilic.  A hypersurface is umbilic if all of the eigenvalues of the second fundamental form are the same; this characterizes the sphere when there are at least two eigenvalues, but is meaningless for curves.

We mention a few related results.  First, Schulze showed a generalization of this for flows by other powers of mean curvature in \cite{Sf1} and \cite{Sf2}.  Second, Sesum showed that the rescaled mean curvature flow converges exponentially to a round sphere in \cite{Se}.

\vskip2mm
Convexity means that every eigenvalue of $A$ has the right sign.  There are weaker conditions that are also preserved under mean curvature flow and where significant results have been obtained.  The first of these is mean convexity, where the hypersurfaces have positive mean curvature.  {\bf{Mean convex}} flows have been analyzed by   Huisken-Sinestrari, \cite{HS1} and \cite{HS2}, and 
  White, \cite{W2} and \cite{W3}.
A related class of hypersurfaces are the {\bf{$2$-convex}} ones, where the sum of any pair of principal curvatures is positive (it is not hard to see that this implies mean convexity); this case has been studied by Huisken-Sinestrari, \cite{HS3}.  In addition, 
 Smoczyk, \cite{Sm1}, showed that {\bf{Star-shaped}} hypersurfaces remain star-shaped under MCF.  Finally, 
Ecker-Huisken, \cite{EH1} and \cite{EH2}, showed that being graphical is also preserved under MCF
and proved estimates for graphical mean curvature flow.  In each of these cases, the maximum principle is used to show that the condition is preserved under MCF.

\section{Width and mean curvature flow}

We saw previously that every closed hypersurface must become extinct under MCF in a finite amount of time.  It is interesting then to estimate this extinction time.  On obvious estimate is in terms of the diameter since the hypersurface must become extinct before a ball that encloses it.  However, there are cases where this estimate is far from sharp.  In this section, we will prove another extinction time estimate in terms of the geometric invariant called the width that was previously introduced.
  
\subsection{Sweepouts and one-dimensional width}

Let $M$ be a smooth closed convex surface in $\RR^3$.{\footnote{This works for convex hypersurfaces in $\RR^{n+1}$ too with the obvious modifications.}}  Convexity implies that $M$ is diffeomorphic to $\SS^2$ and, thus, we can fix  a map $\sigma: \SS^1 \times [0,1 ] \to M$ that maps $\SS^1 \times \{ 0 \}$ and $\SS^1 \times \{ 1 \}$ to points and that  is topologically a degree one map from $\SS^2$ to $\SS^2$.  Let $\Omega_{\sigma}$ denote the homotopy class of such maps.
 \begin{figure}[htbp]
\centering\includegraphics[totalheight=.55\textheight, width=.85\textwidth]{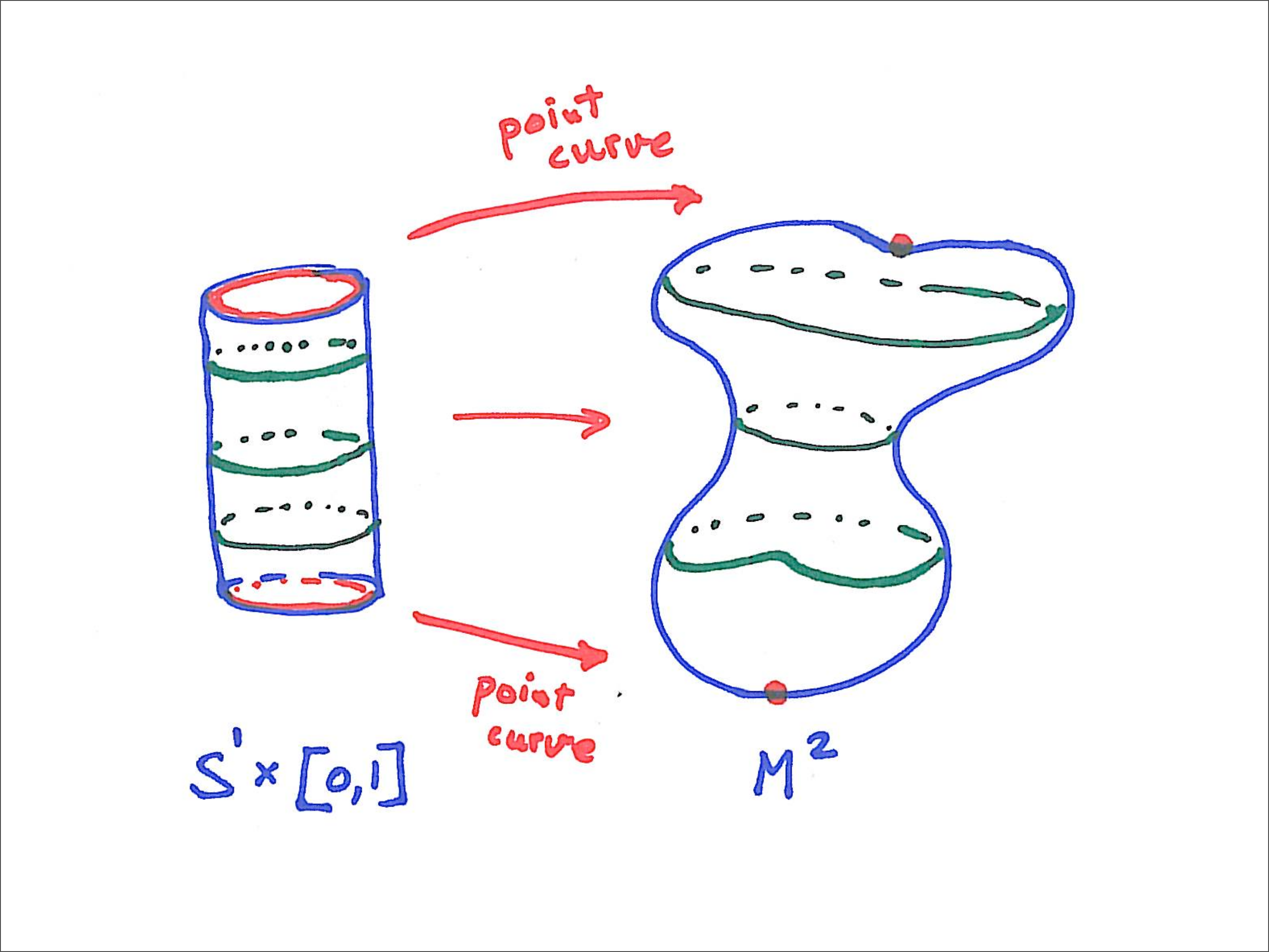}
\caption{A sweepout.}
\end{figure}

Given this homotopy class, the 
 width $W = W(\sigma)$ was defined in \eqr{e:w} to be
\begin{equation}
	W = \inf_{\hat{\sigma} \in \Omega_{\sigma}} \, \, \max_{s \in [0,1]} \,
	\Energy \, (\hat{\sigma} ( \cdot , s)) \, ,  \notag
\end{equation}
where the energy is given by
\begin{equation}
	\Energy \, (\hat{\sigma} ( \cdot , s))  = \frac{1}{2} \, \int_{\SS^1} \, \left| \partial_x \, 
			\hat{\sigma} ( x , s) \right|^2 \, dx \, .  \notag
\end{equation}

It is not hard to see that 
the width is continuous in the metric, though the curve realizing it may not be.

\begin{figure}[htbp]
\centering\includegraphics[totalheight=.35\textheight, width=1\textwidth]{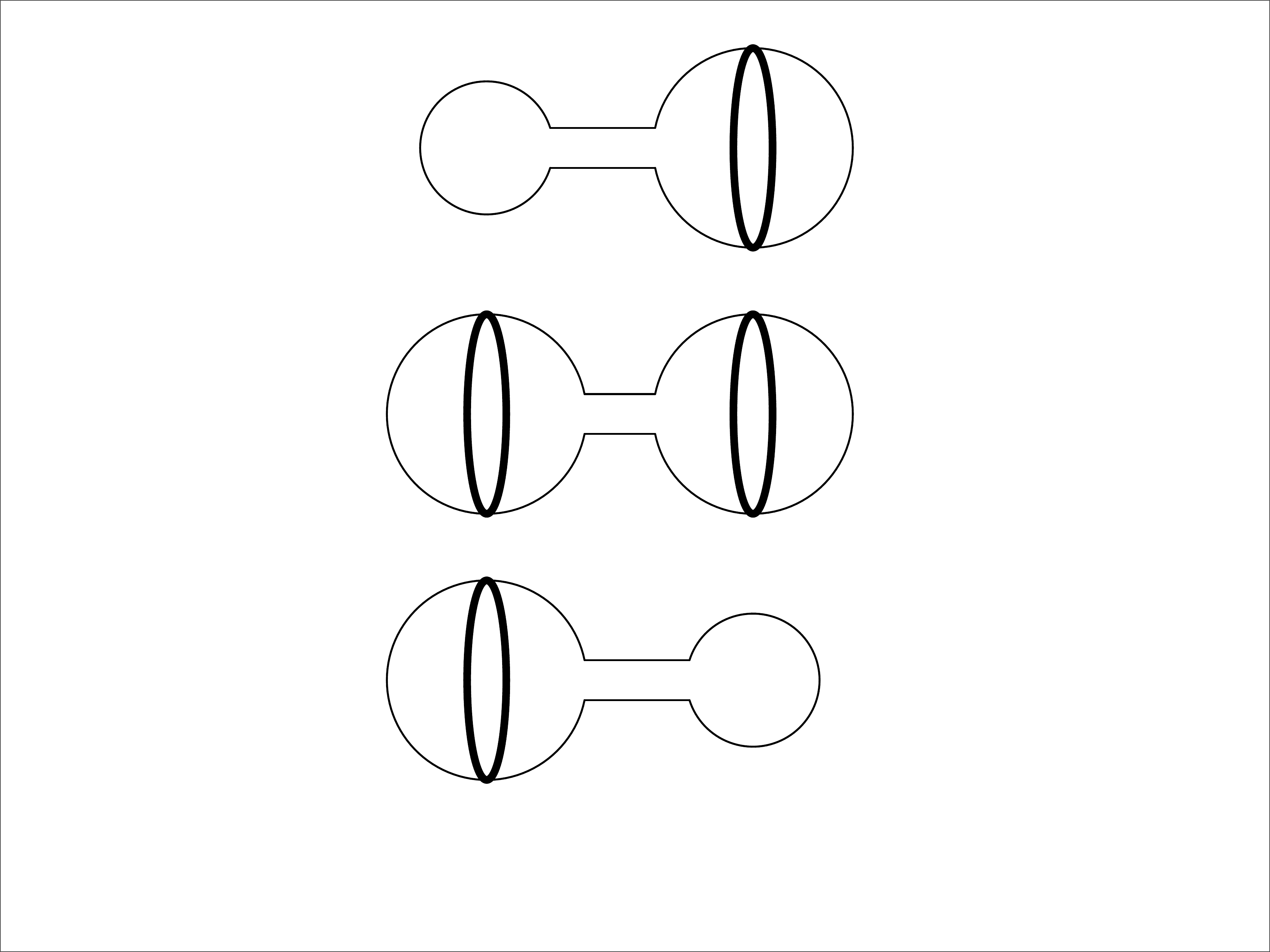}
\caption{Three snapshots of a one-parameter family of ``dumbbell'' metrics.
The geodesic realizing the width jumps from one bell to the other.
The jump occurs in the middle picture where the geodesic is not unique.}
\end{figure}

\subsection{Estimates for rate of change of width under the mean curvature flow}

In \cite{CM2}, we proved the following estimate for the rate of change of the width under MCF:

\begin{Thm}	(Colding-Minicozzi, \cite{CM2})
If 
$M_t$ is a MCF  of closed convex hypersurfaces and 
$W(t)$ is the width of $M_t$, then
\begin{equation}	\label{e:ddf}
	\frac{d}{dt} \, W \leq - 2 \pi  
\end{equation}
in the sense of limsup of forward difference quotients.
\end{Thm}

Note that $W(t)$ is continuous but may not be differentiable in $t$,  so \eqr{e:ddf} may not hold in the classical sense.  Still, the fact that it holds in the sense of limsup of forward difference quotients is enough to integrate and get that
\begin{equation}
	W(t) \leq W(0) - 2 \, \pi \, t \, .   
\end{equation}
Since the width is obviously positive until $M_t$ becomes extinct, we see that 
 $M_t$ becomes extinct by time $\frac{W(0)}{2\pi}$.

 Finally, observe that   capping off a long thin cylinder gives convex surfaces with a fixed bound on the width (coming from the radius of the cylinder) but with arbitrarily large diameter.  For these surfaces, the estimate on the extinction time coming from the width is much better than what one would get from the diameter.

\section{Singularities for MCF}

We will  now leave convex hypersurfaces and go to the general case.  As Grayson's dumbbell showed, there is no higher dimensional analog of Grayson's theorem for curves.
The key for analyzing singularities is a blow up (or rescaling) analysis similar to the tangent cone analysis for minimal surfaces.
As for minimal surfaces, the starting point is a monotonicity formula that gives uniform control over the rescalings.

\subsection{Huisken's monotonicity}

We will need to recall Huisken's monotonicity formula (see \cite{H3}, \cite{E1},
\cite{E2}).
To do this, first
define the non-negative function $\Phi$ on $\RR^{n+1} \times (-\infty,0)$ by
 \begin{equation}
\Phi  (x,t)  = [-4\pi t]^{-\frac{n}{2}}
\,\e^{\frac{|x|^2}{4t}}\, ,
\end{equation}
and then set
$\Phi_{(x_0,t_0)} (x,t) = \Phi (x-x_0, t-t_0)$.  In 1990, G. Huisken proved the following monotonicity formula for mean curvature flow, \cite{H3}:

\begin{Thm}	\label{t:huiskenmon}
(Huisken, \cite{H3})
  If $M_t$ is a solution to
the MCF and $u$ is a $C^2$ function, then
\begin{equation}    \label{e:huisken}
\frac{d}{dt} \int_{M_t} u \, \Phi_{(x_0,t_0)}=
-\int_{M_t} \left| H\nn- \frac{(x-x_0)^{\perp}
}{2\,(t_0-t)}\right|^2 \,u\,\Phi_{(x_0,t_0)} + \int_{M_t} \left[
u_t - \Delta u \right] \, \Phi
 \, .
\end{equation}
\end{Thm}

When $u$ is identically one, we get the monotonicity formula
\begin{equation}    \label{e:huisken2}
\frac{d}{dt} \int_{M_t}   \Phi_{(x_0,t_0)}=
-\int_{M_t} \left| H\nn-\frac{(x-x_0)^{\perp}
}{2\,(t_0-t)}\right|^2 \,\Phi_{(x_0,t_0)}
 \, .
\end{equation}
Huisken's density is the limit of $\int_{M_t} \Phi_{x_0,t_0}$ as $t\to t_0$.
That is,
\begin{equation}
\Theta_{x_0,t_0}=\lim_{t\to t_0} \int_{M_t}\Phi_{x_0,t_0}\, ;
\end{equation}
this limit exists by the monotonicity \eqr{e:huisken2} and the density is non-negative as the integrand $\Phi_{x_0,t_0}$ is non-negative.

It is also interesting to note that 
 Huisken's Gaussian volume is constant in time if and only if   $M_t$ is a self-similar shrinker with
 $$
 M_t = \sqrt{-t} \, \left( M_{-1} \right) \, .
 $$

One drawback to Huisken's formula is that it requires one to integrate over all of space.  In 2001, K. Ecker discovered a local monotonicity formula where the integral is over bounded sets.  This formula is modeled on
  Watson's mean-value formula for the linear heat equation; see \cite{E2} for details.

\subsection{Tangent flows}

If $M_t$ is a MCF, then so is the parabolic scaling for any constant $\lambda > 0$
\begin{equation}
	\tilde{M}_t = \lambda \, M_{\lambda^{-2} \, t}  \, .  \notag
\end{equation}
When $\lambda$ is large, this magnifies a small neighborhood of the origin in space-time.

If we now take a sequence $\lambda_i \to \infty$ and let $M^i_t = \lambda_i \, M_{\lambda_i^{-2} \, t}$, 
then
Huisken's monotonicity gives uniform Gaussian area bounds on the rescaled sequence. Combining this with Brakke's weak compactness theorem for mean curvature flow, \cite{B}, it follows that 
a subsequence of the $M_t^i$ converges to a limiting flow $M^{\infty}_t$  (cf., for instance, page 675--676 of \cite{W2} and chapter 7 of \cite{I2}).  Moreover, Huisken's monotonicity implies that the Gaussian area (centered at the origin) is now constant in time, so we conclude that $M^{\infty}_t$ is  a self-similar shrinker.
This $M^{\infty}_t$ is called a {\emph{tangent flow}} at the origin.  The same construction can be done at any point of space-time.

We will return to this point of view later when we describe results from \cite{CM1} classifying the generic tangent flows.

\subsection{Gaussian integrals and the F functionals}

For $t_0>0$ and $x_0\in \RR^{n+1}$,   define   $F_{x_0,t_0}$  by
\begin{align}
F_{x_0 , t_0} (\Sigma) &= (4\pi t_0)^{-n/2} \, \int_{\Sigma} \, \e^{-\frac{|x- x_0|^2}{4t_0}} \, d\mu\notag\\ &= \int_{\Sigma}\Phi_{x_0,t_0}(\cdot,0)\, .   \notag
\end{align}
We will think of $x_0$ as being the point in space that we focus on and $t_0$ as being the scale.  By convention, we set
$F=F_{0,1}$.

We will next compute the first variation of the $F$ functionals, but we will allow variations in all three parameters: the hypersurface $\Sigma_0$, the center $x_0$, and the scale $t_0$.  Namely, 
fix a hypersurface $\Sigma_0 \subset \RR^{n+1}$ with unit normal $\nn$, a function $f$, vectors $x_0 , y \in \RR^{n+1}$ and constants
$t_0 , h \in \RR$ with $t_0 > 0$.
Define variations (i.e., one-parameter families) by
\begin{align}
	\Sigma_s &= \{ x + s \, f(x) \, \nn (x) \, | \, x \in \Sigma_0 \} \, , 	\notag \\
	x_s &= x_0 + s \, y \, , \notag  \\
	t_s &= t_0 + s \, h \, . \notag
\end{align}
The first variation is given by:

\begin{Lem}	(Colding-Minicozzi, \cite{CM1})  \label{l:varl0}
If $\Sigma_s$, $x_s$ and $t_s$ are variations as above, then 
  $\frac{\partial }{\partial s} \,  \, \left( F_{ x_s , t_s} (\Sigma_s) \right)$ is
\begin{equation}	 
	  (4\pi \, t_0)^{-\frac{n}{2}} \,  \int_{\Sigma}  \left[ f \, \left( H - \frac{ \langle x-x_0 ,  \nn \rangle}{2t_0}
		\right)
		+ h \, \left( \frac{|x-x_0|^2}{4t_0^2} - \frac{n}{2t_0} \right)  +
	  \frac{  \langle x-x_0 , y \rangle }{2t_0}    \right] \, \e^{\frac{-|x-x_0|^2}{4t_0}} \, d\mu \, . \notag
\end{equation}
\end{Lem}

\begin{proof}
 From the first variation formula (for area), we know that
 \begin{equation}	\label{e:fhdm}
 	(d\mu)' = f \, H \, d \mu \, .
 \end{equation}
 The $s$ derivative of the weight  $\e^{-|x-x_s|^2/(4t_s)}$ will have three separate terms coming from the variation of the surface, the variation of $x_s$, and the variation of $t_s$.  Using, respectively, that $\nabla  |x-x_s|^2  = 2\, (x-x_s)$,
 \begin{equation}
	\partial_{t_s}  \log \left[ (4\pi t_s)^{-n/2} \,   \e^{-\frac{|x-x_s|^2}{4t_s}} \right] = \frac{-n}{2t_s} + \frac{|x-x_s|^2}{4t_s^2}
\end{equation}
and  $\partial_{x_s}  |x-x_s|^2  = 2\, (x_s-x)$, we get that
the derivative of  $$\log \, \left[ (4\pi \, t_s)^{-\frac{n}{2}} \, \e^{-|x-x_s|^2/(4t_s)} \right]$$ at $s=0$ is given by
 \begin{equation}	\label{e:logder}
 	  - \frac{ f}{2t_0} \, \langle x-x_0 , \nn \rangle  + h \, \left( \frac{|x-x_0|^2}{4t_0^2} - \frac{n}{2t_0} \right) +
	  \frac{ 1}{2t_0} \, \langle x-x_0 , y \rangle \, .
 \end{equation}
 Combining this with \eqr{e:fhdm} gives the lemma.
\end{proof}

 We will say that $\Sigma$ is a critical point for $F_{x_0 , t_0} $ if it is simultaneously critical with respect to variations in all three parameters, i.e.,   variations in $\Sigma$ and all variations in $x_0$  and   $t_0$.  Strictly speaking, it is the triplet $(\Sigma, x_0 , t_0)$ that is a critical point of $F$, but we will refer   to $\Sigma$ as a critical point of $F_{x_0 , t_0}$.
 The next proposition shows that $\Sigma$ is a critical point for $F_{x_0 , t_0} $ if and only if it is the time $-t_0$ slice of a self-shrinking solution of the mean curvature flow that becomes extinct at the point $x_0$ and time $0$.

 \begin{Pro}	\label{p:critall}
 (Colding-Minicozzi, \cite{CM1})
$\Sigma$ is a critical point for $F_{x_0 , t_0} $ if and only if  $H = \frac{ \langle x-x_0 ,  \nn \rangle}{2t_0}$.
\end{Pro}

By Lemma \ref{l:varl0}, $\Sigma$ is critical if and only if
\begin{equation}	 
	   \int_{\Sigma_0}  \left[ f \, \left( H - \frac{ \langle x-x_0 ,  \nn \rangle}{2t_0}
		\right)
		+ h \, \left( \frac{|x-x_0|^2}{4t_0^2} - \frac{n}{2t_0} \right)  +
	  \frac{  \langle x-x_0 , y \rangle }{2t_0}    \right] \, \e^{\frac{-|x-x_0|^2}{4t_0}}  = 0   \notag
\end{equation}
{\bf{for any}} choice of $f$, $y$ and $h$.

It is clear why this holds for $f$, but why $h$ and $y$?
The answer is that $h$ and $y$ correspond to scalings and translations, respectively, and
these can equivalently be achieved by appropriate choices of $f$.

A self-shrinker $\Sigma$ satisfies $H = \frac{ \langle x , \nn \rangle }{2}$, so it is a critical point of the $F= F_{0,1}$ functional.
We will use these functionals to understand dynamical stability of self-shrinkers.
The first step is to compute the Hessian, or second variation, of this functional.
Before doing this, it will be useful to introduce some of the operators that will arise.

\subsection{Weighted inner products and the drift Laplacian}

Let $f$ and $g$ be functions.
It is natural to look at  the weighted $L^2$ inner product
\begin{equation}
	 \int_{\Sigma} fg \, 
	 \e^{ -\frac{|x|^2}{4} } \, . \notag 
\end{equation}
Similarly, we have the weighted inner product for gradients:
\begin{equation}
	 \int_{\Sigma} \langle \nabla^T f , \nabla^T g \rangle \, 
	 \e^{ -\frac{|x|^2}{4} } \,  .  \notag 
\end{equation}
Since
\begin{equation}
 \dv_{\Sigma}  \, \left( \e^{ -\frac{|x|^2}{4} } \, \nabla^T f \right) = \left( 
 \Delta_{\Sigma} f - \frac{1}{2} \, \langle x , \nabla^{T} f \rangle \,  \right) \, \e^{ -\frac{|x|^2}{4} }
 \, ,  \notag 
 \end{equation}
 the divergence theorem applied to  $\left( g \, \e^{ -\frac{|x|^2}{4} } \, \nabla^T f \right)$ gives
\begin{equation}
	 \int_{\Sigma} \langle \nabla^T f , 
	 \nabla^T  g \rangle \, 
	 \e^{ -\frac{|x|^2}{4} }  =  - \int_{\Sigma} g \, 
	 \left( 
 \Delta_{\Sigma} f - \frac{1}{2} \, \langle x , \nabla^{T} f \rangle \,  \right)
	  \, 
	 \e^{ -\frac{|x|^2}{4} }
	 \, . \notag 
\end{equation}
We call this operator the ``drift Laplacian'' $\cL$
\begin{equation}
	\cL \, f = \Delta_{\Sigma} f - \frac{1}{2} \, \langle x , \nabla^{T} f \rangle 
	 \notag  \, .
\end{equation}
It follows that $\cL$ is symmetric in the weighted space.

The $\cL$ operator plays a similar role for self-shrinkers that the Laplacian did for minimal surfaces.
To see this, 
recall that 
given $f: \RR^{n+1} \to \RR$, the Laplacian on $\Sigma$ applied to $f$ is
\begin{equation}
	\Delta_{\Sigma} f  
	= \sum_{i=1} \Hess_f (e_i , e_i) - \langle \nabla f , \nn \rangle \, H \, ,  \notag
\end{equation}
where $e_i$ is a frame for $\Sigma$ and $\Hess_f$ is the $\RR^{n+1}$ Hessian of $f$.
Therefore, when  $\Sigma$ is a self-shrinker, it follows that
\begin{align}
	\cL x_i = - \frac{1}{2} \, x_i \, , \label{e:cLcoord} \\
	\cL |x|^2 = 2n - |x|^2 \, . \label{e:cLx2}
\end{align}
The second formula is closely related to the fact that self-shrinkers are critical points for variations in all three parameters of the $F_{0,1}$ functional.  Namely, we
 saw earlier that on any self-shrinker $\Sigma$ we must have
\begin{equation}	 
	   \int_{\Sigma}   \left( \frac{|x|^2}{4} - \frac{n}{2} \right)   \, \e^{\frac{-|x|^2}{4}}  = 0   \notag
\end{equation}
This can also be seen by using  $\cL$ operator.  To do this, 
use the symmetry of $\cL$ to get for any $u$ and $v$ (that do not grow too quickly)
\begin{equation}
	  \int_{\Sigma} \left( u \cL v \right) \, 
	 \e^{ -\frac{|x|^2}{4} }   = -
	 \int_{\Sigma} \langle \nabla^T u , 
	 \nabla^T  v \rangle \, 
	 \e^{ -\frac{|x|^2}{4} }  
	 \, . \notag 
\end{equation}
Applying this with $u=1$ and $v = |x|^2$ and using \eqr{e:cLx2} gives
\begin{equation}
	  \int_{\Sigma} \left( 2n -  |x|^2 \right) \, 
	 \e^{ -\frac{|x|^2}{4} }    = 0 
	 \, . \notag 
\end{equation}

We will need another second order operator $L$ which differs from $\cL$ by a zero-th order term:
\begin{equation}
	L = \cL + |A|^2 + \frac{1}{2} \, , 
	\notag
\end{equation}
where the drift Laplacian $\cL$ is given by 
\begin{equation}
	\cL \, f = \Delta_{\Sigma} f - \frac{1}{2} \, \langle x , \nabla^{T} f \rangle 
	  \notag  \, .
\end{equation}
Clearly, $L$ is also symmetric with respect to the weighted inner product.
The operator $L$ plays the role of the second variation operator $\Delta + |A|^2 + \Ric (\nn , \nn)$ for minimal surfaces.

\subsection{Second variation}

Let $\Sigma_0 \subset \RR^{n+1}$ be a self-shrinker with unit normal $\nn$, a function $f$,  a vector $y \in \RR^{n+1}$ and a constant
$h \in \RR$.  We will assume that $\Sigma_0$ is complete, $\partial \Sigma_0 = \emptyset$,  and
 $\Sigma_0$ has polynomial volume growth (so that all of our Gaussian integrals converge).
As before, define variations  
\begin{align}
	\Sigma_s &= \{ x + s \, f(x) \, \nn (x) \, | \, x \in \Sigma_0 \} \, , 	\notag \\
	x_s &=  s \, y \, , \notag \\
	t_s &= 1 + s \, h \, . \notag
\end{align}

 \begin{Thm}	\label{t:secvar}
 (Colding-Minicozzi, \cite{CM1})
If we set $F'' = \partial_{ss} \big|_{s=0} \, \left( F_{x_s , t_s} (\Sigma_s) \right)$, then
 \begin{equation}	\label{e:secvar2}
	F''=    (4\pi)^{-n/2} \, \int_{\Sigma}
	 \left( - f \,     Lf
		  + 2f\,h\,H   -   h^2 \, H^2
		+ f \, \langle y , \nn \rangle  - \frac{\langle y , \nn \rangle^2}{2}
		\right)
		 \, \e^{\frac{-|x|^2}{4}}  \, d\mu  \,  .
\end{equation}
\end{Thm}

 {\bf{First observation}}: All (compact) critical points are unstable in the usual sense.
Namely, if we set $h=0$, $y=0$ and $f\equiv 1$ so that
\begin{equation}
	L \,  1 = \cL \,  1 + |A|^2 + \frac{1}{2} = |A|^2 + \frac{1}{2} \, ,
	\notag
\end{equation}
then we see that
\begin{equation}
	\frac{\partial^2 }{\partial s^2} \, \big|_{s=0} \,  \, \left( F_{ 0 , 1} (\Sigma_s) \right)	
	= (4\pi )^{-\frac{n}{2}} \, \int_{\Sigma_0} \left( - |A|^2 - \frac{1}{2} \right) < 0 \, . 
	\notag
\end{equation}
This instability explains why there are very few examples of embedded self-shrinkers that have been proven to exist.  Namely, they are difficult to construct variationally since 
they tend to be highly unstable critical points of $F_{0,1}$.

In fact, we get the same instability for non-compact self-shrinkers, at least when the volume growth is under control:

\begin{Thm}	\label{t:spectral0}
(Colding-Minicozzi, \cite{CM1})
If $\Sigma \subset \RR^{n+1}$ is a smooth complete self-shrinker without boundary and with polynomial volume growth, then there exists a function $u$ with compact support so that
\begin{equation}
 	-\int (u \, L \, u) \, \e^{ - \frac{|x|^2}{4} } < 0 \, .
\end{equation}
\end{Thm}

\subsection{The L operator applied to H and translations}

The link between the mean curvature $H$ of a self-shrinker $\Sigma$ and the second variation is that 
 $H$ is an eigenfunction for $L$ with eigenvalue $-1$; moreover, if $y$ is a constant vector, 
 then
  $\langle y , \nn \rangle$ is also an eigenfunction for $L$.

 \begin{Thm}	\label{t:spectral}
 (Colding-Minicozzi, \cite{CM1})
The mean curvature $H$ and the normal part $\langle v , \nn \rangle$ of a constant vector field $v $  are eigenfunctions of $L$ with
\begin{equation}	\label{e:spec1}
	LH =  H  {\text{ and }} L \langle v , \nn \rangle =  \frac{1}{2} \,  \langle v , \nn \rangle \, .
\end{equation}
\end{Thm}

This will be important later, but it is worth noting that this explains an odd fact in the second variation formula.  Namely, since 
$L$ is a symmetric operator (in the weighted $L^2$ space) and $H$ and $\langle v , \nn \rangle$ are eigenfunctions with different eigenvalues, it follows that 
   $H$ and $\langle v , \nn \rangle$ are orthogonal.
This explains why there was no $H \, \langle y , \nn \rangle$ term in 
the second variation, Theorem \ref{t:secvar}.

\begin{Rem}
Interestingly, there is an analogous situation for Ricci flow.  In this case, Cao, Hamilton and Ilmanen, 
\cite{CaHI},
 computed the second variation formula for Perelman's shrinker entropy and discovered an analog of the $L$ operator.
Moreover, Cao and Zhu showed in \cite{CaZ} that the Ricci tensor is an eigenvector for this operator.
\end{Rem}

\section{Smooth compactness theorem for self-shrinkers}

In \cite{CM3}, we proved the following smooth compactness theorem for self-shrinkers in $\RR^3$:

\begin{Thm}	\label{c:cpt}
(Colding-Minicozzi, \cite{CM3})
 Given an integer $g \geq 0$ and a  constant $V>0$,
the space  of smooth complete embedded self-shrinkers $\Sigma \subset \RR^3$ with
\begin{itemize}
\item genus at most $g$,
\item $\partial \Sigma = \emptyset$,
\item $\Area \, \left( B_{R}(x_0) \cap \Sigma \right) \leq V \, R^2$  for all $x_0 \in \RR^3$ and all $R > 0$
 \end{itemize}
 is compact.

 Namely, any sequence of these has a subsequence that converges in the
 topology of  $C^m$ convergence on compact subsets for any $m \geq 2$.
\end{Thm}

The surfaces in this theorem are assumed to be homeomorphic to closed surfaces with finitely many disjoint disks removed.
The genus of the surface is defined to be the genus of the corresponding closed surface.  For example,  an annulus is a sphere with two disks removed and, thus, has genus zero.

The main motivation for this result is that self-shrinkers model singularities in mean curvature flow.  Thus, the above theorem can be thought of as a compactness result for the space of all singularities.

This should be compared with the Choi-Schoen compactness theorem for minimal surfaces in a manifold with positive Ricci curvature, \cite{CiSc}.  
However,  the conformal metric is not complete and even the scalar curvature changes sign.

\subsection{The proof of smooth compactness}

There are five main points in the proof of Theorem \ref{c:cpt}:
\begin{enumerate}
\item The bound on the genus plus local area bounds imply local bounds on $\int |A|^2$ (this follows from the local Gauss-Bonnet estimate  in theorem $3$ of \cite{I1}).
\item Using (1) and the Choi-Schoen curvature estimate, \cite{CiSc},  we get a subsequence that converges smoothly, possibly with multiplicity,  away from isolated ``singular points'' where the curvature 
concentrates.
\item By Allard's theorem, \cite{Al}, the existence of singular points implies that the convergence is with multiplicity greater than one.
\item If the limit has multiplicity greater than one, then the limit is stable as a minimal surface in the conformally changed metric (by rescaling to get a Jacobi field as in \cite{CM13}).
\item Combining (4) and Theorem \ref{t:spectral0} shows that there cannot be any singular points.
\end{enumerate}

We will say a bit more about step (4) and why 
multiplicity   implies stability.  The basic point is that
as $2$ sheets come together, they are both graphs over the limit and 
the difference $w_i$ between these $2$ graphs {\underline{does not vanish}} (by embeddedness).  Thus, 
$w_i$ does not change sign and (almost) satisfies the linearized equation  $L w_i = 0$.   The
$w_i$'s go to $0$, but the Harnack inequality gives convergence for (a subsequence of)
\begin{equation}
	u_i = \frac{w_i}{w_i(p)} \, .  \notag
\end{equation}
It is not hard to show that the limiting function $u$ is a positive solution of $Lu=0$ with $u(p) =1$.  Of course, $u$ is initially defined only away from the isolated singular points, but it is possible to show that it extends across these potential singularities.  Finally, 
as we saw for minimal surfaces, this implies positivity of the operator $L$.

\section{The entropy}

The $F_{x_0,t_0}$ functional was defined 
for $t_0>0$ and $x_0\in \RR^{n+1}$ by
\begin{align}
F_{x_0 , t_0} (\Sigma) = (4\pi t_0)^{-n/2} \, \int_{\Sigma} \, \e^{-\frac{|x- x_0|^2}{4t_0}} \, d\mu\, .   \notag
\end{align}
If $M_t$ flows by mean curvature and $t>s$, then Huisken's monotonicity formula gives 
\begin{equation}	\label{e:huiskenF}
		F_{x_0 , t_0} ( M_t)  \leq  F_{x_0 , t_0 + (t-s)} (M_s)  \, .
\end{equation}
Thus, we see that a fixed $F_{x_0,t_0}$ functional is not monotone under the flow, but the supremum over all of these functionals is monotone.    We call this invariant
 the entropy and denote it by
\begin{align}	\label{e:entropy}
\lambda (\Sigma) = \sup_{x_0 , t_0} \, F_{x_0 , t_0} (\Sigma)  \, .  
\end{align}
The entropy has four key properties:
\begin{enumerate}
\item $\lambda$ is invariant under dilations, rotations, and translations.
\item $\lambda (M_t)$ is non-increasing under MCF.
\item If $\Sigma$ is a self-shrinker, then $\lambda (\Sigma) = F_{0,1} (\Sigma) = \Theta_{0,0}$.
\item Entropy is preserved under products with a line, i.e., $\lambda (\Sigma \times \RR) = 
\lambda (\Sigma)$.
\end{enumerate}

\subsection{A few entropies}
Stone, \cite{St}, computed the densities $\Theta_{0,0}$, and thus also $\lambda$, for self-shrinking spheres, planes and cylinders:
\begin{itemize}
\item $\lambda (\RR^2) = 1$.
\item $\lambda (\SS^2_2) = \frac{4}{\e} \approx 1.4715$.
\item $\lambda (\SS^1_{\sqrt{2}}) = \sqrt{ \frac{2\pi}{\e} } \approx 1.5203$.
\end{itemize}
Moreover, 
he also showed that $\lambda (\SS^n)$ is decreasing in $n$.

\subsection{How entropy will be used}
The main point about $\lambda$ is that it can 
 be used to rule out certain singularities because
 of the monotonicity of entropy under MCF and its invariance under dilations:
 
 \begin{Cor}
If $\Sigma$ is a self-shrinker given by a tangent flow for $M_t$ with $t>0$, then
\begin{equation}
  F_{0,1} (\Sigma) = \lambda (\Sigma) \leq \lambda (M_0) \, . \notag
 \end{equation}
 \end{Cor}

\subsection{Classification of entropy stable singularities}

To illustrate our results, we will first specialize to the case where $n=2$, that is to mean curvature flow of surfaces in $\RR^3$.

\begin{Thm}	\label{c:nonlin1a}
(Colding-Minicozzi, \cite{CM1})
Suppose that $\Sigma \subset \RR^3$ is a smooth complete embedded self-shrinker without boundary and with polynomial volume growth.
\begin{itemize}
\item If $\Sigma$ is not a sphere, a plane, or a cylinder,  then there is a graph $\tilde{\Sigma}$ over $\Sigma$ of a compactly supported function with arbitrarily small $C^m$ norm (for any fixed $m$) so that
 $\lambda( \tilde{\Sigma})<\lambda(\Sigma)$.
 \end{itemize}
In  particular,   $\Sigma$  cannot arise as a tangent flow to the MCF starting from $\tilde{\Sigma}$.
\end{Thm}

Thus, spheres, planes and cylinders are the only generic self-shrinkers.
This should be contrasted  with Huisken's result for convex flows, where we see that any small perturbation remains convex and, thus, still becomes extinct at a round point.

Essentially the same result holds in all dimensions, with one small difference: the perturbation does not have compact support if the shrinker is a product of a lines with an unstable shrinker in one dimension less; see \cite{CM1}.

\subsection{F-stability}

We saw that every self-shrinker is unstable as a critical point of $F_{0,1}$ and, in fact, 
it is already unstable just from the variations corresponding to  translations and dilations.  Roughly speaking, we will say that a self-shrinker is
$F$-stable if these are the only sources of instability (this is essentially orbital stability for a corresponding dynamical system).  Namely,
a self-shrinker $\Sigma$ is $F$-stable if for every variation $\Sigma_s$ there exist variations $x_s$ and $t_s$ so that
\begin{equation}
	\frac{d^2}{ds^2} \, \big|_{s=0} \, F_{x_s, t_s} (\Sigma_s) \geq 0 \, . \notag
\end{equation}

It is not hard to see that $\SS^n$ and $\RR^n$ are $F$-stable:

\begin{Lem}	\label{l:sn}
(Colding-Minicozzi, \cite{CM1})
The $n$-sphere of radius $\sqrt{2n}$ in $\RR^{n+1}$ is   $F$-stable.
\end{Lem}

\begin{proof}
Note that
  $x^T = 0$,  $A$ is $1/\sqrt{2n}$ times the metric, and  $L = \Delta + 1$.  Therefore, by Theorem \ref{t:secvar},
  the lemma will follow from showing that
given an arbitrary normal variation $f \nn$, there exist $h \in \RR$ and $y \in \RR^{n+1}$ so that
  \begin{equation}	\label{e:secvarsn2}
	  \int_{\SS^n}
	 \left[ - f \, \left(   \Delta f + f
		\right) + \sqrt{2n} \, f\,h    -   \frac{n}{2} \, h^2
		+ f \, \langle y , \nn \rangle  - \frac{\langle y , \nn \rangle^2}{2}
		\right]    \geq 0  \, .
\end{equation}
Recall that the eigenvalues of the Laplacian{\footnote{See, e.g., ($14$) on page $35$ of Chavel, \cite{Ca}.}} on the $n$-sphere of radius one are given by $k^2 + (n-1) \, k$ for $k= 0 , 1, \dots$ with $0$ corresponding to the constant function and the first non-zero eigenvalue $n$ corresponding to the restrictions of the linear functions in $\RR^{n+1}$.  It follows that the eigenvalues of $\Delta$ on the sphere of radius $\sqrt{2n}$ are given by
\begin{equation}
	\mu_k = \frac{k^2 + (n-1) \, k}{2n} \, ,
\end{equation}
with $\mu_0 = 0$ corresponding to the constant functions and $\mu_1 = \frac{1}{2}$ corresponding to the linear functions.  Let $E$ be the space of $W^{1,2}$ functions that are orthogonal to constants and linear functions; equivalently, $E$ is the span of all the eigenfunctions for $\mu_k$ for all $k \geq 2$.
Therefore, we can choose $a \in \RR$ and $z \in \RR^{n+1}$ so that
\begin{equation}
	f_0 \equiv f - a - \langle z , \nn \rangle  \in E \, .
\end{equation}
Using the orthogonality of the different eigenspaces, we get that
\begin{align}	\label{e:p1}
	\int_{\SS^n}  -f \, (\Delta f +  f) &\geq (\mu_2- 1) \, \int_{\SS^n} f_0^2 + (\mu_1 - 1) \, \int_{\SS^n} \langle z , \nn \rangle^2 + (\mu_0 - 1) \, \int_{\SS^n} a^2 \notag \\
	&= \frac{1}{n} \, \int_{\SS^n} f_0^2  - \frac{1}{2} \, \int_{\SS^n} \langle z , \nn \rangle^2   -  \int_{\SS^n} a^2 \, .
\end{align}
Again using the orthogonality of different eigenspaces, we get
  \begin{equation}	 	\label{e:p2}
	  \int_{\SS^n}
	 \left[ \sqrt{2n} \, f\,h
		+ f \, \langle y , \nn \rangle
		\right]    =    \int_{\SS^n}
	 \left[ \sqrt{2n} \, a \,h
		+  \langle z , \nn \rangle \, \langle y , \nn \rangle
		\right]
		  \, .
\end{equation}
Combining \eqr{e:p1} and \eqr{e:p2}, we get that the left hand side of \eqr{e:secvarsn2} is greater than or equal to
  \begin{align}	 	\label{e:p3}
	 & \int_{\SS^n}  \left[ \frac{f_0^2 }{n}  - \frac{1}{2}   \langle z , \nn \rangle^2   -   a^2+
	  \sqrt{2n} \, a \,h    -   \frac{n}{2} \, h^2
		+  \langle z , \nn \rangle \, \langle y , \nn \rangle   - \frac{\langle y , \nn \rangle^2}{2}
		\right]   \notag \\
		& \quad =
		 \int_{\SS^n}  \left[ \frac{f_0^2 }{n}  -
		 \frac{1}{2} \left(  \langle z , \nn \rangle  - \langle y , \nn \rangle \right)^2   -
		 \left( a - \frac{\sqrt{n} \, h }{\sqrt{2}} \right)^2 		\right]
		  \, .
\end{align}
This can be made non-negative by choosing $y = z$ and $h = \frac{\sqrt{2} \, a }{\sqrt{n}} $.

\end{proof}

\subsection{The splitting theorem}

The importance of $F$-stability comes from the following ``splitting theorem'' from \cite{CM1}:

 If $\Sigma_0$ is a self-shrinker that does not split off a line and
$\Sigma_0$ is $F$-unstable, then there is a compactly supported variation $\Sigma_s$ with
\begin{equation}
	\lambda (\Sigma_s) < \lambda (\Sigma_0) \, \,   \forall \, \, s \ne 0 \, . \notag 
\end{equation}
The precise statement of the splitting theorem is :

 \begin{Thm}	\label{t:nonlin1c}
 (Colding-Minicozzi, \cite{CM1})
Suppose that $\Sigma \subset \RR^{n+1}$ is a  smooth complete embedded self-shrinker with $\partial \Sigma = \emptyset$,  with polynomial volume growth, and $\Sigma$ does not split off a line isometrically. 

 If $\Sigma$ is $F$-unstable, then there is a compactly supported variation $\Sigma_s$ with $\Sigma_0 = \Sigma$ so that $\lambda (\Sigma_s) < \lambda (\Sigma)$ for all $s \ne 0$.
\end{Thm}

  \noindent
{\bf{The idea of the splitting theorem is roughly:}}
\begin{itemize}
\item Use that $\Sigma$ does not split to show that $F_{0,1}$ is a strict maximum for $F_{x_0 , t_0}$.
\item Deform $\Sigma_0$ in the $F$-unstable direction $\Sigma_s$.
\item Consider the function $G(s,x_0 , t_0)$ given by
\begin{equation}
	G(s,x_0,t_0) = F_{x_0,t_0} (\Sigma_s) \, . \notag 
\end{equation}
and show that this has a strict maximum at $x_0 = 0$, $t_0 = 1$ and $s=0$.
\end{itemize}

The precise statement of the first step is:

\begin{Lem}	\label{l:strict}
(Colding-Minicozzi, \cite{CM1})
Suppose that $\Sigma$ is a smooth complete embedded self-shrinker with $\partial \Sigma = \emptyset$,   polynomial volume growth, and $\Sigma$  does not split off a line  isometrically.  Given $\epsilon > 0$, there exists $\delta > 0$ so
\begin{equation}
	\sup \, \{ F_{x_0 , t_0}(\Sigma) \, | \, |x_0| + |\log t_0| > \epsilon \} \, < \, \lambda - \delta \, .
\end{equation}
\end{Lem}

\begin{proof}
(sketch of Theorem \ref{t:nonlin1c}).
   Assume  that $\Sigma$ is not $F$-stable and, thus, there is a
  one-parameter normal variation $\Sigma_s$  for $s \in
[-2\, \epsilon ,   2 \, \epsilon ]$ with $\Sigma_0 = \Sigma$   so that:
\begin{enumerate}
\item[(V1)] For each $s$, the variation vector field is given by a function $f_{\Sigma_s}$ times the normal $\nn_{\Sigma_s}$ where every $f_{\Sigma_s}$ is supported in a fixed compact subset of $\RR^{n+1}$.
\item[(V2)] For any variations $x_s$ and $t_s$ with $x_0 = 0$ and $t_0 = 1$, we get that
\begin{equation}
	\partial_{ss} \big|_{s=0} \, F_{x_s , t_s} (\Sigma_s) < 0 \, .
\end{equation}
\end{enumerate}
We will use this to prove that $\Sigma$ is also entropy-unstable.

\vskip2mm
\noindent
{\bf{Setting up the proof}}:
Define a function $G: \RR^{n+1} \times \RR^{+} \times [-2\, \epsilon ,   2 \, \epsilon ] \to \RR^{+} $ by
\begin{equation}
	G(x_0 , t_0 , s) = F_{x_0 , t_0} \, \left( \Sigma_s \right) \, .
\end{equation}
We will show that there exists some $\epsilon_1 > 0$ so that if $s \ne 0$ and $|s| \leq \epsilon_1$, then
\begin{equation}	\label{e:goal}
	\lambda (\Sigma_s) \equiv \sup_{x_0 , t_0} \, G(x_0 , t_0 , s) < G(0,1,0) = \lambda (\Sigma) \, ,
\end{equation}
and this will give the theorem with $\tilde{\Sigma}$ equal to $\Sigma_s$ for any   $s  \ne 0$ in $ (-\epsilon_1 ,\epsilon_1)$; by taking $s >0$ small enough, we can arrange that $\tilde{\Sigma}$ is as close as we like to $\Sigma_0 = \Sigma$.

The remainder of the proof is devoted to establishing \eqr{e:goal}.  The key points will be:
\begin{enumerate}
\item   $G$ has a strict local maximum at $(0,1,0)$.
\item  The restriction of $G$ to $\Sigma_0$, i.e., $G(x_0 , t_0, 0)$, has a strict global maximum at $(0,1)$.
\item	 $\left| \partial_s G \right|$ is uniformly bounded on compact sets.
\item  $G(x_0,t_0 , s)$ is strictly less than $G(0,1,0)$ whenever  $|x_0|$ is sufficiently large.
\item $G(x_0,t_0 , s)$ is strictly less than $G(0,1,0)$ whenever $\left| \log t_0 \right|$ is sufficiently large.
\end{enumerate}

\vskip2mm
\noindent
{\bf{The proof of \eqr{e:goal} assuming (1)--(5)}}:
We will divide into three separate regions depending on the size of $|x_0|^2 + ( \log t_0 )^2$.

First,
it follows from steps (4) and (5) that there is some $R > 0$ so that \eqr{e:goal} holds for every $s$ whenever
\begin{equation}
	x_0^2 + ( \log t_0 )^2 > R^2 \, .
\end{equation}

Second, as long as $s$ is small, step (1) implies that \eqr{e:goal} holds when $x_0^2 + ( \log t_0 )^2$ is sufficiently small.

Finally,  in the intermediate region where $x_0^2 + ( \log t_0 )^2$  is bounded from above and bounded uniformly away from zero, step (2) says that $G$ is strictly less than $\lambda (\Sigma)$ at $s=0$ and step (3) says that the $s$ derivative of $G$ is uniformly bounded.  Hence,  there exists some $\epsilon_3 > 0$ so that
$G(x_0 , t_0 , s)$ is strictly less than $\lambda (\Sigma)$ whenever $(x_0 , t_0)$ is in the intermediate region as long as  $|s| \leq \epsilon_3$.

This completes the proof of \eqr{e:goal} assuming (1)--(5).  See \cite{CM1} for the proofs of (1)--(5).
\end{proof}

\subsection{Classification of F-stable self-shrinkers}

 \begin{Thm}	\label{t:liketo}
 (Colding-Minicozzi, \cite{CM1})
 If $\Sigma$ is a smooth\footnote{The theorem holds when $n\leq 6$ and $\Sigma$ is an oriented integral varifold that is smooth off of a singular set with locally finite $(n-2)$-dimensional Hausdorff measure.}
 complete embedded self-shrinker in  $\RR^{n+1}$ without boundary and with polynomial volume growth that is $F$-stable with respect to compactly supported variations, then it is either the round sphere or a hyperplane.
 \end{Thm}
 
 Combined with the splitting theorem, this gives the classification of generic self-shrinkers.

 The main steps in the proof of Theorem \ref{t:liketo} are:
 \begin{itemize}
 \item Show that $F$-stability implies mean convexity (i.e., $H \geq 0$).
 \item Classify the mean convex self-shrinkers (see Theorem \ref{t:huisken}
below).
\end{itemize}

The classification of mean convex self-shrinkers began with
  \cite{H3},  where Huisken showed that the only smooth {\underline{closed}}  self-shrinkers with non-negative mean curvature in $\RR^{n+1}$ (for $n > 1$) are round spheres (i.e., $\SS^n$).  When $n=1$,   Abresch and Langer, \cite{AbLa},   had already shown that the circle is the only simple closed self-shrinking curve.
 In a second paper, \cite{H4}, Huisken dealt with the non-compact case.  He showed in \cite{H4} that  the only smooth {\underline{open}} embedded self-shrinkers  in $\RR^{n+1}$
with $H \geq 0$, polynomial volume growth, and $|A|$ bounded are
 isometric products of a round sphere and a linear subspace (i.e. $\SS^k\times \RR^{n-k}\subset \RR^{n+1}$).    We will show  that Huisken's classification holds even without the $|A|$ bound which will be crucial for our applications:

 \begin{Thm}	\label{t:huisken}
 (\cite{H3}, \cite{H4} and \cite{CM1})
 $\SS^k\times \RR^{n-k}$ are the only smooth complete embedded self-shrinkers without boundary, with polynomial volume growth, and   $H \geq 0$ in $\RR^{n+1}$.
  \end{Thm}

  The $\SS^k$ factor in Theorem \ref{t:huisken} is round and has radius $\sqrt{2k}$; we allow the possibilities of a hyperplane (i.e., $k=0$) or a sphere ($n-k = 0$).

\subsection{Proof in the compact case}

Since $L$ is symmetric in the weighted space, its spectral theory is similar to the Laplacian:
\begin{enumerate}
\item There are eigenvalues $\mu_1 < \mu_2 \leq \dots$ with $\mu_i \to \infty$ and eigenfunctions $u_i$ with
\begin{equation}
	L \, u_i = - \mu_i \, u_i \, . \notag 
\end{equation}
\item The lowest eigenfunction $u_1$ does not change sign.
\item If $\mu_i \ne \mu_j$, then $u_i$ and $u_j$ are orthogonal, i.e.,
\begin{equation}
	\int_{\Sigma} \, u_i \, u_j \, \e^{ - \frac{|x|^2}{4} } = 0 \, . \notag
\end{equation}
\end{enumerate}

Let $\Sigma$ be a closed self-shrinker and suppose that $H$ changes sign.
We will show that $\Sigma$ is $F$-unstable.
We know that $L \, H =  H$.  Since $H$ changes sign,  it is NOT the lowest eigenfunction.
It follows that $\mu_1 < -1$.  Let $u_1$ be the corresponding (lowest) eigenfunction and define the variation
\begin{equation}
	\Sigma_s = \{ x + s \, u_1 (x) \, \nn (x) \, | \, x \in \Sigma \} \, . \notag
	\end{equation}
Given variations $x_s = s y$ and $t_s = 1 + s h$, the second variation is $(4\pi)^{- \frac{n}{2} }$ times
\begin{equation}
	\int \left[ \mu_1 \, u_1^2 + 2 u_1 h H - h^2 H^2 + u_1 \langle y , \nn \rangle -
	\frac{ \langle y , \nn \rangle^2}{2}  \right] \, \e^{ - \frac{|x|^2}{4} } \, . \notag 
\end{equation}
But $u_1$ is orthogonal to the other eigenfunctions $H$ and $\langle y , \nn \rangle$, so
\begin{equation}
	= \int \left[ \mu_1 \, u_1^2 - h^2 H^2  -
	\frac{ \langle y , \nn \rangle^2}{2}  \right] \, \e^{ - \frac{|x|^2}{4} } \, . \notag 
\end{equation}
This is obviously negative no matter what $h$ and $y$ are.

\section{An application}

Fix $D$, $V$ and $g$ and let $\cM_{D,V,g}$ be all self-shrinkers in $\RR^3$ with:\\
\begin{itemize}
\item Diameter at most $D$.\\
\item Entropy at most $V$.\\
\item Genus at most $g$.
\end{itemize}

Then: C-M compactness theorem implies that $\cM_{D,V,g}$ is smoothly compact.

Combined with our entropy stability:  There exists $\epsilon > 0$ so that if $\Sigma \in \cM_{D,V,g}$ is not $\SS^2$, then there is a graph $\tilde{\Sigma}$ over $\Sigma$ with

\begin{equation}
	\lambda (\tilde \Sigma) \leq \lambda (\Sigma) - \epsilon \, . \notag
\end{equation}

\subsection{Piece-wise MCF}

We next define an ad hoc notion
of generic MCF that requires the least amount of technical set-up, yet should suffice for many applications.

A piece-wise MCF    is a finite collection of MCF's $M^i_t$ on time intervals $[t_i , t_{i+1}]$ so that each $M^{i+1}_{t_{i+1}}$ is the graph over $M^i_{t_{i+1}}$ of a function $u_{i+1}$,

   \begin{align}
   	\Area \, \left( M^{i+1}_{t_{i+1}} \right) &= \Area \, \left( M^i_{t_{i+1}} \right) \, , \notag  \\ 
    \lambda \, \left( M^{i+1}_{t_{i+1}} \right) &\leq  \lambda \, \left( M^i_{t_{i+1}} \right) \, . \notag
   \end{align}

With this definition, area is non-increasing in $t$ even across the jumps.

\subsection{Generic compact singularities}

In this subsection, we will 
assume  that the multiplicity one conjecture of Ilmanen holds.  Under this assumption, we have the following generalization of the Grayson-Huisken theorems:

\begin{Thm} (Colding-Minicozzi, \cite{CM1})
For any closed embedded surface $\Sigma \subset \RR^3$, there exists a piece-wise MCF $M_t$ starting at $\Sigma$ and defined up to time $t_{0}$ where the surfaces become singular.  
Moreover, $M_t$ can be chosen so that if
\begin{equation}	\label{e:boundeddiam}
\liminf_{t \to t_{0}} \, \, \frac{\text{diam} M_t}{\sqrt {t_{0}-t}}< \infty\, , \notag
\end{equation}
then $M_t$ becomes extinct in a round point.
\end{Thm}

 \section{Non-compact self-shrinkers}

\subsection{The  spectrum of $L$ when $\Sigma$ is non-compact}

If $\Sigma$ is non-compact,  there may not be a lowest eigenvalue for $L= \cL +|A|^2 + \frac{1}{2}$.  
However, we can still define the bottom of the spectrum (which we still call $\mu_1$) by
\begin{equation}	\label{e:lam1}
	\mu_1 = \, \inf_{f} \,  \, \frac{    - \int_{\Sigma} \left(  f \, L f \right) \, \e^{- \frac{|x|^2}{4} } }{   \int_{\Sigma} f^2 \, \e^{- \frac{|x|^2}{4} } } \, ,
	\notag
\end{equation}
where the infimum is taken over
 smooth functions $f$ with compact support.   

\noindent
{\bf{Warning}}:
Since $\Sigma$ is non-compact,   we must allow the possibility that $\mu_1 = - \infty$.

\vskip2mm
 We have the following characterization of $\mu_1$ generalizing the compact case (as usual $\Sigma$ is a complete self-shrinker without boundary and with polynomial volume growth):

 \begin{Pro}	\label{p:cmnon}
 (Colding-Minicozzi, \cite{CM1})
 If $\mu_1 \ne - \infty$, then:
 \begin{itemize}
 \item 
 There is a positive function $u$ on $\Sigma$ with $L \, u = - \mu_1 \, u$.
 \item   
 If $v$ is in the weighted $W^{1,2}$ space and $L \, v = - \mu_1 \, v$, then $v = C \, u$ for  $C \in \RR$.
 \item $|A| \, |x|$ is in the weighted $L^2$ space.
 \end{itemize}
 \end{Pro}

 (This proposition combines lemmas $9.15$
 and $9.25$ in \cite{CM1}.)

 \subsection{$\mu_1$ when $H$ changes sign}

We have already seen that the mean curvature $H$ is an eigenfunction of $L$ with eigenvalue $-1$.
The next theorem shows that if $H$ changes, then the bottom of the spectrum $\mu_1$ is strictly less than $-1$.

 \begin{Thm}	\label{c:pde}
 (Colding-Minicozzi, \cite{CM1})
 If the mean curvature $H$ changes sign, then $\mu_1 < -1$.
 \end{Thm}

The  idea of the proof follows:  
\begin{enumerate}
\item
We can assume that $\mu_1 \ne - \infty$.  Thus,  Proposition \ref{p:cmnon} gives that $|A| \, |x|$ is in the weighted $L^2$ space.
\item Differentiating the self-shrinker equation $H = \frac{1}{2} \, \langle x , \nn \rangle$ gives
\begin{equation}
	2\, \nabla_{e_1} H =  - A_{ij} \langle x , e_j \rangle \, . \notag 
\end{equation}
It follows from this and (1) 
 that $H$, $\nabla H$, and $|A| \, H$ are  in the weighted $L^{2}$ space.
\item  The bounds in (3) are enough to justify  using $H$ as a test function in the definition of $\mu_1$ and get   that  $\mu_1 \leq -1$.

\item It remains to rule out that $\mu_1 = -1$.  But this would imply that $H$ does not change sign by the uniqueness part of Proposition \ref{p:cmnon}.
\end{enumerate}

\subsection{The $F$-unstable variation when $\mu_1 < -1$}

We have shown that if $H$ changes sign, then  $\Sigma$ has $\mu_1 < -1$.  
When $\Sigma$ was closed, it followed immediately from this and  the orthogonality of eigenfunctions with different eigenvalues (for a symmetric operator) that $\Sigma$ was $F$-unstable. 
However, this orthogonality uses an integration by parts  which is  not justified when $\Sigma$ is open.  
  Instead, we show that the lowest eigenfunction on a sufficiently large ball is {\emph{almost}} orthogonal to $H$ and the translations.  
 This turns out to be enough to prove $F$- instability:

   \begin{Lem}  (Colding-Minicozzi, \cite{CM1})	\label{l:ortho}
   If $\mu_1 < -1$, then there exists $\bar{R}$ so that if $R \geq \bar{R}$ and $u$ is a Dirichlet eigenfunction for $\mu_1 (B_R)$, then
   for any $h \in \RR$ and any $y \in \RR^{n+1}$ we have
   \begin{equation}	\label{e:ortho}
   	\left[  -u \, L \, u + 2 u \, h \, H + u \,
	 \langle y , \nn \rangle
		    -   h^2 \, H^2
		  - \frac{\langle y , \nn \rangle^2}{2}  \right]_{B_R} < 0 \, .
   \end{equation}
   \end{Lem}
   
   Here, in \eqr{e:ortho}, we used $\left[ \cdot \right]_{B_R}$ to denote the Gaussian weighted integral over the ball $B_R$.

\vskip2mm
To illustrate how the almost orthogonality comes in, we will explain a simple case of Lemma \ref{l:ortho} when $\mu_1 < - \frac{3}{2}$ (this still leaves the possibility that $\mu_1$ is between $-\frac{3}{2}$ and $-1$).

\vskip1mm
\noindent
{\bf{Sketch of \eqr{e:ortho} when $\mu_1 < - \frac{3}{2}$}}:
 Using the Cauchy-Schwartz inequality $ ab \leq \frac{1}{2} \, (a^2 + b^2)$ on the cross-term
  $ u \,
	 \langle y , \nn \rangle $,  the left hand side of \eqr{e:ortho} is bounded from above by
\begin{equation}	     	\label{e:ortho2}
	 \left[ \left( \frac{1}{2} + \mu_1(B_R) \right) \, u^2 + 2 u \, h H
		    -   h^2 \, H^2
		\right]_{B_R} 	  \, .
\end{equation}
We will show that this is  negative when $R$ is   large.
 Since
   $\mu_1 < -\frac{3}{2} $, we can choose $\bar{R}$ so that $\mu_1 (B_{\bar{R}}) < - \frac{3}{2}$.
 Given any $R \geq \bar{R}$, then \eqr{e:ortho2} is strictly less than
   \begin{equation}	
	 \left[ -  u^2 + 2 u \, h H
		    -   h^2 \, H^2
		\right]_{B_R}  =
	- \left[ 	 \left(    u  - h H \right)^2
		\right]_{B_R}	  \, ,
\end{equation}
which  gives \eqr{e:ortho} in this case.

\subsection{Classification of mean convex self-shrinkers}

Throughout this subsection, $\Sigma$ is a complete, non-compact self-shrinker with polynomial volume growth, $\partial \Sigma = \emptyset$ and $H> 0$.  We will sketch the proof of the classification theorem, i.e., Theorem \ref{t:huisken}, which gives that $\Sigma$ is a cylinder.

The classification relies heavily upon the following ``Simons' identity'' for the second fundamental form $A$ of a self-shrinker $\Sigma$:
\begin{equation}
	L \, A = A \, , \notag 
\end{equation}
where we have extended $L$ to act on tensors in the natural way.
Taking the trace of this recovers that $L \, H = H$ since traces and covariant derivatives commute (the metric is parallel).

Roughly speaking, the identity $L A = A$ says that
 the whole matrix $A$ is a lowest eigenfunction for $L$.   
 Moreover, the matrix strong maximum principle shows that the kernel of $A$ consists of parallel vector fields that split off a factor of $\RR^k$.
The remaining principle curvatures must  then all be multiples of each other (by the uniqueness of the lowest eigenfunctions for $L$).  
We will show that the remaining non-zero principle curvature are in fact the same, i.e., $\Sigma$ is the product of an affine space and a totally umbilic submanifold.

 \vskip2mm
The main steps in the proof Theorem \ref{t:huisken} are:
\begin{enumerate}
\item
  Using $L \, A = A$, it follows that
 \begin{equation}	\label{e:LmodA}
	L \, |A| = |A|  +  \frac{\left| \nabla A \right|^2 - \left| \nabla |A| \right|^2}{ |A|}
	 \geq |A|  \, .   \notag
\end{equation}
\item Using $L \, A = A$ and the ``stability inequality'' coming from $H>0$, we show that
\begin{equation}
	\int \left( |A|^2 + |A|^4 + |\nabla |A||^2 + |\nabla A|^2 \right) \, \e^{ - \frac{|x|^2}{4} } < \infty \, . 
	\notag
\end{equation}
(This should remind you of the Schoen-Simon-Yau, \cite{ScSiY}, curvature estimates for stable minimal hypersurfaces.)
\item Using (2) to show that various integrals converge and justify various integrations by parts,  (1) and $L \, H = H$ imply that $H = C \, |A|$ for $C > 0$.
\item 
The combination of $|\nabla A|^2 = |\nabla |A||^2$ and $H = C \, |A|$ - and some work - give the classification.  This last step is essentially the same argument as in \cite{H3}.
\end{enumerate}

% ----------------------------------------------------------------

\end{document}